\documentclass[11pt]{amsart}

\usepackage{amssymb, amstext, amscd, amsmath, color}

\usepackage{url}
\usepackage{tikz}
\usepackage{epstopdf}
\usepackage{amsfonts}
\usepackage{amsthm}
\usepackage{amsfonts}
\usepackage{array}
\usepackage{epsfig}
\usepackage{eucal}
\usepackage{latexsym}
\usepackage{mathrsfs}
\usepackage{textcomp}
\usepackage{verbatim}
\usepackage{setspace}
\usepackage{enumerate}
\usepackage[colorlinks=true,linkcolor=blue,urlcolor=blue,citecolor=red]{hyperref}

\newtheorem{thm}{Theorem}[section]

\newtheorem{claim}[thm]{Claim}

\newtheorem{cor}[thm]{Corollary}

\newtheorem{example}[thm]{Example}

\newtheorem{lem}[thm]{Lemma}

\newtheorem{prop}[thm]{Proposition}
\newtheorem{rem}[thm]{Remark}

\numberwithin{equation}{section}

\newtheorem{innercustomthm}{Case}
\newenvironment{case}[1]
  {\renewcommand\theinnercustomthm{#1}\innercustomthm}
  {\endinnercustomthm}

\renewcommand{\H}{{\mathcal{H}}}

  \newcommand{\U}{{\mathcal{U}}}

\newcommand{\rank}{\operatorname{rank}}

\newcommand{\Iso}{\operatorname{Isom}}
\newcommand{\reg}{\operatorname{Reg}}
\newcommand{\sreg}{\operatorname{Str}}
\newcommand{\creg}{\operatorname{Com}}
\newcommand{\csreg}{\operatorname{ComStr}}
\newcommand{\grig}{\operatorname{GRig}}
\newcommand{\config}{\mathcal{C}(G,p)}
\newcommand{\confige}{\mathcal{C}(G - vw,p)}
\newcommand{\sm}{\setminus}

\usepackage{blkarray}

\tikzstyle{vertex}=[circle, draw, fill=black, inner sep=0pt, minimum size=3pt]
\tikzstyle{fadedvertex}=[vertex, fill=black!30]
\tikzstyle{edge}=[line width=1pt]
\tikzstyle{fadededge}=[edge,color=black!30]

\begin{document}

\title[Uniquely realisable graphs in analytic normed planes]{Uniquely realisable graphs in analytic normed planes}
\author[Sean Dewar]{Sean Dewar}
\address{Johann Radon Institute\\ Altenberger Strasse 69\\ 4040\\ Linz\\
Austria}
\email{sean.dewar@ricam.oeaw.ac.at}
\author[John Hewetson]{John Hewetson}
\address{Dept.\ Math.\ Stats.\\ Lancaster University\\
Lancaster LA1 4YF \\U.K. }
\email{j.hewetson2@lancaster.ac.uk}
\author[Anthony Nixon]{Anthony Nixon}
\address{Dept.\ Math.\ Stats.\\ Lancaster University\\
Lancaster LA1 4YF \\U.K.}
\email{a.nixon@lancaster.ac.uk}

\begin{abstract}
A bar-joint framework $(G,p)$ in the Euclidean space $\mathbb{E}^d$ is globally rigid if it is the unique realisation, up to rigid congruences, of $G$ in $\mathbb{E}^d$ with the edge lengths of $(G,p)$. Building on key results of Hendrickson \cite{hendrickson} and Connelly \cite{C2005}, Jackson and Jord\'{a}n \cite{J&J} gave a complete combinatorial characterisation of when a generic framework is global rigidity in $\mathbb{E}^2$. We prove an analogous result when the Euclidean norm is replaced by any norm that is analytic on $\mathbb{R}^2 \setminus \{0\}$. More precisely, we show that a graph $G=(V,E)$ is globally rigid in a non-Euclidean analytic normed plane if and only if $G$ is 2-connected and $G-e$ contains 2 edge-disjoint spanning trees for all $e\in E$. 
The main technical tool is a recursive construction of 2-connected and redundantly rigid graphs in analytic normed planes. We also obtain some sufficient conditions for global rigidity as corollaries of our main result and prove that the analogous necessary conditions hold in $d$-dimensional analytic normed spaces.
\end{abstract}

\date{}
\maketitle

\section{Introduction}

A \emph{bar-joint framework} $(G,p)$ in $\mathbb{E}^d$ is an ordered pair consisting of a finite, simple graph $G=(V,E)$ and a map $p:V\rightarrow \mathbb{R}^d$.
Given a framework $(G,p)$ in $\mathbb{E}^d$, a fundamental question is whether the edge lengths of $(G,p)$ determine a unique framework up to rigid congruences of $\mathbb{E}^d$. This is the concept of global rigidity which has been researched intensively over the last 40 years (e.g. \cite{con,GHT,J&J}).

The study of global rigidity is motivated by a number of fundamental questions in various areas of applied mathematics, including protein structure determination \cite{Lav} and sensor network localisation \cite{And}. These applications often involve distance measures other than the standard Euclidean norm, which motivates the topic of this paper; the study of global rigidity for frameworks in normed spaces. A research program in this direction was initiated by the first and third author in \cite{DN}. A (real) normed plane is \emph{analytic} if it is non-Euclidean and the norm restricted to the non-zero points is a real analytic function. While analytic norms are special, we can approximate any norm by a uniformly convergent sequence of analytic norms. Hence it is reasonable to expect that the analytic case is representative of the general situation. In this article we give a complete combinatorial description of global rigidity for the class of analytic normed planes.

One may also consider the related concept of local rigidity, where there must exist only a finite number of realisations (up to isometric transformations) in the given normed space. This has also been studied intensively for many years in the Euclidean case (see, for example, \cite{asi-rot,laman,PG} and the recent survey \cite{NSW}), and research in the non-Euclidean setting was recently initiated by Kitson and Power \cite{kit-pow-1}. Indeed the main result of \cite{kit-pow-1}, which we describe below, will be important in our analysis. Their work influenced subsequent study of a number of related rigidity problems, including for polyhedral and matrix norms \cite{K15,KL}, general normed planes \cite{D19}, higher dimensions \cite{DKN}, and in the presence of symmetry \cite{KKM,KNS}.

Many papers have also considered the ostensibly similar concepts of rigidity and global rigidity under non-Euclidean metrics, such as hyperbolic and Minkowski metrics \cite{GTfields,NW, SalW}. However these contexts are sufficiently similar to the Euclidean setting for the same combinatorial techniques to be applied. Moreover, at the level of infinitesimal rigidity there is an elegant projective invariance \cite{NSW,SalW}. In the case of non-Euclidean norms these conveniences are unavailable, and alternative combinatorial ideas are required to study new classes of graphs and matroids. Furthermore, since distance constraints are not quadratic in the non-Euclidean case, unlike in the Euclidean case, we do not have the luxury of an obviously accessible ``equilibrium stress matrix'' approach (\cite{con,C2005,GHT} inter alia). We overcome these issues by restricting to analytic normed planes, deploying combinatorial tools from \cite{DHN} and \cite{JNglobal}, and making use of the analytic geometry techniques developed in \cite{DN}. In particular we will apply the theory of generalised rigid body motions introduced by the first and third author in \cite{DN}. 

The main result of the paper is as follows. (The definitions of globally rigid and redundantly rigid graphs are formalised in Section \ref{subsec:defs}.)

\begin{thm}\label{thm:grne}
	Let $G=(V,E)$ be a graph with $|V|\geq 2$ and $X$ be an analytic normed plane. 
	Then $G$ is globally rigid in $X$ if and only if $G$ is 2-connected and redundantly rigid in $X$.
\end{thm}

To prove the combinatorial part of the theorem we make use of the strategy developed in \cite{JKN,JNglobal}. Those papers concern the global rigidity of frameworks in $\mathbb{R}^3$ that are restricted to lie on the surface of a cylinder. While we know of no direct equivalence, it turns out that, for both rigidity and global rigidity, the `generic' situation with any analytic norm is equivalent to the generic situation on the surface of a cylinder with the Euclidean norm. That is, the graphs underlying generically rigid frameworks on the surface of the cylinder with the Euclidean norm are the same as those underlying `generically' rigid frameworks under any analytic norm. The equivalence for infinitesimal rigidity can be read off by comparing the main results of \cite{dew2,kit-pow-1} and \cite{NOP}, while the corresponding equivalence for global rigidity follows in the same manner by comparing our main result to the main result of \cite{JNglobal}.

There is one major combinatorial complication that we face in extending the strategy from \cite{JNglobal}. In that context the key difficulty was to establish global rigidity for `circuits' in an appropriate matroid; the general case then followed by a relatively straightforward argument. We may establish global rigidity for the appropriate circuits in analytic normed planes using their combinatorial result (stated as Theorem \ref{thm:construction} below) in combination with the results of \cite{DHN,DN} and Sections \ref{sec:hendrickson} and \ref{sec:proof}. However in Section \ref{sec:circuits} we will explain why their argument for the general case cannot (to our knowledge) be applied to analytic normed planes. Instead we use `ear decompositions' to develop a more detailed combinatorial theory of the class of redundantly rigid and 2-connected graphs in analytic normed planes. Our main technical combinatorial result is a recursive construction, Theorem \ref{thm:1}, which is a generalisation of \cite[Theorem 2.1]{JNglobal} from the special case of circuits to arbitrary redundantly rigid and 2-connected graphs.

We conclude the introduction with an outline of the paper. In Section \ref{sec:background} we introduce the necessary background from rigidity theory, survey the Euclidean case and what is known in other normed planes, and provide key background results concerning analytic normed planes. We then prove precise analogues to the so-called Hendrickson conditions (see \cite{hendrickson} for the Euclidean case) in Section \ref{sec:hendrickson}; these are graph theoretic conditions necessary for a framework to be globally rigid in an arbitrary dimension analytic normed space with finitely many linear isometries. The remainder of the paper is restricted to the 2-dimensional case, and we proceed in a purely combinatorial manner. Sections \ref{sec:circuits}, \ref{sec:mcon} and \ref{sec:comb} form the main technical part of the paper. In these sections we analyse the simple $(2,2)$-sparse matroid $\mathcal{M}(2,2)$, equate the necessary conditions of Section \ref{sec:hendrickson} with a connectivity condition in $\mathcal{M}(2,2)$, deduce key `gluing' properties of the graphs for which $\mathcal{M}(2,2)$ satisfies the connectivity condition, and use three graph operations to give a recursive construction of all such graphs. We then use this construction in Section \ref{sec:proof} to prove that the necessary graph conditions are also sufficient, completing the proof of our characterisation. The geometric extension part of our inductive proof relies on the results of \cite{DHN,DN} and one new geometric argument which we develop. We conclude this section, and the paper, with a number of results giving sufficient combinatorial conditions for a graph to be global rigidity in analytic normed planes. These corollaries to Theorem \ref{thm:grne} include connectivity conditions, spectral conditions and vertex-redundant rigidity conditions which guarantee global rigidity, as well as evaluations of when a transitive graph, an Erd\"{o}s-R\'{e}nyi random graph and a random regular graph are globally rigid.

\section{Normed space rigidity theory} \label{sec:background}

In this section we introduce the concepts and results from rigidity theory in normed spaces that we require throughout the paper.

\subsection{Normed space geometry}

All normed spaces will be finite-dimensional real linear spaces.
There are two main types:
\emph{Euclidean spaces}, where the norm satisfies the parallelogram law, and \emph{non-Euclidean spaces}, where the norm does not.
As Euclidean spaces of the same dimension are isometrically isomorphic,
we shall denote any Euclidean space of dimension $d$ by $\mathbb{E}^d$. 

We require the following terminology.
The dual of a normed space $X$ will be denoted by $X^*$
and will have the norm 
\begin{align*}
    \|f\|^* := \sup \{ f(z) : z \in X,~ \|z\| = 1\}.
\end{align*}
A point $x$ in a normed space $X$ is \emph{supported by $f \in X^*$} (equivalently, $f$ is a \emph{support functional of $x$}) if $f(x) = \|x\|^2$ and $\|f\|^* = \|x\|$.
It follows from the Hahn-Banach theorem that every point has at least one support functional, and every linear functional of $X$ is a support functional of a point in $X$.
A non-zero point $x \in X$ is said to be \emph{smooth} if it has exactly one support functional, which we will denote by $\varphi_x$;
we will also fix $\varphi_0$ to be the zero function of $X$.
We say a normed space $X$ is \emph{smooth} if every non-zero point is smooth,
and we say $X$ is \emph{strictly convex} if every element of $X^*$ is the support functional of exactly one point of $X$.
Given a smooth normed space $X$,
we define
\begin{align*}
    \varphi : X \rightarrow X^*, ~ x \mapsto \varphi_x
\end{align*}
to be the \emph{support map of $X$}.
The map $\varphi$ is continuous and homogeneous (i.e., $\varphi_{cx} = c \varphi_x$ for all $c \in \mathbb{R}$ and $x \in X$), and will be a homeomorphism if and only if $X$ is strictly convex (see for example \cite[Part III]{beauzamy} or \cite[Ch. II]{cio}).

We shall be particularly interested in \emph{analytic normed spaces};
i.e., non-Euclidean normed spaces where the norm is a real analytic function when restricted to the set of non-zero points in the space.
Although the norm of a Euclidean space will also be a real analytic function when restricted to its set of non-zero points,
we opt to consider Euclidean spaces as entirely separate entities.
All analytic normed spaces are smooth and strictly convex \cite[Lemma 3.1]{DN}.

\subsection{Rigidity in normed spaces}

For a finite set $V$ and normed space $X$, a \emph{placement of $V$ in $X$} is any element $p = (p_v)_{v \in V}$ in the vector space $X^V$.
We denote the restriction of $p$ to a subset $U \subset V$ by $p|_U := (p_v)_{v \in U}$.
For a given placement $p$ and any map $g:X \rightarrow X$, 
we define the placement $g \circ p := (g(p_v))_{v \in V}$.
A placement is \emph{spanning} if the affine span of the set $\{p_v : v \in V\}$ is $X$;
equivalently, $p$ is spanning if the only affine map $g:X \rightarrow X$ where $g \circ p = p$ is the identity map, denoted by $I$.
Given that $\Iso (X)$ is the set of isometries of $X$,
we say a placement is \emph{isometrically full} if, given $g \in \Iso (X)$, we have $g \circ p = p$ if and only if $g = I$.
If $|V| \geq \dim X +1$,
then the set of placements of $V$ that are not spanning form a \emph{null} subset of $X^V$ (i.e., a set with Lebesgue measure zero) and hence the set of spanning placements is a \emph{conull} subset of $X^V$ (i.e., a set with a null complement set)\footnote{Any conull subset is dense, and any closed null subset is nowhere dense, but the converse of these statements are not true;
for example,
the set of rationals is a null dense subset of $\mathbb{R}$,
and any fat Cantor set is a closed nowhere dense subset of $\mathbb{R}$ with positive Lebesgue measure.}. 
Since all isometries of a normed space are affine \cite{mazurulam},
the set of isometrically full placements forms a conull subset of $X^V$.

A \emph{framework in $X$} is a pair $(G,p)$ where $G = (V,E)$ is a (finite, simple) graph and $p$ is a placement of $V$ in $X$;
here we also say that $p$ is a \emph{placement of $G$ in $X$}.
We denote the set of placements of $G$ with non-zero edge lengths by
\begin{align*}
X_G := \left\{ p \in X^{V} : p_v \neq p_w \text{ for all } vw \in E \right\}.
\end{align*}
$X_G$ is an open conull subset of $X^{V}$; 
further, if $\dim X >1$ then $X_G$ is path-connected.
We define the \emph{rigidity map of $G$ in $X$} to be the map
\begin{align*}
f_G : X^{V}  \rightarrow \mathbb{R}^{E}, ~ p = (p_v)_{v \in V} \mapsto \left( \frac{1}{2}\|p_v - p_w \|^2 \right)_{vw \in E}.
\end{align*}
Two frameworks $(G,p)$ and $(G,q)$ in a normed space $X$ are \emph{equivalent} 
(denoted $(G,p) \sim (G,p')$) if $f_G(p) = f_G(p')$,
and are \emph{congruent} (or $p \sim p'$) if there exists $g \in \Iso (X)$ such that $p' = g \circ p$.
We say that a framework $(G,p)$ in $X$ is \emph{locally rigid} if there exists $\epsilon >0$ such that every framework $(G,q)$ in $X$ that is equivalent to $(G,p)$ and has $\|p_v-q_v\| < \epsilon$ for each $v \in V$ is necessarily congruent to $(G,p)$;
    otherwise
    $(G,p)$ is \emph{locally flexible}. (See Figures \ref{fig:grexamplesone} and \ref{fig:grexamplesone2} for basic examples.)
    We say that $(G,p)$ is \emph{globally rigid} if every framework in $X$ that is equivalent to $(G,p)$ is also congruent to $(G,p)$.

\begin{figure}[htp]
\begin{center}
     \begin{tikzpicture}[scale=0.8]
    \node[vertex] (a) at (2,0.5) {};
	\node[vertex,label={[white]270:$y$}] (b) at (1,-1) {};
    \node[vertex] (c) at (1,1) {};
	
    \node[vertex] (e) at (-2,-0.5) {};
	\node[vertex] (f) at (-1,-1) {};
    \node[vertex] (g) at (-1,1) {};
	
	\draw[edge] (a)edge(b);
	\draw[edge] (a)edge(c);
	\draw[edge] (b)edge(c);
	
	\draw[edge] (a)edge(e);
	\draw[edge] (b)edge(f);
	\draw[edge] (c)edge(g);
	
	\draw[edge] (e)edge(f);
	\draw[edge] (e)edge(g);
	\draw[edge] (f)edge(g);
\end{tikzpicture}\qquad \qquad
    \begin{tikzpicture}[scale=0.8]
    \node[vertex] (a) at (-2,-1) {};
	\node[vertex] (b) at (-2,1) {};
	
    \node[vertex,label={90:$x$}] (c) at (0,1) {};
	\node[vertex,label={270:$y$}] (d) at (0,-1) {};
	
    \node[vertex] (e) at (2,-1) {};
	\node[vertex] (f) at (2,1) {};

	\draw[edge] (a)edge(b);
	\draw[edge] (a)edge(c);
	\draw[edge] (a)edge(d);
	\draw[edge] (b)edge(c);
	\draw[edge] (b)edge(d);
    
	\draw[edge] (c)edge(d);
	\draw[edge] (c)edge(e);
	\draw[edge] (c)edge(f);
	\draw[edge] (d)edge(e);
	\draw[edge] (d)edge(f);
	\draw[edge] (e)edge(f);
\end{tikzpicture}
\end{center}
\vspace{-0.3cm}
\caption{(Left): A locally rigid but not globally rigid framework in the Euclidean plane.
(Right): A locally rigid framework in $\ell_p^2$ for any $2 < p < \infty$; see \cite[Proposition 6.6]{D21} for a direct proof.
Although the underlying graph does have globally rigid placements for all even values of $p$ (Theorem \ref{thm:1}), this particular framework is not globally rigid since the labelled vertices $x,y$ lie on a line of reflection in any $\ell_p^2$.}
\label{fig:grexamplesone}
\end{figure}
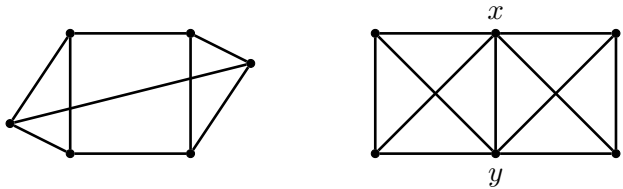

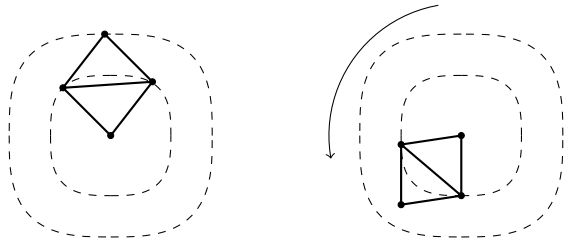
\begin{figure}[htp]
\begin{center}

\begin{tikzpicture}[scale=0.8]
\draw[thick] (0,0) -- (-0.794,0.794); 
\draw[thick] (0,0) -- (0.694,0.894); 

\draw[thick] (-0.794,0.794) -- (-0.1,1.687); 
\draw[thick] (-0.794,0.794) -- (0.694,0.894); 
\draw[thick] (0.694,0.894) -- (-0.1,1.687); 

\draw [dashed,domain=0:90] plot ({(cos(\x))^(2/3)}, {(sin(\x))^(2/3)});
\draw [dashed,domain=90:180] plot ({-(-cos(\x))^(2/3)}, {(sin(\x))^(2/3)});
\draw [dashed,domain=180:270] plot ({-(-cos(\x))^(2/3)}, {-(-sin(\x))^(2/3)});
\draw [dashed,domain=270:360] plot ({(cos(\x))^(2/3)}, {-(-sin(\x))^(2/3)});

\draw [dashed,domain=0:90] plot ({1.687*(cos(\x))^(2/3)}, {1.687*(sin(\x))^(2/3)});
\draw [dashed,domain=90:180] plot ({-1.687*(-cos(\x))^(2/3)}, {1.687*(sin(\x))^(2/3)});
\draw [dashed,domain=180:270] plot ({-1.687*(-cos(\x))^(2/3)}, {-1.687*(-sin(\x))^(2/3)});
\draw [dashed,domain=270:360] plot ({1.687*(cos(\x))^(2/3)}, {-1.687*(-sin(\x))^(2/3)});

\draw[fill] (0,0) circle [radius=0.05];
\draw[fill] (-0.794,0.794) circle [radius=0.05];
\draw[fill] (0.694,0.894) circle [radius=0.05];
\draw[fill] (-0.1,1.687) circle [radius=0.05];

\end{tikzpicture}\qquad\qquad
\begin{tikzpicture}[scale=0.8]
\draw [dashed,domain=0:90] plot ({(cos(\x))^(2/3)}, {(sin(\x))^(2/3)});
\draw [dashed,domain=90:180] plot ({-(-cos(\x))^(2/3)}, {(sin(\x))^(2/3)});
\draw [dashed,domain=180:270] plot ({-(-cos(\x))^(2/3)}, {-(-sin(\x))^(2/3)});
\draw [dashed,domain=270:360] plot ({(cos(\x))^(2/3)}, {-(-sin(\x))^(2/3)});

\draw [dashed,domain=0:90] plot ({1.687*(cos(\x))^(2/3)}, {1.687*(sin(\x))^(2/3)});
\draw [dashed,domain=90:180] plot ({-1.687*(-cos(\x))^(2/3)}, {1.687*(sin(\x))^(2/3)});
\draw [dashed,domain=180:270] plot ({-1.687*(-cos(\x))^(2/3)}, {-1.687*(-sin(\x))^(2/3)});
\draw [dashed,domain=270:360] plot ({1.687*(cos(\x))^(2/3)}, {-1.687*(-sin(\x))^(2/3)});

\draw[fill] (0,0) circle [radius=0.05];

\draw[thick] (0,0) -- (-1,-0.15); 
\draw[thick] (0,0) -- (0,-1); 

\draw[thick] (-1,-0.15) -- (0,-1); 
\draw[thick] (-1,-0.15) -- (-1,-1.15); 
\draw[thick] (0,-1) -- (-1,-1.15); 

\draw[fill] (-1,-0.15) circle [radius=0.05];
\draw[fill] (0,-1) circle [radius=0.05];
\draw[fill] (-1,-1.15) circle [radius=0.05];

\draw[->,domain=100:190] plot ({2.2*cos(\x)}, {2.2*sin(\x)});

\end{tikzpicture}
\end{center}
\vspace{-0.3cm}
\caption{A flexible framework in the analytic plane $\ell_4^2$ (left) and an equivalent but non-congruent framework that can be reached by a continuous motion (right).}
\label{fig:grexamplesone2}
\end{figure}

Determining whether a framework in $\mathbb{E}^d$ is locally rigid is NP-hard when $d \geq 2$ \cite{abbott}.
One would expect that this is also true for (almost all) normed spaces of dimension two and higher.
This motivates us to consider whether a framework can be deformed by infinitesimal motions. For suitable frameworks this will provide a sufficient condition for local rigidity, and it allows us to discuss when local rigidity is a property of a graph. 

Our change in approach requires some additional definitions.
For a graph $G = (V,E)$, a placement $p$ of $G$ in $X$ and the corresponding framework $(G,p)$ are said to be \emph{well-positioned} if $f_G$ is differentiable at $p$ and $p \in X_G$; equivalently, $(G,p)$ is well-positioned if $p_v-p_w$ is smooth for all $vw \in E$.
It is immediate that, given $G$ has at least one edge, the space $X$ is smooth if and only if the set of well-positioned placements of $G$ is exactly the set $X_G$.
If $(G,p)$ is well-positioned, we define the \emph{rigidity operator of $(G,p)$} to be the derivative 
\begin{align*}
df_G(p) : X^{V}  \rightarrow \mathbb{R}^{E}, ~ (x_v)_{v \in V} \mapsto (\varphi_{p_v-p_w}(x_v - x_w ))_{vw \in E}
\end{align*}
of $f_G$ at $p$.
Any element $u \in \ker df_G(p)$ is called an \emph{infinitesimal flex of $(G,p)$}.
An infinitesimal flex $u \in X^V$ is \emph{trivial} if there exists a linear map $T:X \rightarrow X$ and a point $z \in X$ where $u_v = T(p_v) +z$ for all $v \in V$,
and for every point $x \in X$ with support functional $f \in X^*$,
we have $f \circ T(x) = 0$.
The map $T$ will always be a tangent vector of the linear isometry group of $X$ at the identity map $I$\footnote{If we consider the map $T$ as a linear vector field, then for any integral curves $\alpha,\beta: \mathbb{R} \rightarrow X$ (i.e., differentiable maps where $\alpha'(t) = T(\alpha(t))$) and any $t \in \mathbb{R}$, the distance $\|\alpha(t) - \beta(t)\|$ must be constant as $t$ varies.
This is because the support function equation satisfied by $T$ at every point in $X$ implies that every generalised derivative of the map $t \mapsto \|\alpha(t) - \beta(t)\|$ is constant; see \cite{D21} for more details.
Hence $T$ can be constructed to be the derivative of a differentiable path in the linear isometry group of $X$,
i.e., $T$ is a tangent vector of the linear isometry group of $X$ at $I$.
The aforementioned integral curves can be constructed for any two points easily since $T$ is linear; see \cite[Section 4.1]{manifold} for more details.}.

We say that a well-positioned framework $(G,p)$ in a normed space $X$ is \emph{infinitesimally rigid} if every infinitesimal flex of $(G,p)$ is trivial.
The following result outlines a relationship between local and infinitesimal rigidity.
We first recall that a well-positioned framework $(G,p)$ is \emph{regular} if the rigidity operator $df_G(p)$ has maximal rank over all well-positioned placements of $G$ in $X$.

\begin{thm}\label{t:asimowroth}
	Let $(G,p)$ be a well-positioned framework in a normed space $X$.
	\begin{enumerate}[(i)]
	    \item \cite[Theorem 3.9]{D21} If $(G,p)$ is infinitesimally rigid, then it is also locally rigid.
	    \item \cite[Theorem 1.1 \& Lemma 4.4]{D19} If $(G,p)$ is regular and locally rigid, and the set of smooth points of $X$ is open,
	    then $(G,p)$ is infinitesimally rigid.
	\end{enumerate}
\end{thm}

Given a well-positioned framework $(G,p)$ in a normed space $X$, the following three potential properties of $(G, p)$ are of particular interest to us: $(G,p)$ is \emph{independent} if $\rank df_G(p) =|E|$;
$(G,p)$ is \emph{minimally (infinitesimally) rigid} if $(G,p)$ is both infinitesimally rigid and independent; and 
$(G,p)$ is \emph{redundantly (infinitesimally) rigid} if $(G-e,p)$ is infinitesimally rigid for every edge $e \in E$.

\subsection{Regular frameworks and placements} \label{subsec:reg}

We will require the following strengthening of the notion of a regular framework.
A well-positioned framework $(G,p)$ is \emph{strongly regular} if every framework equivalent to $(G,p)$ has maximal rank.
The placement $p$ of a set $V$ is \emph{completely regular} if for all graphs $H$ with $V(H) := V$, $(H,p)$ is well-positioned and regular,
and \emph{completely strongly regular} if for all graphs $H$ with $V(H) := V$, $(H,p)$ is well-positioned and strongly regular.
We denote the set of regular placements of $G$ in $X$ by $\reg(G;X)$, the set of strongly regular placements of $G$ in $X$ by $\sreg(G;X)$, the set of completely regular placements of a set $V$ in $X$ by $\creg (V;X)$,
and the set of completely strongly regular placements of a set $V$ in $X$ by $\csreg (V;X)$.
It is immediate that if a framework is either strongly or completely regular then it is regular,
however, Figure \ref{fig:notcomreg} highlights the fact that
completely regular and strongly regular are properties that do not obey a hierarchy.

 \begin{figure}[htp]
\begin{center}
\begin{tikzpicture}[scale=.4]

\filldraw (0,0) circle (3pt)node[anchor=east]{};
\filldraw (-0.5,3.5) circle (3pt)node[anchor=east]{};
\filldraw (3.5,0) circle (3pt)node[anchor=north]{};
\filldraw (4,3.5) circle (3pt)node[anchor=south]{};
\filldraw (1.75,3.5) circle (3pt)node[anchor=east]{};

 \draw[black,thick]
(0,0) -- (3.5,0) -- (4,3.5) -- (0,0);

 \draw[black,thick]
(3.5,0) -- (4,3.5) -- (0,0);

 \draw[black,thick]
(3.5,0) -- (1.75,3.5) -- (0,0);

 \draw[black,thick]
(3.5,0) -- (-0.5,3.5) -- (0,0);

\draw[black,dashed]
(-1,3.5) -- (4.5,3.5);

        \end{tikzpicture}
         \hspace{0.5cm}
     \begin{tikzpicture}[scale=.4]
\filldraw (0,0) circle (3pt)node[anchor=east]{};
\filldraw (0,3.5) circle (3pt)node[anchor=east]{};
\filldraw (3.5,0) circle (3pt)node[anchor=north]{};
\filldraw (3.5,3.5) circle (3pt)node[anchor=south]{};

 \draw[black,thick]
(0,0) -- (0,3.5) -- (3.5,3.5) -- (3.5,0) -- (0,0);

\end{tikzpicture}
\end{center}
\vspace{-0.3cm}
\caption{Two frameworks in the Euclidean plane. The framework on the left is strongly regular but not completely regular as it has a colinear triple. The framework on the right is completely regular but not strongly regular as we can flatten the framework into a colinear non-regular framework.}
\label{fig:notcomreg}
\end{figure}
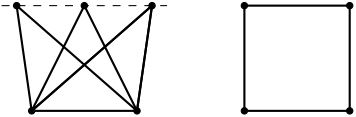

\begin{rem}\label{rem:diffman}
Let $(G,p)$ be a strongly regular framework in a normed space. It follows immediately from the definition of strongly regular that the point $f_G(p)$ is a regular value of the map $f_G$. This implies that the set $f_G^{-1}(f_G(p))$ of equivalent placements forms a differentiable manifold (see \cite[Theorem 3.5.4]{manifold}).
\end{rem}

The set $\reg(G;X)$ will always be an open subset of the set of well-positioned placements of $G$ in $X$ (see \cite[Lemma 4.4]{D19}),
which in turn is a conull subset of $X^V$.
Hence the set $\reg(G;X)$ is not a null set.
Fortunately, the situation is drastically improved for analytic normed spaces.

\begin{prop}[{\cite[Propositions 3.2 and 3.6]{DN}}]\label{prop:regcreg}
	Let $X$ be an analytic normed space and $G=(V,E)$ any graph.
	Then $\reg (G;X)$, $\sreg (G;X)$, $\creg (V;X)$ and $\csreg (V;X)$ are open conull subsets of $X^V$.
\end{prop}

Similarly to analytic normed spaces,
the sets of regular, strongly regular, completely regular and completely strongly regular placements in any Euclidean space are all open conull sets.
Many results in Euclidean rigidity theory are worded in terms of \emph{generic placements};
i.e., a placement $p:V \rightarrow \mathbb{E}^d$ where the coordinates of $p$ form an algebraically independent set.
In our language, any generic placement is completely strongly regular (this follows from \cite[Proposition 3.3]{C2005}), however not every completely strongly regular framework will be generic; this is immediate since the set $\csreg(V;\mathbb{E}^d)$ is open and conull, while the set of generic placements is only conull.

\subsection{Combinatorial rigidity in normed planes}\label{subsec:defs}

Given a connected graph, the placement of this graph in X that takes every vertex to the same point will give a locally rigid framework. The existence of such degenerate placements motivates us to consider, rather than whether there exists a placement giving a locally rigid framework, whether there exists an open set of placements corresponding to locally rigid frameworks.
If a framework $(G,p)$ is infinitesimally rigid in a normed space $X$,
then there exists an open set $U \subset X^V$ where for all $q \in U$, the framework $(G,q)$ is locally rigid \cite[Corollary 3.10]{D21}. This fact underlines the usefulness of infinitesimal rigidity and leads us to give the following definitions.
A graph $G = (V,E)$ is \emph{rigid} (respectively, \emph{minimally rigid}, \emph{redundantly rigid}) in a normed space $X$ if there exists a well-positioned framework $(G,p)$ in $X$ which is infinitesimally rigid (respectively, minimally rigid, redundantly rigid).
A graph that is not rigid in $X$ is said to be \emph{flexible}.

\begin{rem}\label{rem:infnorm}
An equivalent condition for a graph $G$ to be minimally rigid in a normed space $X$ is that
	$G$ is rigid in $X$ and $G-e$ is flexible in $X$ for all $e \in E$
	 (see \cite[Proposition 1.3.17]{dewphd} for a full proof). 
	However, perhaps surprisingly, it is not true for all normed spaces that $G$ is redundantly rigid in $X$ if and only if $G -e$ is rigid for all $e \in E$ in $X$.
	For an example, take the normed space $\ell_\infty^2 := (\mathbb{R}^2, \| \cdot \|_{\infty})$, where $\|(a,b)\|_\infty := \max\{|a|,|b|\}$.
	For any well-positioned framework $(G,p)$,
	we can define a partition $E_1 \cup E_2$ of the edges of $G=(V,E)$,
	where $vw \in E_1$ (respectively, $vw \in E_2$) if and only if, given $(x_1,x_2) = p_v-p_w$, we have $|x_1| > |x_2|$ (respectively, $|x_2| > |x_1|$).
	It was shown in \cite{K15} that the rank of $df_G(p)$ is equal to the rank of the oriented incidence matrices of the \emph{monochromatic subgraphs} $G_1 = (V,E_1)$ and $G_2 = (V,E_2)$ of $(G,p)$.
	Hence a well-positioned framework is infinitesimally rigid in $\ell_\infty^2$ if and only if both its monochromatic subgraphs are connected,
	and redundantly rigid in $\ell_\infty^2$ if and only if both its monochromatic subgraphs are 2-edge-connected.
	Consider the graph $K_5^-$, the complete graph on five vertices minus an edge.
	While $K_5^- - e$ is rigid in $\ell^2_\infty$ for any edge $e \in E(K_5^-)$ (see Theorem \ref{thm:rigidne} below),
	any placement will not be redundantly rigid, as the edges of $K_5^-$ cannot be decomposed into two 2-edge-connected spanning subgraphs.
\end{rem}

Given a non-negative integer $k$, a graph $G=(V,E)$ is \emph{$(2,k)$-sparse} if $|E'|\leq 2|V'|-k$ for every subgraph $(V',E')$ of $G$ (with $|E'|>0$). Further $G$ is \emph{$(2,k)$-tight} if it is $(2,k)$-sparse and $|E|=2|V|-k$.
The following result of Pollaczek-Geiringer \cite{PG} exactly characterises which graphs are minimally rigid in the Euclidean plane.

\begin{thm}\label{thm:rigide}
	A graph is minimally rigid in $\mathbb{E}^2$ if and only if it is isomorphic to $K_{1}$ or it is $(2,3)$-tight.
\end{thm}

The next result (proved in full generality in \cite{dew2}, building on the special case of $\ell_p^2$ \cite{kit-pow-1}) gives an analogous characterisation of the class of minimally rigid graphs for any non-Euclidean normed plane.

\begin{thm}\label{thm:rigidne}
	A graph is minimally rigid in a non-Euclidean normed plane if and only if it is $(2,2)$-tight.
\end{thm}

By a classical result of Nash-Williams \cite{N-W}, a graph is $(2,2)$-tight if and only if it is the edge-disjoint union of two spanning trees. Such graphs are well studied, for example in tree packing problems in combinatorial optimisation \cite{Fra}.
By comparing Theorem \ref{thm:rigide} and Theorem \ref{thm:rigidne} we observe that there are graphs which are rigid in non-Euclidean normed planes that are flexible in $\mathbb{E}^2$, 
an example being two copies of $K_4$ intersecting in a single vertex, 
and vice versa, 
an example being the complete bipartite graph $K_{3,3}$.

We define a graph $G$ to be \emph{globally rigid} in a normed space $X$ if the set
\begin{align*}
    \grig (G; X) := \left\{ p \in X^V : (G,p) \text{ is globally rigid} \right\}
\end{align*}
has an open interior.
In the literature it is common to define a graph $G$ to be globally rigid in $\mathbb{E}^d$ if (and only if) there exists a globally rigid generic framework $(G,p)$ in $\mathbb{E}^d$.
It follows from the ``averaging theorem'' of Connelly and Whiteley \cite{CW} that this definition is equivalent to our own.
Furthermore, it was proven in \cite{C2005,GHT} that global rigidity in Euclidean spaces is a ``generic property'' (i.e., if the property holds for a single generic placement in the space then it holds for all of them). Hence the interior of the set $\grig (G; \mathbb{E}^d)$
will either be empty (in which case $G$ is not globally rigid) or an open conull set (in which case $G$ is globally rigid).

We next state two key combinatorial results for global rigidity in Euclidean spaces. The main purpose of this paper is to prove precise analogues of both results for analytic normed spaces.

\begin{thm}[\cite{hendrickson}]\label{thm:euclidhendrickson}
	If a graph $G$ is globally rigid in $\mathbb{E}^d$, then $G$ is either a complete graph on at most $d+1$ vertices or $G$ is $(d+1)$-connected and redundantly rigid in $\mathbb{E}^d$.
\end{thm}

\begin{thm}[\cite{C2005,hendrickson,J&J}]\label{thm:euclidglobal}
	A graph $G$ is globally rigid in $\mathbb{E}^2$ if and only if $G$ is either a complete graph on at most 3 vertices or $G$ is 3-connected and redundantly rigid in $\mathbb{E}^2$.
\end{thm}

It is easy to find examples of graphs that are globally rigid in the Euclidean plane but not in other normed planes. 
In particular,
any wheel graph is globally rigid in $\mathbb{E}^2$ by \cite{BJ},
but in any non-Euclidean normed plane such graphs are minimally rigid and hence not redundantly rigid. 
In due course we will see that redundant rigidity is also a necessary condition for global rigidity in analytic normed planes (see Theorem \ref{thm:hendrickson}). 
Conversely, it was shown in \cite{DN} that a graph formed by taking two large complete graphs and gluing them at a single edge is globally rigid in any analytic normed plane, 
whereas in $\mathbb{E}^2$ there is an obvious violation of global rigidity obtained by reflecting one of the parts through the line defined by the edge in common.

\section{Necessary conditions for global rigidity in normed spaces} \label{sec:hendrickson}

In this section we develop necessary Hendrickson-type conditions \cite{hendrickson} for graphs to be globally rigid. We will work in the generality of normed spaces, though we will occasionally require the additional assumption that the normed space contains only finitely many linear isometries.
After this section we will focus solely on non-Euclidean normed planes, which always have finitely many linear isometries (see \cite[pg. 83]{minkowski}).

\subsection{Globally rigid graphs are 2-connected}

We first prove that all globally rigid graphs are 2-connected.

\begin{thm}\label{thm:connectivity}
	If $G=(V,E)$ is globally rigid in a non-Euclidean normed space $X$ and $|V| \geq 2$,
	then $G$ is $2$-connected.
\end{thm}

\begin{proof}
	Suppose for a contradiction that $G$ is globally rigid in $X$ but it is not 2-connected. 
	First suppose that $G$ is a complete graph with two vertices $v,w$.
	Fix $(G,p)$ to be a globally rigid framework in $X$. 
    By applying translations to $p$,
    we may suppose that $p_w = 0$,
    and hence $p_v \in B_r := \{ x \in X : \|x\|=r\}$.
    For any point $x \in B_r$,
    note that the framework $(G,q)$, where $q_w = 0$ and $q_v = x$, will be equivalent to $(G,p)$.
    As $(G,p)$ is globally rigid, there exists a linear isometry of $X$ that maps $p_v$ to $x$.
    Hence the set of linear isometries of $X$ act transitively on $B_r$.
    However, this implies $X$ is Euclidean (see \cite[Corollary 3.3.5]{minkowski}),
    a contradiction.
	
	Now suppose that $G$ is not the complete graph with 2 vertices.
	Hence there exists $u \in V$ and a partition $V_1,V_2$ of $V \setminus \{u\}$ so that there is no edge of $G$ connecting a vertex in $V_1$ to a vertex in $V_2$.
	Choose any well-positioned placement $p$ of $G$ with $p_u = 0$ so that $p$ lies in the interior of $\grig (G;X)$.
	By perturbing if necessary,
	we may also assume that: (i) for any distinct vertices $v,w$ we have $p_v\neq p_w$, and (ii) there exist vertices $v_1 \in V_1$ and $v_2 \in V_2$ where $\|p_{v_1} - p_{v_2}\| \neq \|p_{v_1} + p_{v_2}\|$.
	To see why we may assume point (ii),
	we note that if $\|p_{v_1} - p_{v_2}\| = \|p_{v_1} + p_{v_2}\|$,
	we will exchange $p$ for the placement $q \in \grig (G;X)$ where $q_{v_1} = p_{v_1} + \delta(p_{v_1}+p_{v_2})$ and $q_{v_2} = p_{v_2} + \delta(p_{v_1}+p_{v_2})$ for sufficiently small $\delta>0$ and $q_v=p_v$ otherwise,
	as then $\|q_{v_1} - q_{v_2}\| = \|p_{v_1} - p_{v_2}\|$ and $\|q_{v_1} + q_{v_2}\| = (1+\delta)\|p_{v_1} + p_{v_2}\|$;
	indeed if $\|p_{v_1} - p_{v_2}\| = (1+\delta)\|p_{v_1} + p_{v_2}\|$ also,
	then $p_{v_1}=p_{v_2}$, contradicting our assumption that $p$ is an injective map.
	
	Define the placement $p'$ of $G$ with
$p'_v =	p_v$ if $v \in V_1 \cup \{u\}$ and $p'_v:=	-p_v$ if  $v \in V_2$.
	Then $(G,p) \sim (G,p')$ but $p \not\sim p'$ since $\|p'_{v_1} - p'_{v_2} \| = \|p_{v_1} + p_{v_2} \| \neq \|p_{v_1} - p_{v_2} \|.$
	Thus $(G,p)$ is not globally rigid.
	Hence $\grig (G;X)$ cannot have a non-empty interior so $G$ is not globally rigid in $X$, a contradiction.
\end{proof}

Compared to Euclidean spaces, this is a relatively low connectivity requirement.
Indeed, any graph with at least $d+2$ vertices that is globally rigid in $\mathbb{E}^d$ must be $(d+1)$-connected by Theorem \ref{thm:euclidhendrickson}.
Figure \ref{fig:grexamples} shows two examples and relates them to this connectivity result.
However, for any normed space with finitely many linear isometries, 2-connectivity is a tight necessary condition as we now prove.

  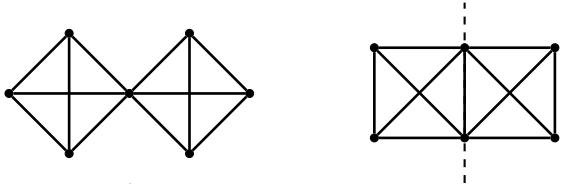
\begin{figure}[htp]
\begin{center}
     \begin{tikzpicture}[scale=0.8]
    \node[vertex] (a) at (2,0) {};
	\node[vertex] (b) at (1,-1) {};
    \node[vertex] (c) at (1,1) {};
    
	\node[vertex] (d) at (0,0) {};
	
    \node[vertex] (e) at (-2,0) {};
	\node[vertex] (f) at (-1,-1) {};
    \node[vertex] (g) at (-1,1) {};
	
	\draw[edge] (a)edge(b);
	\draw[edge] (a)edge(c);
	\draw[edge] (a)edge(d);
	\draw[edge] (b)edge(c);
	\draw[edge] (b)edge(d);
	\draw[edge] (c)edge(d);
    
	\draw[edge] (d)edge(e);
	\draw[edge] (d)edge(f);
	\draw[edge] (d)edge(g);
	\draw[edge] (e)edge(f);
	\draw[edge] (e)edge(g);
	\draw[edge] (f)edge(g);
	
	\filldraw[black] (0,-1.5) circle {};
\end{tikzpicture}\qquad \qquad
    \begin{tikzpicture}[scale=0.6]
    \node[vertex] (a) at (-2,-1) {};
	\node[vertex] (b) at (-2,1) {};
	
    \node[vertex] (c) at (0,1) {};
	\node[vertex] (d) at (0,-1) {};
	
    \node[vertex] (e) at (2,-1) {};
	\node[vertex] (f) at (2,1) {};

	\draw[edge] (a)edge(b);
	\draw[edge] (a)edge(c);
	\draw[edge] (a)edge(d);
	\draw[edge] (b)edge(c);
	\draw[edge] (b)edge(d);
	\draw[edge] (c)edge(d);
    
	\draw[edge] (c)edge(e);
	\draw[edge] (c)edge(f);
	\draw[edge] (d)edge(e);
	\draw[edge] (d)edge(f);
	\draw[edge] (e)edge(f);

\draw[thick,dashed]
(0,-2) -- (0,2);
\end{tikzpicture}
\end{center}
\vspace{-0.3cm}
\caption{
(Left): A graph that is minimally rigid in any non-Euclidean normed plane, but dependent and flexible in the Euclidean plane.
However, it is not globally rigid in any normed space as it is not 2-connected.
(Right): A graph that is rigid in any normed plane. 
This graph is also globally rigid in all analytic normed planes \cite{DN} (see Theorem \ref{lem:basegraphs} below). It is not globally rigid in the Euclidean plane; for almost any placement, the left two vertices may be reflected across the dashed line to obtain an equivalent but non-congruent framework. (Note that this reflection does not exist in most normed planes.)}
\label{fig:grexamples}
\end{figure}

\begin{prop}
    Let $X$ be a normed space with finitely many linear isometries.
    Suppose there exists a graph $G=(V,E)$ that is globally rigid in $X$.
    Then there exists a graph $G'=(V',E')$ that is globally rigid in $X$ but not 3-connected.
\end{prop}

\begin{proof}
    Choose an edge $v_1v_2 \in E$.
    Define $G'=(V',E')$ to be the graph formed from gluing two copies $G_1 = (V_1,E_1)$ and $G_2 = (V_2,E_2)$ of $G$ at the edge $v_1v_2$.
    As $X$ only has finitely many linear isometries,
    there exists an open dense set of points that are not invariant under any non-trivial linear isometry of $X$.
    Hence we may choose an open set $U \subset \grig(G;X)$ where for each $p \in U$, the vector $p_{v_1}-p_{v_2}$ is not invariant under any non-trivial linear isometry of $X$.
    Define $U'$ to be the set of placements of $G'$ in $X$ where for each $p \in U'$ we have $p|_{V_i} \in U$ for each $i \in \{1,2\}$.
    Since the set $U$ is open, it follows that $U'$ is also open.
    Choose a placement $p \in U'$,
    and choose a placement $q \in X^{V'}$ so that $(G,q) \sim (G,p)$ and $q_{v_1} = p_{v_1}$.
    By applying translations, we may assume $p_{v_1}=0$.
    Since both $(G_1,p|_{V_1})$ and $(G_2,p|_{V_2})$ are globally rigid,
    there exist linear isometries $T_1,T_2$ of $X$ such that $T_i(p_v)= q_v$ for all $v \in V_i$.
    Importantly, $T_1(p_{v_2}) = T_2(p_{v_2}) = q_{v_2}$.
    As $p_{v_2} - p_{v_1} = p_{v_2}$ is invariant under linear isometries and $p_{v_2} = T_2^{-1} \circ T_1 (p_{v_2})$,
    we must have $T_1 = T_2$.
    Hence $q \sim p$ and $(G,p)$ is globally rigid.
\end{proof}

\subsection{Global rigidity implies redundant rigidity} \label{sec:redundant}

Recalling the relevant notation from Section \ref{sec:background},
we define for a normed space $X$ and any graph $G=(V,E)$ the set
\begin{align*}
	X^{IF}_{G} := \left\{ p \in X_{G} : g \circ p = p \Rightarrow g = I \text{ for all } g \in \Iso (X) \right\}.
\end{align*}
Note that if $p \in X_G$ is spanning,
then $p \in X_G^{IF}$.
Using this fact we immediately obtain the following result.

\begin{lem}\label{l:quotopenconull}
	Let $X$ be a $d$-dimensional normed space and $G = (V,E)$.
	If $|V| \geq d+1$ then $X^{IF}_{G}$ is an open conull set.
\end{lem}

We can also determine whether a placement lies in $X^{IF}_{G}$ solely by the dimension of its kernel,
although this will only be useful for sufficiently large graphs.

\begin{lem}[{\cite[Lemma 3.9]{DN}}]\label{l:quotcount}
	Let $X$ be a $d$-dimensional smooth normed space,
	$G = (V,E)$ and $p \in X_G$ be a placement of $G$ with $\dim \ker df_G(p) < d-1 + |V|$.
	Then $p \in X^{IF}_{G}$.
\end{lem}

Given a normed space $X$ and two placements $p,p'$ of a set $V$, we remember that $p \sim p'$ if and only if there exists an isometry $g \in \Iso (X)$ with $p' = g \circ p$. It is immediate that if $p \in X_G$ and $g \in \Iso (X)$ then $g \circ p \in X_G$. Define the quotient space $X^V/\sim$ with equivalence classes $\tilde{p} := \{ q \in X^V : q \sim p\}$ and quotient map $\pi : X^V \rightarrow X^V/\sim$.

\begin{lem}\label{l:quotman}
	Let $X$ be a normed space and $M$ be a differentiable submanifold of $X^{IF}_{G}$, 
	where $g \circ p \in M$ for each $p \in M$ and $g \in \Iso (X)$.
	Then the set $\pi(M)$ is a differentiable submanifold of $X^V/\sim$ with $\dim \pi(M) = \dim M - \dim \Iso (X)$,
	and the map $\pi$ restricted to $M$ and $\pi(M)$ is a submersion (i.e., a differentiable map with surjective derivative everywhere).
\end{lem}

\begin{proof}
	For each point $p \in M$, since $p \in X^{IF}_{G}$ the map $g \mapsto g \circ p$ is injective. 
	The result now follows from \cite[Lemma 3.3]{D19} and \cite[Proposition 4.1.23]{mechanics}.
\end{proof}

We note that if $p \sim q$ then $f_G(p)=f_G(q)$ and hence the \emph{quotient rigidity map} 
\begin{align*}
\tilde{f}_G : X^{IF}_{G}/\sim \rightarrow \mathbb{R}^{E}, ~ \tilde{p} \mapsto \left( \frac{1}{2}\|p_v - p_w \|^2 \right)_{vw \in E}
\end{align*}
is well-defined.
Define the set 
\begin{align*}
\config := \tilde{f}_G^{-1} \left( \tilde{f}_G(\pi(p)) \right).
\end{align*}
If $f_G$ is differentiable at a point $p$, it follows that $\tilde{f}_G$ is differentiable at $\pi (p)$.

We now focus on $d$-dimensional smooth normed spaces with only finitely many linear isometries. Since all isometries are a combination of a linear isometry and a translation \cite{mazurulam}, the isometry group is $d$-dimensional.

\begin{lem}\label{l:config}
	Let $X$ be a $d$-dimensional smooth normed space with finitely many linear isometries,
	$(G,p)$ be a strongly regular framework in $X$ and $k := \dim \ker df_G(p) - d$.
	If $k < |V|-1$ and $G$ is connected, then $\config$ is a compact differentiable submanifold of $X^{IF}_{G}/\sim$ of dimension $k$.
\end{lem}

\begin{proof}
	We first note that $f_G^{-1}(f_G(p))$  is a subset of $X_G^{IF}$;
	this follows from noting that, by our assumption that $(G,p)$ is strongly regular, for every $p' \in f_G^{-1}(f_G(p))$ we have $\dim \ker df_G(p') = \dim \ker df_G(p) < d -1 + |V|$ 
	and then applying Lemma \ref{l:quotcount}.
	As $(G,p)$ is strongly regular the set $f_G^{-1}(f_G(p))$ is a differentiable submanifold of $X_G^{IF}$ (see Remark \ref{rem:diffman}).
	Since $\config = \pi (f_G^{-1}(f_G(p)))$, the set $\config$ is a differentiable submanifold of $X^{IF}_{G}/\sim$ with dimension $k$ by Lemma \ref{l:quotman}.
	
	Choose any $v_0 \in V$ and define the closed set $S := \{ q \in X^V : f_G(q) = f_G(p), ~ q_{v_0} = 0 \}$.
	Since $G$ is connected,
	the set $S$ is compact.
	As $\pi (S) = \config$ and the map $\pi$ is continuous,
	the set $\config$ must also be compact.
\end{proof}

We are now ready to prove the following sufficient condition for global rigidity.

\begin{thm}\label{thm:red}
	Let $(G,p)$ be a completely strongly regular globally rigid framework in a smooth non-Euclidean $d$-dimensional normed space $X$ with finitely many linear isometries.
	If $G$ has at least two vertices,
	then $(G,p)$ is redundantly rigid.
\end{thm}

\begin{proof}
	Assume $G$ is not redundantly rigid in $X$, 
	i.e.~there exists $vw \in E$ such that $G- vw$ is flexible.
	The case when $G$ is flexible trivially contradicts globally rigid by Theorem \ref{t:asimowroth}.
	Hence we may assume $G$ is rigid and so $|V| > 2$ (this follows since $K_2$ is always flexible, see \cite[Theorem 5.8]{D19}).
	Since $\dim \ker df_G(p) = d$,
	we have that $\dim \ker df_{G-vw}(p) - d = 1 < |V| - 1$.
	Hence, by Lemma \ref{l:config}, 
	the set $\config$ is finite and $\confige$ is a compact 1-dimensional differentiable manifold. Therefore the connected component of $\confige$ containing $\tilde{p}$ is diffeomorphic to a circle (see \cite[Theorem 5.27]{manifoldlee}). Thus
	there exists a continuously differentiable map $\phi : [0,1] \rightarrow \confige$ with $\phi(0)=\phi(1)=\tilde{p}$ and $\phi(t_1) \neq \phi(t_2)$ if $0<t_1 < t_2< 1$.
	Define the continuously differentiable function
	\begin{align*}
		h_{vw}: f^{-1}_{G-vw} \left( f_{G-vw} (p) \right) \rightarrow \mathbb{R}, ~ q \mapsto \|q_v - q_w \|^2.
	\end{align*}
	If $\pi(q) = \pi(q')$ then $h_{vw}(q) = h_{vw}(p)$ and
	hence the map
	\begin{align*}
		\tilde{h}_{vw}: \confige \rightarrow \mathbb{R}, ~ \tilde{q} \mapsto \|q_v - q_w \|^2
	\end{align*}
	is well-defined.
	Since both $h_{vw}$ and $\pi$ are continuously differentiable and $\tilde{h}_{vw} \circ \pi = h_{vw}$,
	the map $\tilde{h}_{vw}$ must also be continuously differentiable.
	With this,
	we can now define the continuously differentiable function 
	\begin{align*}
		g: [0,1] \rightarrow \mathbb{R}, ~ t \mapsto \tilde{h}_{vw}(\phi(t)).
	\end{align*}
	We will denote $t_{\max}$ and $t_{\min}$ to be the points where $g$ attains its maximum and minimum values respectively.
	Note that we cannot have $t_{\max}= t_{\min}$ (i.e., $g$ is constant),
	as this would imply the set $\config$ contains infinitely many points, 
	contradicting that $(G,p)$ is locally rigid and $p$ is strongly regular.
	
	First suppose that $t_{\max}=0$ or $t_{\min}=0$.
	As $0$ is a local minimum/maximum of $g$, we must have $g'(0)=0$,
	i.e., given $\tilde{u}$ is the unique (up to scaling) vector in the tangent space of $\confige$ at $\tilde{p}$, we have $d\tilde{h}_{vw}(\tilde{p})(\tilde{u}) = 0$.
	As $\pi$ restricted to $\confige$ and $\pi(\confige)$ is a submersion (Lemma \ref{l:quotman}),
	there exists a vector $u \in X^V$ in the tangent space of $f^{-1}_{G-vw} \left( f_{G-vw} (p) \right)$ at $p$ such that: (i) $dh_{vw}(p)(u) = 0$, and (ii) $d \pi (p)(u) \neq 0$.
	This corresponds to $u$ being a non-trivial infinitesimal flex of $(G-vw,p)$ where $\varphi_{p_v- p_w}(u_v-u_w) = 0$,
	contradicting that $(G,p)$ is infinitesimally rigid.
	
	Without loss of generality,
	we may now assume $0 < t_{\min} < t_{\max} < 1$.
	By the Intermediate Value Theorem, there exists $t_{\min} <t_0 < t_{\max}$ such that $g(t_0) = g(0)$.
	Hence $\tilde{f}_G(\phi(t_0)) = \tilde{f}_G(\tilde{p})$ but $\phi(t_0) \neq \tilde{p}$.
	It follows that for any placement $q$ of $G$ with $\pi(q)=\phi(t_0)$, 
	we have $(G,q) \sim (G,p)$ but $q \not\sim p$,
	thus $(G,p)$ is not globally rigid.
\end{proof}

\subsection{Global rigidity in analytic normed spaces}

It remains to show that these necessary conditions for frameworks to be globally rigid extend to necessary conditions for graphs to be globally rigid. To do this we
must also restrict attention to analytic normed spaces.

While any globally and infinitesimally rigid framework must be strongly regular,
it is not required to be completely strongly regular.
We also do not know if such a placement would even exist in a given normed space.
For example,
there exist no completely regular placements of any set with five or more elements in $\ell_\infty^2$,
as one of the monochromatic subgraphs of any $K_5$ subgraph must contain a cycle, and the framework restricted to that monochromatic subgraph will not be regular (see Remark \ref{rem:infnorm} for more details). Fortunately this not the case for analytic normed spaces (see Proposition \ref{prop:regcreg}), allowing us to obtain the main result of this section.

\begin{thm}\label{thm:hendrickson}
	Let $G = (V,E)$ be globally rigid in an analytic normed space $X$ with finitely many linear isometries.
	If $|V| \geq 2$, then $G$ is redundantly rigid and $2$-connected.
\end{thm}

\begin{proof}
	By Theorem \ref{thm:connectivity},	$G$ is $2$-connected.
	As $G$ is globally rigid, there exists an open set $U \subset X^V$ of globally rigid placements of $G$.
	By Proposition \ref{prop:regcreg}, we may choose a completely strongly regular placement $p \in U$.
	Since all analytic normed spaces are smooth, Theorem \ref{thm:red} implies that $(G,p)$ is redundantly rigid, and hence $G$ is redundantly rigid in $X$.
\end{proof}

We conjecture that the theorem remains valid in a wider family of normed spaces including, for example, all smooth $\ell_p$-spaces.
The remainder of the paper will be devoted to showing that for analytic normed planes these conditions are sufficient as well as necessary.

\section{Sparse graphs and circuits} \label{sec:circuits}

In the next three sections we provide a detailed study of the graphs relevant to global rigidity in analytic normed planes.
We begin by describing the circuits in the following matroid.
Let $\mathcal{M}(2,2)$ denote the \emph{simple $(2,2)$-sparse matroid}, that is the matroid on $E(K_n)$\footnote{It is more common in the literature for the $(2,2)$-sparse matroid to be defined on the complete multigraph (i.e., the multigraph with two parallel edges between every pair of vertices) since parallel edges may be independent. Hence, while all graphs in this article are simple, we use the word simple for the $(2,2)$-sparse matroid to emphasise this distinction.} in which a set $E\subseteq E(K_n)$ is independent if $E = \emptyset$ or $E \neq \emptyset$ and the subgraph induced by $E$ is $(2,2)$-sparse. 
We will slightly abuse notation by using \emph{$(2,2)$-circuit} to refer both to a circuit in the matroid $\mathcal{M}(2,2)$ as well as the graph induced by a circuit in the matroid $\mathcal{M}(2,2)$.
In graph theoretic terms, $G = (V,E)$ is a $(2,2)$-circuit if and only if $|E|=2|V|-1$ and $|E'|\leq 2|V'|-2$ for all proper subgraphs $(V',E')$ of $G$. 
Three examples of $(2,2)$-circuits are illustrated in Figure \ref{fig:smallgraphs1}.
 
\begin{figure}[htp]
\begin{center}
\begin{tikzpicture}[scale=.4]
\filldraw (-.5,0) circle (3pt);
\filldraw (0,3) circle (3pt);
\filldraw (3.5,0) circle (3pt);
\filldraw (3,3) circle (3pt);
\filldraw (1.5,-1.5) circle (3pt);

\draw[black,thick]
(1.5,-1.5) -- (-.5,0) -- (0,3) -- (3.5,0) -- (3,3) -- (-.5,0) -- (3.5,0);

\draw[black,thick]
(1.5,-1.5) -- (0,3) -- (3,3) -- (1.5,-1.5);

\node [rectangle, draw=white, fill=white] (b) at (1.5,-3) {$K_5^-$};
        \end{tikzpicture}
         \hspace{0.5cm}
     \begin{tikzpicture}[scale=.4]
\filldraw (0,0) circle (3pt);
\filldraw (0,3.5) circle (3pt);
\filldraw (3.5,0) circle (3pt);
\filldraw (3.5,3.5) circle (3pt);
\filldraw (7,0) circle (3pt);
\filldraw (7,3.5) circle (3pt);

\draw[black,thick]
(0,0) -- (0,3.5) -- (3.5,0) -- (3.5,3.5) -- (0,0) -- (3.5,0) -- (7,3.5);

\draw[black,thick]
(0,3.5) -- (3.5,3.5) -- (7,3.5) -- (7,0) -- (3.5,3.5);

\draw[black,thick]
(7,0) -- (3.5,0);

\node [rectangle, draw=white, fill=white] (b) at (3.5,-2) {$B_{1}$};
\end{tikzpicture}
       \hspace{0.5cm}
    \begin{tikzpicture}[scale=.4]
\filldraw (0,-1) circle (3pt);
\filldraw (-1.5,2) circle (3pt);
\filldraw (3,0) circle (3pt);
\filldraw (1.5,3) circle (3pt);
\filldraw (6,-1) circle (3pt);
\filldraw (4.5,3) circle (3pt);
\filldraw (7.5,2) circle (3pt);

\draw[black,thick]
(0,-1) -- (-1.5,2) -- (3,0) -- (1.5,3) -- (0,-1) -- (3,0) -- (4.5,3);

\draw[black,thick]
(-1.5,2) -- (1.5,3) -- (4.5,3) -- (6,-1) -- (7.5,2) -- (4.5,3);

\draw[black,thick]
(6,-1) -- (3,0) -- (7.5,2);

\node [rectangle, draw=white, fill=white] (b) at (3,-3) {$B_{2}$};
         \end{tikzpicture}
\end{center}
\vspace{-0.3cm}
\caption{Three small $(2,2)$-circuits.}
\label{fig:smallgraphs1}
\end{figure}
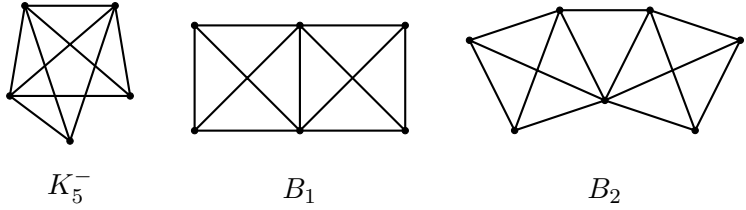

\subsection{Recursive constructions}

A recursive construction of $(2,2)$-circuits was first proved in \cite{Nix}. In order to describe it we require decomposition results for $(2,2)$-circuits, taken from \cite{Nix}, which use the graph operations defined in Figure \ref{fig:sums}.
These operations are as follows for a pair of graphs $G_1=(V_1,E_1)$ and $G_2=(V_2,E_2)$ that possibly share edges or vertices.
We denote the neighbourhood of a vertex $v$ in a graph $G$ by $N_G(v)$, and we denote the degree of $v$ in $G$ by $d_G(v)$; when the graph in question is clear we shall denote the neighbourhood and degree of $v$ by $N(v)$ and $d(v)$ respectively.
\begin{enumerate}[(i)]
    \item Suppose that for each $i \in \{1, 2\}$ there exists a subgraph $A_{i}=(U_i,F_i)$ of $G_{i}$ such that $A_1$ is the complete graph with $U_1 = \{a,b\}$, and $A_{2}$ is the complete graph with $U_2 =  \{a,b,c,d\}$.
    If $V_1 \cap V_2 = \{a,b\}$ and $d_{G_{2}}(c) = d_{G_{2}}(d) = 3$,
    then the \emph{1-join} of $(G_{1}, G_{2})$ is the graph $G=(V,E)$ where 
    \begin{align*}
        V = V_{1} \cup (V_{2} \setminus \{c, d\}), \qquad E = (E_{1} \setminus F_{1}) \cup (E_{2} \setminus F_{2}).
    \end{align*}
    See the top of Figure \ref{fig:sums} for a diagram of this operation.
    
    \item Suppose that for each $i \in \{1, 2\}$ there exists a subgraph $A_{i} = (U_{i}, F_{i})$ of $G_{i}$ such that $A_{1}$ is the complete graph with $U_1 = \{a,b,c_1,d_1\}$, and $A_{2}$ is the complete graph with $U_2 = \{a,b,c_2,d_2\}$.
    If $V_1 \cap V_2 = \{a,b\}$ and $d_{G_{2}}(c_i) = d_{G_{2}}(d_i) = 3$ for each $i\in \{1,2\}$,
    then the \emph{2-join} of $(G_{1}, G_{2})$ is the graph $G=(V,E)$ where 
    \begin{align*}
        V = (V_{1} \setminus \{c_1, d_1\}) \cup (V_{2} \setminus \{c_2, d_2\}), \qquad E = (E_{1} \setminus F_{1}) \cup (E_{2} \setminus F_{2}) + ab.
    \end{align*}
    See the middle of Figure \ref{fig:sums} for a diagram of this operation.
    
    \item Suppose that for each $i \in \{1,2\}$,
    there exists a vertex $v_i$ with neighbourhood $\{a_i,b_i,c_i\}$ in $G_i$,
    and set $F_i= \{a_iv_i,b_iv_i, c_i v_i\}$.
    If $V_1 \cap V_2 = \emptyset$,
    then the \emph{3-join} of $(G_{1}, G_{2})$ is the graph $G=(V,E)$ where 
    \begin{align*}
        V = (V_1 - v_1) \cup (V_2 - v_2), \qquad E = (E_{1} \setminus F_{1}) \cup (E_{2} \setminus F_{2}) \cup \{a_1a_2, b_1b_2, c_1c_2\}.
    \end{align*}
     See the bottom of Figure \ref{fig:sums} for a diagram of this operation.
\end{enumerate}

For every join operation, we can also define a separation operation that acts an inverse.
To define these operations, we recall two types of separation for a graph $G=(V,E)$. 
A \emph{$k$-vertex-separation} is a pair of induced subgraphs $G_1=(V_1,E_1)$, $G_2=(V_2,E_2)$ that share exactly $k$ vertices where $G=G_1 \cup G_2$, and neither $V_1\sm V_2$ nor $V_2 \sm V_1$ are empty;
the set $V_1 \cap V_2$ is then called a \emph{vertex cut}.
A \emph{$k$-edge-separation} is a pair of vertex-disjoint induced subgraphs $G_1=(V_1,E_1)$, $G_2=(V_2,E_2)$ where $G_1 \cup G_2 = G-S$ for some edge set $S \subseteq E$ with $|S| = k$;
the set $S$ is then called an \emph{edge cut}.
We will mainly be interested in 2-vertex-separations and 3-edge-separations.
We define a 2-vertex-separation $(G_1,G_2)$ to be {\em non-trivial} if neither $G_1$ nor $G_2$ are isomorphic to $K_4$,
and we define a 3-edge-separation $(G_1, G_2)$ to be {\em non-trivial} if the corresponding edge cut $S$ is independent (i.e., no two edges in $S$ share a vertex).
Using these definitions, we now define the following operations for a graph $G=(V,E)$ and pair of subgraphs $H_1=(U_1,F_1), H_2 = (U_2,F_2)$.
For any disjoint non-empty sets $S,S' \subseteq V$,
we denote the complete graph with vertex set $S$ by $K[S]$,
and the complete bipartite graph with parts $S,S'$ by $K[S,S']$.
\begin{enumerate}[(i)]
    \item Suppose that $(H_1,H_2)$ is a 2-vertex-separation with $U_1 \cap U_2 = \{a,b\}$ and $F_1 \cap F_2 = \emptyset$.
    A \emph{1-separation of $G$ (with respect to $(H_1,H_2)$)} is an ordered pair $(G_1,G_2)$ with $G_1 = H_1 \cup K[\{a,b\}]$ and $G_2 = H_2 \cup K[\{a,b,c,d\}]$ for two new vertices $c,d$.
    
    \item Suppose that $(H_1,H_2)$ is a 2-vertex-separation with $U_1 \cap U_2 = \{a,b\}$ and $F_1 \cap F_2 = \{ab\}$.
    A \emph{2-separation of $G$ (with respect to $(H_1,H_2)$)} is an ordered pair $(G_1,G_2)$ with $G_1 = H_1 \cup K[\{a,b,c_1,d_1\}]$ and $G_2 = H_2 \cup K[\{a,b,c_2,d_2\}]$ for four new vertices $c_1,c_2,d_1,d_2$.
    
    \item Suppose that $(H_1,H_2)$ is a non-trivial 3-edge-separation where the edges $a_1a_2,b_1b_2,c_1c_2$ connect $H_1$ and $H_2$.
    A \emph{3-separation of $G$ (with respect to $(H_1,H_2)$)} is an ordered pair $(G_1,G_2)$ with $G_1 = H_1 \cup K[\{v_1\}, \{a_1,b_1,c_1\}]$ and $G_2 = H_2 \cup K[\{v_2\}, \{a_2,b_2,c_2\}]$ for two new vertices $v_1,v_2$.
\end{enumerate}

\begin{center}
\begin{figure}[ht]
\centering
\begin{tikzpicture}[scale=0.8]

\filldraw (0,3.5) circle (2pt) node[anchor=south]{$a$};
\filldraw (0,2.5) circle (2pt) node[anchor=north]{$b$};

\filldraw (-4.5,2.5) circle (2pt) node[anchor=east]{$b$};
\filldraw (-4.5,3.5) circle (2pt) node[anchor=east]{$a$};
\filldraw (4.5,2.5) circle (2pt) node[anchor=west]{$b$};
\filldraw (4.5,3.5) circle (2pt) node[anchor=west]{$a$};
\filldraw (3.5,2.5) circle (2pt) node[anchor=east]{$d$};
\filldraw (3.5,3.5) circle (2pt) node[anchor=east]{$c$};

\draw[thick] plot[smooth, tension=1] coordinates{(-1.5,3.48) (-1.5,2.5) (-.5,2.3) (.3,2.4) (-.3,3) (.3,3.6)(-.5,3.7) (-1.5,3.48)};

\draw[thick] plot[smooth, tension=1] coordinates{(1.5,3.48) (1.5,2.5) (.5,2.3) (-.3,2.4) (.3,3) (-.3,3.6)(.5,3.7) (1.5,3.48)};

\draw[thick] plot[smooth, tension=1] coordinates{(-6,3.48) (-6,2.5) (-5,2.3) (-4.2,2.4) (-4.8,3) (-4.2,3.6)(-5,3.7) (-6,3.48)};

\draw[thick] plot[smooth, tension=1] coordinates{(6,3.48) (6,2.5) (5,2.3) (4.2,2.4) (4.8,3) (4.2,3.6)(5,3.7) (6,3.48)};

\draw[black] (-4.5,3.5) -- (-4.5,2.5);

\draw[black] (4.5,3.5) -- (4.5,2.5) -- (3.5,2.5) -- (3.5,3.5) -- (4.5,3.5) -- (3.5,2.5);

\draw[black] (4.5,2.5) -- (3.5,3.5);

\draw[black] (-3,3) -- (-2.5,3) -- (-2.6,3.1);

\draw[black] (-2.5,3) -- (-2.6,2.9);

\draw[black] (3,3) -- (2.5,3) -- (2.6,2.9);

\draw[black] (2.5,3) -- (2.6,3.1);

\filldraw (0,.5) circle (2pt) node[anchor=south]{$a$}; 
\filldraw (0,-.5) circle (2pt) node[anchor=north]{$b$};

\filldraw (-4.5,-.5) circle (2pt) node[anchor=east]{$b$};
\filldraw (-4.5,.5) circle (2pt) node[anchor=east]{$a$};
\filldraw (-3.5,-.5) circle (2pt) node[anchor=west]{$d_1$};
\filldraw (-3.5,.5) circle (2pt) node[anchor=west]{$c_1$};
\filldraw (4.5,-.5) circle (2pt) node[anchor=west]{$b$};
\filldraw (4.5,.5) circle (2pt) node[anchor=west]{$a$};
\filldraw (3.5,-.5) circle (2pt) node[anchor=east]{$d_2$};
\filldraw (3.5,.5) circle (2pt) node[anchor=east]{$c_2$};

\draw[thick] plot[smooth, tension=1] coordinates{(-1.5,0.48) (-1.5,-.5) (-.5,-.7) (.3,-.6) (-.3,0) (.3,.6)(-.5,.7) (-1.5,.48)};

\draw[thick] plot[smooth, tension=1] coordinates{(1.5,0.48) (1.5,-.5) (.5,-.7) (-.3,-.6) (.3,0) (-.3,.6)(.5,.7) (1.5,.48)};

\draw[thick] plot[smooth, tension=1] coordinates{(-6,0.48) (-6,-.5) (-5,-.7) (-4.2,-.6) (-4.8,0) (-4.2,.6)(-5,.7) (-6,.48)};

\draw[thick] plot[smooth, tension=1] coordinates{(6,0.48) (6,-.5) (5,-.7) (4.2,-.6) (4.8,0) (4.2,.6)(5,.7) (6,.48)};

\draw[black] (-4.5,.5) -- (-4.5,-.5) -- (-3.5,-.5) -- (-3.5,.5) -- (-4.5,.5) -- (-3.5,-.5);

\draw[black] (-4.5,-.5) -- (-3.5,.5);

\draw[black] (4.5,.5) -- (4.5,-.5) -- (3.5,-.5) -- (3.5,.5) -- (4.5,.5) -- (3.5,-.5);

\draw[black] (4.5,-.5) -- (3.5,.5);

\draw[black] (0,.5) -- (0,-.5);

\draw[black] (-3,0) -- (-2.5,0) -- (-2.6,.1);

\draw[black] (-2.5,0) -- (-2.6,-.1);

\draw[black] (3,0) -- (2.5,0) -- (2.6,-.1);

\draw[black] (2.5,0) -- (2.6,.1);

\draw (-5,-3) circle (27pt); 
\draw (5,-3) circle (27pt); 
\draw (1.2,-3) circle (27pt); 
\draw (-1.2,-3) circle (27pt);

\filldraw (-3.4,-3) circle (2pt) node[anchor=south]{$v_1$};
\filldraw (3.4,-3) circle (2pt) node[anchor=south]{$v_2$};
\filldraw (-4.2,-3) circle (2pt) node[anchor=east]{$b_1$};
\filldraw (4.2,-3) circle (2pt) node[anchor=west]{$b_2$};
\filldraw (-4.5,-3.5) circle (2pt) node[anchor=east]{$c_1$};
\filldraw (-4.5,-2.5) circle (2pt) node[anchor=east]{$a_1$};
\filldraw (4.5,-3.5) circle (2pt) node[anchor=west]{$c_2$};
\filldraw (4.5,-2.5) circle (2pt)node[anchor=west]{$a_2$};
\filldraw (0.7,-2.5) circle (2pt)node[anchor=west]{$a_2$};
\filldraw (.7,-3.5) circle (2pt) node[anchor=west]{$c_2$};
\filldraw (.4,-3) circle (2pt) node[anchor=west]{$b_2$};
\filldraw (-.7,-2.5) circle (2pt) node[anchor=east]{$a_1$};
\filldraw (-.7,-3.5) circle (2pt) node[anchor=east]{$c_1$};
\filldraw (-.4,-3) circle (2pt) node[anchor=east]{$b_1$};

\draw[black] (-.7,-2.5) -- (.7,-2.5);

\draw[black] (-.7,-3.5) -- (.7,-3.5);

\draw[black] (-.4,-3) -- (.4,-3);

\draw[black] (-4.5,-3.5) -- (-3.4,-3) -- (-4.2,-3);

\draw[black] (-3.4,-3) -- (-4.5,-2.5);

\draw[black] (4.5,-3.5) -- (3.4,-3) -- (4.2,-3);

\draw[black] (3.4,-3) -- (4.5,-2.5);

\draw[black] (-3,-3) -- (-2.5,-3) -- (-2.6,-3.1);

\draw[black] (-2.5,-3) -- (-2.6,-2.9);

\draw[black] (3,-3) -- (2.5,-3) -- (2.6,-3.1);

\draw[black] (2.5,-3) -- (2.6,-2.9);
\end{tikzpicture}
\caption{The 1-, 2- and 3-join operations.}
\label{fig:sums}
\end{figure}
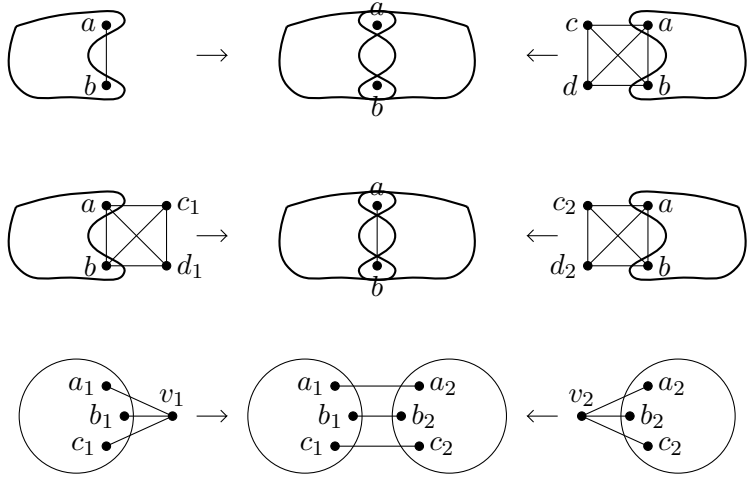
\end{center}

The join and separation operations are useful for constructing and deconstructing $(2,2)$-circuits.\footnote{For $(2,2)$-circuits, the 1-join operation can be defined based on the density of each `side' of the 2-vertex-separation. However in the next section we will apply these operations to a broader family of graphs where that convenience is unavailable. For this reason we state the next three lemmas differently to how they appeared in \cite{Nix}.}

\begin{lem}[{\cite[Lemmas 3.1, 3.2, 3.3]{Nix}}]\label{lem:circuit_dec}
Let $G_1$ and $G_2$ be graphs and suppose that $G$ is a $j$-join of $G_1$ and $G_2$ for some $1\leq j\leq 3$. If $G_1$ and $G_2$ are $(2,2)$-circuits then $G$ is a $(2,2)$-circuit.
\end{lem}

\begin{lem}[{\cite[Lemmas 3.2, 3.3]{Nix}}]\label{lem:circuit_dec2}
Let $G$ be a graph and suppose that $(G_1, G_2)$ is a $j$-separation of $G$
for some $j \in \{2, 3\}$. If $G$ is a $(2,2)$-circuit then $G_1$ and $G_2$ are $(2,2)$-circuits.
\end{lem}

\begin{lem}[{\cite[Lemma 3.1]{Nix}}]\label{lem:circuit_dec3}
Let $G$ be a graph with 2-vertex-separation $(H_1,H_2)$.
Suppose that $(G_{1}, G_{2})$ is a 1-separation of $G$ with respect to $(H_{1}, H_{2})$, and $(G'_2, G'_1)$ is a 1-separation of $G$ with respect to $(H_{2}, H_{1})$.
If $G$ is a $(2,2)$-circuit then either $G_{1}$ and $G_{2}$ are $(2,2)$-circuits
and $G'_{1}$ and $G'_{2}$ are not $(2,2)$-circuits,
or $G'_{1}$ and $G'_{2}$ are $(2,2)$-circuits
and $G_{1}$ and $G_{2}$ are not $(2,2)$-circuits.
\end{lem}

The recursive construction of $(2,2)$-circuits in \cite{Nix} begins with the three $(2,2)$-circuits shown in Figure \ref{fig:smallgraphs1} and uses the $j$-join operations as well as the well-known 1-extension operation. Given a graph $G=(V,E)$, a \emph{1-extension} forms a new graph by deleting an edge $xy\in E$ and adding one new vertex $v$ and three new edges $xv,yv,zv$ for some $z\in V\sm \{x,y\}$. Conversely, a \emph{1-reduction} deletes a vertex $v$ of degree 3 and adds an edge between any two non-adjacent neighbours of $v$.

\begin{thm}[{\cite[Theorem 1.1]{Nix}}]\label{thm:recurse}
Suppose $G$ is a $(2,2)$-circuit. Then $G$ can be obtained from disjoint copies of $K_5^-$, $B_1$, and $B_2$ by recursively applying $1$-extension operations within connected components and $1$-, $2$- and $3$-join operations of connected components.
\end{thm}

An alternative recursive construction of $(2,2)$-circuits was provided in \cite{JNglobal}. To describe this construction we first define two local graph operations on $G = (V,E)$. The first operation, the {\em $K_4^-$-extension}, deletes an edge $v_1v_2$, adds two new vertices $u_1$ and $u_2$, and adds five new edges $v_1 u_1, v_1u_2 , v_2 u_1, v_2 u_2, u_1 u_2$.
Conversely, a {\em $K_4^-$-reduction} deletes two adjacent vertices $u_1,u_2$ of degree 3, where $N_G(u_1) \cap N_G(u_2) = \{v_1,v_2\}$ and $v_1v_2 \notin E$, and then adds the edge $v_1v_2$.
The second operation, the {\em generalised vertex split}, is defined as follows: choose $v \in V$ and a partition $N_1,N_2$ of the neighbours of $v$; then delete $v$ from $G$ and add two new vertices $v_1$ and $v_2$, joined to $N_1$ and $N_2$ respectively; finally add two new edges $v_1v_2$ and $v_1x$ for some $x \in V \setminus N_1$.
Conversely, a \emph{edge-reduction} contracts an edge to a vertex and then deletes the resulting loop and one other neighbour of the new vertex.
These operations are illustrated in Figures \ref{fig:types_a} and \ref{fig:vsplit} respectively.

\begin{center}
\begin{figure}[ht]
\centering
\begin{tikzpicture}[scale=0.8]

\draw (-4,10) circle (27pt);
\draw (4,10) circle (27pt);

\filldraw (-2,10.5) circle (2pt) node[anchor=south]{$v_1$};
\filldraw (-2,9.5) circle (2pt) node[anchor=north]{$v_2$};

\draw[black] (-3.5,10.6) -- (-2,10.5) -- (-2,9.5) -- (-3.5,9.7);

\draw[black] (-3.5,10.4) -- (-2,10.5);

\draw[black] (-3.7,9.5) -- (-2,9.5) -- (-3.6,9.3);

\draw[black] (0,10) -- (1,10) -- (0.9,9.9);

\draw[black] (1,10) -- (0.9,10.1);

\filldraw (6,10.5) circle (2pt) node[anchor=south]{$v_1$};
\filldraw (6,9.5) circle (2pt) node[anchor=north]{$v_2$};
\filldraw (7,10.5) circle (2pt) node[anchor=south]{$u_1$};
\filldraw (7,9.5) circle (2pt) node[anchor=north]{$u_2$};

\draw[black] (4.5,10.6) -- (6,10.5) -- (7,10.5) -- (7.,9.5) -- (6,9.5) -- (7,10.5);

\draw[black] (6,9.5) -- (4.5,9.7);

\draw[black] (4.5,10.4) -- (6,10.5) -- (7,9.5);

\draw[black] (4.3,9.5) -- (6,9.5) -- (4.4,9.3);
\end{tikzpicture}
\caption{The $K_4^-$-extension.
The $K_4^-$-reduction is the inverse of the $K_4^-$-extension.} \label{fig:types_a}
\end{figure}
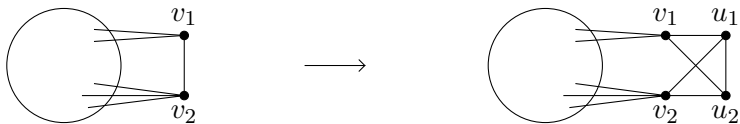

\begin{figure}[ht]
\begin{tikzpicture}[scale=0.7]
\draw (-3,-5) circle (35pt);
\draw (4,-5) circle (35pt);

\filldraw (-3,-2.5) circle (2pt) node[anchor=south]{$v$};
\filldraw (3.5,-2.5) circle (2pt) node[anchor=south]{$v_1$};
\filldraw (4.5,-2.5) circle (2pt) node[anchor=south]{$v_2$};

\filldraw (-3,-5) circle (2pt) node[anchor=north]{$x$};
\filldraw (4,-5) circle (2pt) node[anchor=north]{$x$};

\filldraw (-4,-4.5) circle (2pt);
\filldraw (-3.5,-4.5) circle (2pt);
\filldraw (-2.5,-4.5) circle (2pt);
\filldraw (-2,-4.5) circle (2pt);
\filldraw (3,-4.5) circle (2pt);
\filldraw (3.5,-4.5) circle (2pt);
\filldraw (4.5,-4.5) circle (2pt);
\filldraw (5,-4.5) circle (2pt);

\draw[black] (0,-5) -- (1,-5) -- (0.9,-5.1);

\draw[black] (1,-5) -- (0.9,-4.9);

\draw[black] (-4,-4.5) -- (-3,-2.5) -- (-3.5,-4.5);

\draw[black] (-2.5,-4.5) -- (-3,-2.5) -- (-2,-4.5);

\draw[black] (3,-4.5) -- (3.5,-2.5) -- (3.5,-4.5);

\draw[black] (4.5,-4.5) -- (4.5,-2.5) -- (5,-4.5);

\draw[black] (4,-5) -- (3.5,-2.5) -- (4.5,-2.5);
\end{tikzpicture}
\caption{Generalised vertex split. An edge-reduction is the inverse of the generalised vertex split.} \label{fig:vsplit}
\end{figure}
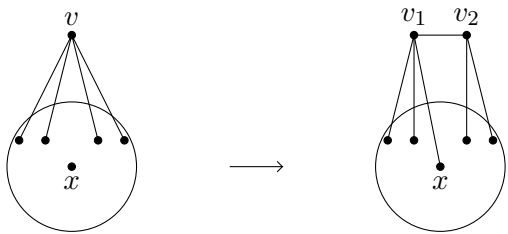
\end{center}

Note that the usual vertex splitting operation (see \cite{Whi}) is the special case when $x$ is chosen to be a neighbour of $v_2$. Also, the special case when $v_1$ has degree 3 (and $v_2=v$) is a 1-extension.

\begin{thm}[{\cite[Theorem 2.1]{JNglobal}}]\label{thm:construction}
Every $(2,2)$-circuit $G$ can be obtained from $K_5^-$ or $B_1$ (see Figure \ref{fig:smallgraphs1}) by recursively applying the operations of $K_4^-$-extension and generalised vertex splitting such that each intermediate graph is a $(2,2)$-circuit.
\end{thm}

Later we will show that the operations in the previous theorem preserve global rigidity in any analytic normed plane $X$. It then quickly follows that every $(2,2)$-circuit is globally rigid in $X$. It is tempting to think that combining this with a simple gluing argument suffices to prove that 2-connected and redundantly rigid graphs\footnote{In Section \ref{sec:mcon} we will show that all such graphs are obtained by an `ear decomposition' into $(2,2)$-circuits.} are globally rigid in $X$. We can easily show that gluing globally rigid frameworks, when done correctly, will create a globally rigid of framework.

\begin{prop}\label{prop:union}
	Let $H = (U,F)$ and $H' = (U',F')$ be graphs with $|U\cap U'| \geq d$ and let $X$ be a $d$-dimensional normed space.
	Let $(G,p)$ be a completely regular framework in $X$ where $G=H \cup H'$,
	and define the placements $q:=p|_{U}$, $q':= p|_{U'}$ and $q \cap q' := p|_{U \cap U'}$. 
	Suppose that both $(H,q)$ and $(H',q')$ are globally rigid in $X$ and $q \cap q'$ is isometrically full. 
	Then $(G,p)$ is globally rigid in $X$.
\end{prop}

\begin{proof}
	Let $(G,\tilde{p})$ be a framework in $X$ that is equivalent to $(G,p)$. 
	Since both $(H,q)$ and $(H',q')$ are globally rigid in $X$,
	there exist isometries $g,g'$ where $g \circ q= \tilde{p}|_{U}$ and $g' \circ q'= \tilde{p}|_{U'}$.
	In particular $\tilde{p}_v=(g \circ p)_v = (g' \circ p)_v$ for $v \in U\cap U'$. 
	Since $q \cap q'$ is isometrically full,
	we have that $g=g'$.
	Hence $\tilde{p} =g \circ p$ and $(G,\tilde{p})$ is congruent to $(G,p)$.
\end{proof}

Unfortunately,
this does not show that gluing globally rigid graphs over a sufficient number of vertices will preserve global rigidity.
This is a consequence of the fact that it is currently unknown whether global rigidity is a \emph{generic property}, i.e., a graph is globally rigid in a normed space $X$ if and only if $\grig(G;X)$ is an open and dense set. 
For example, suppose that $H,H'$ are globally rigid graphs in a normed space $X$,
and suppose that we can find isometrically full placements of their intersection.
Even knowing that both $H$ and $H'$ have open sets of globally rigid realisations in $X$ is not sufficient to apply Proposition \ref{prop:union},
as there may not exist globally rigid placements of both $H$ and $H'$ that agree on the vertices in their intersection.

\subsection{Properties of \texorpdfstring{$(2,2)$}{(2,2)}-circuits}

We require a number of technical results about $(2,2)$-circuits from \cite{Nix} and some mild extensions of these results.
It will be convenient to refer to vertices of degree 3 as \emph{nodes}. We say that a node $v$ (and the corresponding 1-reduction) in a $(2,2)$-circuit $G$ is \emph{admissible} if there exists a 1-reduction operation of $G$ at $v$ resulting in another $(2,2)$-circuit. 
A key technical theorem used in the proof of Theorem \ref{thm:recurse} establishes connectivity conditions which are sufficient to guarantee the existence of admissible nodes. That result was strengthened in \cite{JNglobal}, and we state the stronger version here.

\begin{thm}[{\cite[Theorem 2.2]{JNglobal}}]\label{thm:circnodes}
Suppose $G$ is a $(2,2)$-circuit which is distinct from $K_5^-$, $B_1$ and $B_2$. If $G$ has no non-trivial 2-vertex-separation and no non-trivial 3-edge-separation then $G$ has at least two admissible nodes.
\end{thm}

We omit the proof of the next elementary result.

\begin{lem} \label{lem:1-ex(2,2)-circ}
    Let $G'$ be obtained from $G$ by a 1-extension. If $G$ is a $(2,2)$-circuit then so is $G'$.
\end{lem}

Given a graph $G = (V,E)$ and a non-empty subset $X$ of $V$,
$G[X]$ is the subgraph of $G$ induced by $X$ and $i_G(X)$ is the number of edges of $G[X]$.
The set $X$ is said to be \emph{critical in $G$} (or simply \emph{critical} if the context is clear) if $i_G(X) = 2|X|-2$. 
For a non-empty subset $Y$ of $V$ where $X$ and $Y$ are disjoint, we define $d_G(X,Y):=|\{xy \in E: x \in X, y \in Y\}|$. 
When the graph $G$ is clear from the context, we abbreviate the above to $i(X)$ and $d(X,Y)$ respectively.

\begin{lem}[{\cite[Lemma 2.2]{Nix}}]\label{lem:union}
Let $G = (V,E)$ be a $(2,2)$-circuit and let $X,Y \subsetneq V$ be critical such that $|X\cap Y|\geq 1$ and $|X\cup Y|\leq |V|-1$. Then $X\cap Y$ and $X\cup Y$ are both critical, and $d(X \setminus Y, Y \setminus X) = 0$.
\end{lem}

\begin{lem}[{\cite[Lemma 2.3]{Nix}}]\label{lem:connectivity}
Let $G=(V,E)$ be a $(2,2)$-circuit and let $X\subsetneq V$ be critical. Then $G$ is 2-connected and 3-edge-connected, and $G[X]$ is connected.
\end{lem}

Note that the last lemma immediately implies that the minimum degree of a $(2,2)$-circuit $G$, denoted by $\delta(G)$, is at least 3. Since $|E|=2|V|-1$, it follows from the handshaking lemma that $\delta(G) = 3$.

\begin{lem}[{\cite[Lemma 2.4]{Nix}}]\label{lem:nix1}
Let $G=(V,E)$ be a $(2,2)$-circuit and let $X \subsetneq V$ be critical. Then $V \setminus X$ contains a node in $G$.
\end{lem}

\begin{lem}[{\cite[Lemma 2.5]{Nix}}]\label{lem:nix3}
Let $G = (V,E)$ be a $(2,2)$-circuit and let $v$ be a node in $G$ with neighbourhood $\{x,y,z\}$. 
Then the 1-reduction at $v$ adding $xy$ is not admissible if and only if either $xy \in E$, or there is a critical set $X\subsetneq V$ with $x,y\in X$ and $v,z\notin X$.
\end{lem}

For a graph $G=(V,E)$, let $V_3 =\{v\in V: v \mbox{ is a node}\}$. 
Let $V_3^*\subseteq V_3$ be the set of nodes which are not contained in copies of $K_4$ in $G$. Later we will use the elementary fact that if $v\in V_3^*$ and $u$ is a neighbour of $v$ then $u \in V_3$ if and only if $u \in V_3^*$.

We call a node $v$ with $d_{G[V_3^*]}(v) \leq 1$ a \emph{leaf} node and a node with $d_{G[V_3^*]}(v) = 2$ a \emph{series} node. 
Any node $v$ with $d_{G[V_3^*]}(v) = 3$ can never be admissible (and so are of no use to us) since any 1-reduction will reduce the minimum degree.

\begin{lem}\label{lem:forest}
Let $G = (V,E)$ be a $(2,2)$-circuit. Then $G[V_3]$ contains no cycles. 
\end{lem}

\begin{proof}
Suppose $C\subseteq V_3$ is a cycle in $G[V_3]$, and hence induces a cycle in $G$.
As G is not a cycle, $V \setminus C \neq \emptyset$. 
Also, $d(C, V \setminus C) = |C|$ since $C \subseteq V_3$. 
Now
$$i(V \setminus C)=2|V|-1-i(C)-d(C,V \setminus C)=2|V|-1-2|C|=2(|V|-|C|)-1=2|V \setminus C|-1, $$
a contradiction. Hence $G[V_3]$ contains no cycles.
\end{proof}

For a node $v$ in a $(2,2)$-circuit $G=(V,E)$ with $N(v) =\{x,y,z\}$, we say that a critical set $X$ in $G$ is \emph{$v$-critical} if $x,y \in X$ and $v,z \notin X$.
If $X$ is a $v$-critical set containing $\{x,y\}$ for some node $v$ with neighbourhood $\{x,y,z\}$ and $d_G(z) \geq 4$, then $X$ is \emph{node-critical}.
It follows from a quick counting argument that if $X$ is a $v$-critical set in $G$ containing $\{x,y\}$,
then $xy \notin E$.

The next lemma is a strengthening of \cite[Lemma 2.10]{Nix}.

\begin{lem} \label{lem:nix4}
    Let $G = (V,E)$ be a $(2,2)$-circuit containing a node $v$ with neighbourhood $\{x, y, z\}$. 
    Suppose $d_{G}(z) \geq 4$ and there exists a $v$-critical set $X$ in $G$ containing $x, y$. 
    Further suppose that there exists a non-admissible node $u \in V \setminus (X + v)$ such that either:
    \begin{enumerate}[(i)]
        \item \label{lem:nix4.1} $u$ is a series node with precisely one neighbour $w$ in $X$, and $w$ is a node, or
        \item \label{lem:nix4.2} $u$ is a leaf node with no edges between its neighbours.
    \end{enumerate}
    Then there is a node-critical set $X'$ in $G$ such that $X \subsetneq X'$.
\end{lem}

\begin{proof}
Suppose that (\ref{lem:nix4.1}) holds and let $N_{G}(u) = \{w, a, b\}$. 
As $u$ is a series node we may suppose without loss of generality that $d_{G}(a) = 3$ and $d_{G}(b) \geq 4$. 
Lemma \ref{lem:forest} implies that $G[V_3]$ contains no cycles, so $wa \notin E$.
As $u$ is a non-admissible node and $wa \notin E$, Lemma \ref{lem:nix3} implies there exists a $u$-critical set $Y$ in $G$ containing $\{w, a\}$. 
As $w \in X \cap Y$ and $u, b \notin X \cup Y$, Lemma \ref{lem:union} implies that $X \cup Y$ is a $u$-critical set in $G$
containing $\{w, a\}$. 
Moreover, $X \cup Y$ is node-critical as $d_G(b) \geq 4$.
Finally, as $a \in Y \setminus X$, we have $X \subsetneq X \cup Y$.

Now suppose that (\ref{lem:nix4.2}) holds and let $N_{G}(u) = \{a, b, c\}$. 
As $v \notin X + u$ we must have $|N_{G}(u) \cap X| \leq 2$, since
otherwise $G[X +u]$ is a proper subgraph in $G$ with $2|X +u|-1$ edges. 
If $|N_{G}(u) \cap X| = 2$ then $X + u$ is a node-critical $v$-critical set in $G$ properly containing $X$ and we are done.
So we may suppose that $|N_{G}(u) \cap X| \leq 1$.
As $u$ is a leaf node we may suppose without loss of generality that $d_{G}(a), d_{G}(c) \geq 4$. 
As $u$ is non-admissible and there are no edges between the neighbours of $u$, Lemma \ref{lem:nix3} implies there exist $u$-critical sets, $Y_{1}$ and $Y_{2}$, in $G$ on $\{a, b\}$ and $\{b,c\}$ respectively. As $b \in Y_{1} \cap Y_{2}$ and $u \notin Y_{1} \cup Y_{2}$, Lemma \ref{lem:union} implies $Y_{1} \cup Y_{2}$ is a critical set. Moreover, as $N_{G}(u) \subseteq Y_{1} \cup Y_{2}$, it follows that $Y_{1} \cup Y_{2} = V -u$. Hence $Y_{1} \cap X \neq \emptyset$ or $Y_{2} \cap X \neq \emptyset$. We may suppose, without loss of generality, that $Y_{1} \cap X \neq \emptyset$.

As $Y_{1} \cap X \neq \emptyset$ and $u \notin Y_{1} \cup X$, Lemma \ref{lem:union} implies $Y_{1} \cup X$ is critical in $G$. Moreover, as $a, b \in Y_{1} \cup X$ and $|N_{G}(u) \cap X| \leq 1$ we have that $X \subsetneq Y_{1} \cup X$. If $c \notin Y_{1} \cup X$ then $Y_{1} \cup X$ is node-critical in $G$ and we are done. On the other hand, if $c \in Y_{1} \cup X$ then $c \in X$. So $Y_{2} \cap X \neq \emptyset$ and $N_{G}(u) \cap X = \{c\}$. As $u, a \notin Y_{2} \cup X$, the set $Y_{2} \cup X$ is a node-critical $u$-critical set in $G$ by Lemma \ref{lem:union}. Finally, as $b \in Y_{2} \setminus X$, we have $X \subsetneq Y_{2} \cup X$.
\end{proof}

The final result in this section strengthens \cite[Lemma 2.8]{Nix}.
Similarly to 1-reductions,
we define a $K_4^-$-reduction of a $(2,2)$-circuit to be {\em admissible} if the resulting graph is also a $(2,2)$-circuit.

\begin{lem}\label{lem:nix2}
    Let $G = (V,E)$ be a $(2,2)$-circuit containing a node $v$ with neighbourhood $\{x, y, z\}$. If $yz \notin E$ and $xz \in E$, then either:
    \begin{enumerate}[(i)]
        \item \label{lem:nix2.1} $xy \in E$, $\{y, z\}$ is not a vertex cut of $G$, and the $1$-reduction of $G$ at $v$ adding $yz$ is admissible,
        \item \label{lem:nix2.2} $xy \in E$, $\{y, z\}$ is a vertex cut of $G$,  and either the $1$-reduction of $G$ at $v$ adding $yz$ is admissible or the $K_4^-$-reduction of $G$ deleting $v$ and $x$ and adding $yz$ is admissible, or
        \item \label{lem:nix2.3} $xy \notin E$ and there is an admissible $1$-reduction of $G$ at $v$.
    \end{enumerate}
\end{lem}

\begin{proof}
First suppose that $xy \in E$. We will prove (\ref{lem:nix2.1}) and (\ref{lem:nix2.2}) simultaneously.
If there does not exist a $v$-critical set $X$ containing $\{y, z\}$ in $G$,
then the $1$-reduction of $G$ at $v$ adding $yz$ is admissible by Lemma \ref{lem:nix3}.
So suppose there exists a $v$-critical set $X$ containing $\{y, z\}$ in $G$. The induced subgraph on the set $X \cup \{v,x\}$ has at least $2|X \cup \{v,x\}| - 1$ edges,
it follows that $V = X \cup \{v,x\}$ and $N_G(x) = \{ v,y,z\}$.
This implies $\{y, z\}$ is a vertex cut of $G$.
Define $H_1,H_2$ to be the induced subgraphs of $G$ on the vertex sets $N_G(v)+v$ and $V \setminus \{v, x\}$ respectively.
Let $(G_1, G_2)$ be the 1-separation of $G$ on $(H_1,H_2)$ and let $(G'_1, G'_2)$ be the 1-separation of $G$ on $(H_2,H_1)$. 
As $G_1 \cong K_4$, Lemma \ref{lem:circuit_dec3} implies that $G'_1$ is a $(2,2)$-circuit.
The $K_4^-$-reduction of $G$ that deletes $v$ and $x$ and adds $yz$ gives $G'_{1}$ and hence is admissible.

The final case, i.e. when $xy \notin E$, was proved in \cite{Nix}.
\end{proof}

In the next two sections we will show that the characterisation in Theorem \ref{thm:construction} essentially remains true for arbitrary 2-connected and redundantly rigid graphs in analytic normed planes. 
This is a non-trivial extension, since redundantly rigid graphs in analytic normed planes need not contain spanning $(2,2)$-circuits, and we will need some additional combinatorial concepts and results to prove it.
A simple example showing that redundantly rigid graphs in analytic normed planes may not contain spanning $(2,2)$-circuits is the complete bipartite graph $K_{3,6}$. Note that $K_{3,6}$ has twice as many edges as it has vertices and so is not a $(2,2)$-circuit. Since every vertex in the part of size 6 has degree 3 and every edge is incident to some such vertex, $K_{3,6}$ does not contain a spanning $(2,2)$-circuit. Nevertheless it is redundantly rigid in any analytic normed plane by Theorem \ref{thm:rigidne}.

\section{\texorpdfstring{$\mathcal{M}(2,2)$}{M(2,2)}-connected graphs} \label{sec:mcon}

Recall that a matroid $M=(E,r)$ with ground set $E$ and rank function $r$ is \emph{connected} if every pair $e,f \in E$ is contained in a common circuit.
Equivalently, we may define a relation on $E$ by saying that $e, f \in E$ are
related if $e = f$ or there is a circuit $C$ in $M$ with $e, f \in C$. It is well-known that this
is an equivalence relation. The equivalence classes are called the components of $M$. 
If $M$ has at least two elements, then $M$ is connected if and only if there is only one component of $M$.
A graph $G = (V,E)$ is {\em $\mathcal{M}(2,2)$-connected} if $G$ has no isolated vertices, $|E| \geq 2$ and every pair of elements of $E$ belongs to a common $(2,2)$-circuit in $G$.

It is easy to see that any $\mathcal{M}(2,2)$-connected graph contains a spanning $(2,2)$-tight subgraph and hence is rigid in any non-Euclidean normed plane (Theorem \ref{thm:rigide}). In fact more is true.

\begin{lem}\label{lem:Mcon}
    Let $(G,p)$ be a completely regular framework in a non-Euclidean normed plane $X$. 
    Then $G$ is $\mathcal{M}(2,2)$-connected if and only if $G$ is 2-connected and $(G,p)$ is redundantly rigid.
\end{lem}

\begin{proof}
    Since $p$ is a completely regular placement,
    $(G,p)$ is redundantly rigid if and only if $G-e$ is rigid in $X$ for every edge $e \in E$.
    When combined with Theorem \ref{thm:rigidne},
    it follows that $(G,p)$ is redundantly rigid if and only if $G$ contains a $(2,2)$-tight spanning subgraph and every edge is contained in a $(2,2)$-circuit.
    We now apply the following two equivalences:
    (i) $G$ is redundantly rigid on the cylinder if and only $G$ contains a $(2,2)$-tight spanning subgraph and every edge is contained in a $(2,2)$-circuit (see \cite[Theorem 5.4]{NOP}),
    and (ii) $G$ is redundantly rigid on the cylinder and 2-connected if and only if $G$ is $\mathcal{M}(2,2)$-connected (see \cite[Theorem 5.4]{Nix}).
\end{proof}

\subsection{Graph operations preserving \texorpdfstring{$\mathcal{M}(2,2)$}{M(2,2)}-connectivity}
We begin with two standard graph operations; 1-extensions and edge additions.

\begin{lem}\label{lem:add}
Let $G'=(V',E')$ be an $\mathcal{M}(2,2)$-connected graph and suppose that $G'$ is obtained from $G=(V,E)$ by a 1-extension or an edge addition. Then $G'$ is $\mathcal{M}(2,2)$-connected.
\end{lem}

\begin{proof}
First suppose that $G'$ is obtained from $G$ by adding a new edge $e$. By Theorem \ref{thm:rigidne} and Lemma \ref{lem:Mcon}, $G$ is spanned by a $(2,2)$-tight subgraph. Hence there exists a $(2,2)$-circuit $C$ in $G'$ with $e\in C$. The $\mathcal{M}(2,2)$-connectivity of $G'$ now follows from the transitivity of the relation that defines $\mathcal{M}(2,2)$-connectivity.

Next suppose that $G'$ is obtained from $G$ by a 1-extension operation that deletes an edge $e$ and adds a new vertex $v$ with three incident edges $e_1, e_2, e_3$. For $1 \leq i \leq 3$, let $e_i = vu_i$ and suppose that $e = u_{1}u_{2}$. 
Since $G$ is $\mathcal{M}(2,2)$-connected, $\delta(G)\geq  3$ and hence $d_{G}(u_3) \geq 3$. Thus we can choose $f \in E \cap E'$ such that $f$ is incident to $u_3$. 
Now, take $f' \in E'-f$. If $f' \in E$ then there exists a $(2,2)$-circuit $C$ in $G$ containing $f$ and $f'$. If $e$ is not contained in $C$ then $C$ is also in $G'$, so there exists a $(2,2)$-circuit $C' = C$ in $G'$ containing $f$ and $f'$. If $e \in C$ then, by Lemma \ref{lem:1-ex(2,2)-circ}, there exists a $(2,2)$-circuit $C'$ in $G'$ containing $f$ and $f'$. So, for all $f' \in (E \cap E')-f$ there exists a $(2,2)$-circuit $C'$ in $G'$ containing $f$ and $f'$. 
If $f' \notin E$ then $f' \in \{e_1, e_2, e_3\}$. As there exists a $(2,2)$-circuit $C$ in $G$ containing $e$ and $f$, Lemma \ref{lem:1-ex(2,2)-circ} implies there exists a $(2,2)$-circuit $C'$ in $G'$ containing $f, e_1, e_2$, and $e_3$. Therefore, for all $f' \in E'-f$ there exists a $(2,2)$-circuit $C'$ in $G'$ containing $f$ and $f'$. Hence, by the transitivity of the relation that defines $\mathcal{M}(2,2)$-connectivity, $G'$ is $\mathcal{M}(2,2)$-connected.
\end{proof}

We next extend Lemmas \ref{lem:circuit_dec}, \ref{lem:circuit_dec2} and \ref{lem:circuit_dec3} to $\mathcal{M}(2,2)$-connected graphs.
To prove these extensions, we use two technical results.
For a graph $G=(V,E)$ and a non-empty subset $F \subseteq E$,
define $V[F] = \{ v \in V : vw \in F \}$ and $G[F] = (V[F],F)$. 

\begin{lem} \label{lem:CircuitOver2VerSep}
Let $G=(V,E)$ be a graph with subgraphs $H_1=(U_1,F_1)$ and $H_2=(U_2,F_2)$.
Suppose that $(H_1,H_2)$ is a 2-vertex-separation of $G$ and $F_1 \cap F_2 = \{f\}$.
If $G$ is $\mathcal{M}(2,2)$-connected then 
for all $e_{1} \in F_1 \setminus F_2$ and all $e_{2} \in F_2 \setminus F_1$, there exists a $(2,2)$-circuit $C \subseteq E$ containing $\{f,e_{1}, e_{2}\}$.
\end{lem}

\begin{proof}
Fix the edges $e_{1} \in F_1 \setminus F_2$, $e_{2} \in F_2 \setminus F_1$,
and choose a $(2,2)$-circuit $C$ such that $e_{1}, e_{2} \in C$.
As $G[C]$ is $2$-connected by Lemma \ref{lem:connectivity}, $U_1 \cap U_2 \subseteq V[C]$.
If $f \in C$ then we are done, so suppose that $f \notin C$.
Let $G' = G[C +f]$. 
Lemma \ref{lem:add} implies $G'$ is $\mathcal{M}(2,2)$-connected, so for each $i \in \{1, 2\}$ we fix a $(2,2)$-circuit $C_{i} \subseteq C + f$ such that $\{f, e_{i}\} \subseteq C_{i}$. If $e_1 \in C_2$ or $e_2 \in C_1$ then we are done, so we may suppose that $e_1 \notin C_2$ and $e_2 \notin C_1$. 

Take a $(2,2)$-circuit $C'_{1} \subseteq C + f$ such that $f \in C'_{1}$.
If $C_1'\neq C_1$ then the circuit exchange axiom implies there exists a $(2,2)$-circuit $D \subseteq (C_1 \cup C'_1)-f$. 
If $e_{2} \notin C'_{1}$ then $D \subseteq (C_1 \cup C'_1)-f \subsetneq C$, contradicting the fact that no proper subset of a circuit is a circuit. 
Hence $e_2\in C_1'$. 
Since $C_1'$ was arbitrary, $C_{1}$ is the unique circuit in $C + f$ containing $f$ but not $e_{2}$. 
Similarly $C_2$ is the unique circuit in $C + f$ containing $f$ but not $e_{1}$.

As $f \in C_{1} \cap C_{2}$ and $C_{1} \neq C_{2}$, there exists a $(2,2)$-circuit $C'$ such that $C' \subseteq (C_{1} \cup C_{2}) - f \subseteq C$. Hence $C' = (C_{1} \cup C_{2}) - f = C$ and $|C_{1} \cup C_{2}| = |C|+1 = 2|V[C]|$. Now, if $C_{1} \cap C_{2} = \{f\}$ then
\begin{align*}
2|V[C]| &= |C_{1} \cup C_{2}| \\
&= |C_{1}|+|C_{2}|-|C_{1} \cap C_{2}| \\
&= (2|V[C_1]|-1)+(2|V[C_2]|-1)-1 \\
&= 2(|V[C_1]|+|V[C_2]|+|V[C]|-|V[C]|)-3 \\
&= 2|V[C]|+2(|V[C_1]|+|V[C_2]|-|V[C]|)-3,
\end{align*}
a contradiction.
Hence there exists an edge $e \in (C_{1} \cap C_{2}) - f$.
Note that, as both $C_1 - f$ and $C_2 - f$ are contained in $C$, 
we have $e \in C$.

The circuit exchange axiom implies there exists a $(2,2)$-circuit $C'' \subseteq C + f$ such that $C'' \subseteq (C \cup C_{1})- e$.
As $e \notin C''$ we have $C'' \notin \{C, C_{1}, C_{2}\}$, and as $C'' \neq C$ and no proper subset of a circuit is a circuit we have $f \in C''$.
By the uniqueness of $C_{1}$ and $C_{2}$ it follows that $f, e_{1}, e_{2} \in C''$. 
Hence $C''$ is a $(2,2)$-circuit contained in $E$ such that $\{f,e_{1}, e_{2}\} \subseteq C''$.
\end{proof}

\begin{lem} \label{lem:circK_4}
Let $G = (V,E)$ be a graph, and suppose $H \cong K_{4}$ is a subgraph of $G$ containing exactly two vertices of degree 3 in $G$.
If $G$ is $\mathcal{M}(2,2)$-connected,
then for all $e \in E$ there exists a $(2,2)$-circuit $C \subseteq E$ such that $E(H) \cup \{e\} \subseteq C$.
\end{lem}

\begin{proof}
Since any $(2,2)$-circuit has minimum degree 3, this follows from Lemma \ref{lem:CircuitOver2VerSep}.
\end{proof}

\begin{lem}\label{lem:connected_i-join}
Let $G_{1}=(V_1,E_1)$ and $G_{2}=(V_2,E_2)$ be graphs and suppose that $G= (V,E)$ is the $j$-join of $(G_{1}, G_{2})$ for some $j \in \{1, 2, 3\}$. 
If $G_{1}$ and $G_{2}$ are $\mathcal{M}(2,2)$-connected then $G$ is $\mathcal{M}(2,2)$-connected.
\end{lem}

\begin{proof}
By the transitivity of the equivalence relation that defines $\mathcal{M}(2,2)$-connectivity, it suffices to show that there exists $e \in E$ such that for all $f \in E-e$ there exists a $(2,2)$-circuit $C \subseteq E$ such that $e, f \in C$.
We consider each of the three types of $j$-join in turn.

Firstly suppose $G$ is the 1-join of $(G_{1}, G_{2})$.
Let $a,b$ be the unique vertices shared by $G_1,G_2$ and let $c,d$ be the vertices of $G_2$ deleted by the 1-join.
Fix $H_2 = (U_2,F_2)$ to be the complete graph with $U_2 = \{a,b,c,d\}$,
and fix an edge $e \in E \cap E_1$.
Choose any edge $f \in E-e$.
Now, either $f \in E_1$ or $f \in E_2$.
If $f \in E_1$ then there exists a $(2,2)$-circuit $C_1 \subseteq E_{1}$ such that $e,f \in C_1$.
If $ab \notin C_1$ then we are done.
If $ab \in C_1$ then Lemma \ref{lem:circK_4} implies there exists a $(2,2)$-circuit $C_{2} \subseteq E_2$ such that $F_2 \subseteq C_2$.
The 1-join of $(G_{1}[C_{1}],G_{2}[C_{2}])$,
denoted by $G^*=(V^*,E^*)$, is a $(2,2)$-circuit by Lemma \ref{lem:circuit_dec}.
Moreover, $G^{*}$ is a subgraph of $G$ and $e, f \in E^*$.
If $f \in E_2$ then there exists a $(2,2)$-circuit $C_{3} \subseteq E_{1}$ such that $e,ab \in C_3$.
Lemma \ref{lem:circK_4} implies there exists a $(2,2)$-circuit $C_{4} \subseteq E_2$ such that $F_{2} +f \subseteq C_{4}$. 
The 1-join of $(G_{1}[C_{3}],G_{2}[C_{4}])$,
denoted by $G'=(V',E')$, is a $(2,2)$-circuit by Lemma \ref{lem:circuit_dec}.
Moreover, $G'$ is a subgraph of $G$ and $e, f \in E'$.

Next suppose $G$ is the 2-join of $(G_{1}, G_{2})$.
Let $a,b$ be the unique vertices shared by $G_1,G_2$, and for each $i \in \{1,2\}$ let $c_i,d_i$ be the vertices of $G_i$ deleted by the 2-join.
Fix $H_1 = (U_1,F_1)$ to be the complete graph with $U_1 =\{a,b,c_1,d_1\}$ and $H_2 = (U_2,F_2)$ to be the complete graph with $U_2 = \{a,b,c_2,d_2\}$.
Let $e = ab$ and choose any edge $f \in E$.
By relabelling if necessary, we may assume $f \in E_1$.
As $G_{1}$ and $G_{2}$ are $\mathcal{M}(2,2)$-connected, Lemma \ref{lem:circK_4} implies that there exist $(2,2)$-circuits $C_1 \subseteq E_1$ and $C_2 \subseteq E_2$ such that $F_1 +f \subseteq C_1$ and $F_2 \subsetneq C_2$.
The 2-join of $(G_{1}[C_{1}],G_{2}[C_{2}])$,
denoted by $G^*=(V^*,E^*)$, is a $(2,2)$-circuit by Lemma \ref{lem:circuit_dec}.
Moreover, $G^{*}$ is a subgraph of $G$ and $e, f \in E^*$.

Finally, suppose $G$ is the 3-join of $(G_{1}, G_{2})$.
For each $i \in \{1,2\}$ let $v_i$ be the node in $G_i$ that is deleted by the 3-join operation, and let $a_i,b_i,c_i$ be the neighbours of $v_i$ such that $a_1a_2,b_1b_2,c_1c_2$ are the edges added.
Let $e = a_{1}a_{2}$ and for $i \in \{1, 2\}$ take any $f_i \in E \cap E_i$.
As $G_{1}$ and $G_{2}$ are $\mathcal{M}(2,2)$-connected, for $i \in \{1, 2\}$
there exists a $(2,2)$-circuit $C_i \subseteq E_i$ such that $a_iv_i, f_i \in C_i$.
Since $\delta(G_i[C_i]) = 3$ for $i \in \{1, 2\}$, we have that $F:=\{a_iv_i,b_iv_i,c_iv_i\} \subseteq C_{i}$ for each $i \in \{1, 2\}$.
The 3-join of $(G_{1}[C_{1}],G_{2}[C_{2}])$,
denoted by $G^*$, is a $(2,2)$-circuit by Lemma \ref{lem:circuit_dec}. 
Moreover, $G^{*}$ is a subgraph of $G$ and $F \cup \{f_1, f_2\} \subseteq E^*$. As $E = (E \cap E_1) \cup (E \cap E_2) \cup F$ we have shown that for all $f \in E-e$ there exists a $(2,2)$-circuit $C \subseteq E$ such that $e, f \in C$.
\end{proof}

\begin{lem}\label{lem:connected_2/3-sep}
Let $(G_{1}, G_{2})$ be a $j$-separation of a graph $G = (V,E)$ for some $j \in \{2, 3\}$.
If $G$ is $\mathcal{M}(2,2)$-connected then $G_{1}=(V_1,E_1)$ and $G_{2}=(V_2,E_2)$ are $\mathcal{M}(2,2)$-connected.
\end{lem}

\begin{proof}
First suppose $(G_{1}, G_{2})$ is a 2-separation of $G$ with respect to the 2-vertex-separation $(H_{1}, H_{2})$,
where $H_i = (U_i,F_i)$ for each $i \in \{1,2\}$.
By definition, for each $i \in \{1,2\}$ we have that $G_{i} = H_{i} \cup K[\{a, b, c_{i}, d_{i}\}]$ for some vertices $c_{i}, d_{i} \notin V$.
For each $i \in \{1, 2\}$ choose an edge $f_i \in F_i-ab$. As $G$ is $\mathcal{M}(2,2)$-connected, Lemma \ref{lem:CircuitOver2VerSep} implies there exists a $(2,2)$-circuit $C \subseteq E$ such that $ab, f_1, f_2 \in C$. For $i \in \{1, 2\}$ let $H'_i = H_i[C \cap F_i]$.
Then $(H'_1, H'_2)$ is a 2-vertex-separation of $G[C]$ and $ab \in E(H'_1) \cap E(H'_2)$.
Let $(G'_1, G'_2)$ be the 2-separation of $G[C]$ with respect to $(H'_1, H'_2)$. 
Then $E(K[\{a, b, c_i, d_i\}]) +f_i \subseteq E(G'_i)$ for each $i \in \{1,2\}$. Lemma \ref{lem:circuit_dec2} implies that $G'_i$ is a $(2, 2)$-circuit for each $i \in \{1, 2\}$.
Hence the following holds for each $i \in \{1,2\}$: for any edge $e_i \in E_i-ab$,
there exists a $(2,2)$-circuit $C_i \subseteq E_i$ such that $e_i, ab \in C_i$. By the transitivity of the equivalence relation that define $\mathcal{M}(2,2)$-connectivity it follows that both $G_1$ and $G_2$ are $\mathcal{M}(2,2)$-connected.

Now suppose $(G_{1}, G_{2})$ is a 3-separation of $G$ with respect to the non-trivial 3-edge-separation $(H_{1}, H_{2})$,
where $H_i = (U_i,F_i)$ for each $i \in \{1,2\}$.
By definition, for each $i \in \{1,2\}$ we have that $G_{i} = H_{i} \cup K[\{v_i\},\{a_i, b_i, c_{i}\}]$ for some vertices $a_i,b_i, c_{i} \in V$, $v_i \notin V$.
For each $i \in \{1, 2\}$ choose an edge $f_i \in F_i$. As $G$ is $\mathcal{M}(2,2)$-connected, there exists a $(2,2)$-circuit $C \subseteq E$ such that $f_1, f_2 \in C$. By Lemma \ref{lem:connectivity} we have that $a_1a_2, b_1b_2, c_1c_2 \in C$. For each $i \in \{1, 2\}$ let $H'_i = H_i[C \cap F_i]$.
Then $(H'_1, H'_2)$ is a non-trivial 3-edge-separation of $G[C]$ with corresponding edge cut $\{a_1a_2, b_1b_2, c_1c_2\}$.
Let $(G'_1, G'_2)$ be the 3-separation of $G[C]$ with respect to $(H'_1, H'_2)$.
Then $\{f_i, a_iv_i, b_iv_i, c_iv_i\} \subset E(G'_i)$ for each $i \in \{1,2\}$. 
Lemma \ref{lem:circuit_dec2} implies that $G'_i$ is a $(2, 2)$-circuit for each $i \in \{1, 2\}$.
Hence the following holds for each $i \in \{1,2\}$: for any edge $e_i \in E_i-a_iv_i$,
there exists a $(2,2)$-circuit $C_i \subseteq E_i$ such that $e_i, a_iv_i \in C_i$.
By the transitivity of the equivalence relation that defines $\mathcal{M}(2,2)$-connectivity it follows that $G_1$ and $G_2$ are $\mathcal{M}(2,2)$-connected.
\end{proof}

\begin{lem}\label{lem:connected_1-sep}
Let $G=(V,E)$ be a graph with a 2-vertex-separation $(H_{1}, H_{2})$. 
Suppose that $(G_{1}, G_{2})$ is a 1-separation of $G$ with respect to $(H_{1}, H_{2})$, and $(G'_2, G'_1)$ is a 1-separation of $G$ with respect to $(H_{2}, H_{1})$. 
If $G$ is $\mathcal{M}(2,2)$-connected then
\begin{enumerate}[(i)]
\item \label{lem:conn_1-sep.1} $G'_1$ and $G_2$ are $\mathcal{M}(2,2)$-connected, and
\item \label{lem:conn_1-sep.2} $G_1$ is $\mathcal{M}(2,2)$-connected or $G'_2$ is $\mathcal{M}(2,2)$-connected. 
\end{enumerate}
\end{lem}

\begin{proof}
Let $H_i=(U_i,F_i)$ for each $i\in \{1,2\}$,
and let $a,b$ be the unique vertices contained in both $U_1$ and $U_2$.
By the definition of 1-separation, $ab \notin E$.
By Lemma \ref{lem:add}, $G + ab$ is $\mathcal{M}(2,2)$-connected. 
Now we may take the 2-separation of $G$ with respect to the 2-vertex-separation $(H_{1}+ab, H_{2}+ab)$, which gives the ordered pair $(G'_1, G_{2})$. Theorem \ref{lem:connected_2/3-sep} gives us that $G'_1$ and $G_{2}$ are $\mathcal{M}(2,2)$-connected. 

Now let us suppose, in pursuit of a contradiction, that neither $G_{1}$ nor $G'_2$ are $\mathcal{M}(2,2)$-connected.
By the transitivity of the equivalence relation that defines $\mathcal{M}(2,2)$-connectivity, there exists $e_{1} \in E_1 \cap E$ such that for each $(2,2)$-circuit $C \subseteq E_1$, if $ab \in C$ then $e_{1} \notin C$. Similarly, there exists $e_2 \in E'_2 \cap E$ such that for all $(2,2)$-circuits $C \subseteq E'_2$, if $ab \in C$ then $e_2 \notin C$. 
As $G$ is $\mathcal{M}(2,2)$-connected, there exists a $(2,2)$-circuit $C^* \subseteq E$ such that $e_{1}, e_2 \in C^*$. 
As $e_{1} \in E_1 - ab$ and $e_2 \in E'_2 - ab$, 
it follows that $e_{1} \in F_1 \cap C^*$ and $e_2 \in F_2 \cap C^*$.
Fix $G^*= G[C^*]$.
As $G^*$ is $2$-connected (Lemma \ref{lem:connectivity}) and contains vertices in both $U_1\sm\{a,b\}$ and $U_2\sm\{a,b\}$, we have that $\{a,b\}$ is a vertex cut of $G^*$ and $(H_{1} \cap G^*, H_{2} \cap G^*)$ is the corresponding 2-vertex-separation.
Hence we may take the 1-separations of $G^*$ with respect to both $(H_{1} \cap G^*, H_{2} \cap G^*)$ and $(H_{2} \cap G^*, H_{1} \cap G^*)$, and we denote the resulting ordered pairs of graphs by $(G^*_1, G^*_2)$ and $(G^{*\prime}_2, G^{*\prime}_1)$ respectively.

By our choice of $G^*$, we have that $G^*_1$ is a subgraph of $G_1$ containing $e_{1}, ab$ and $G^{*\prime}_2$ is a subgraph of $G_2'$ containing $e_2, ab$.
However, Lemma \ref{lem:circuit_dec3} implies that one of $G^*_1$, $G^{*\prime}_2$ must be a $(2,2)$-circuit which gives a contradiction.
Hence $G_{1}$ is $\mathcal{M}(2,2)$-connected or $G'_2$ is $\mathcal{M}(2,2)$-connected and we are done.
\end{proof}

We next consider the effect of generalised vertex splits and $K_4^-$-extensions on $\mathcal{M}(2,2)$-connected graphs.
If $G'$ is obtained from a $\mathcal{M}(2,2)$-connected graph $G$ by a generalised vertex split, then $G'$ may not be $\mathcal{M}(2,2)$-connected. (This can fail in many ways, for instance the new graph could have a vertex of degree less than 3.)
Later when analysing the reduction step of our recursive construction, we will need to take care to only apply edge-reductions such that the starting graph is the result of a generalised vertex split on the reduced graph that does preserve $\mathcal{M}(2,2)$-connectivity.
However, we next deduce from applying Lemmas \ref{lem:connected_i-join} and \ref{lem:connected_1-sep} that this complication does not arise for $K_4^-$-extensions and $K_4^-$-reductions. 

\begin{lem}\label{l:k4-moves}
    Let $G=(V,E)$ be a $\mathcal{M}(2,2)$-connected graph.
    Then any $K_4^-$-extension of $G$ is also $\mathcal{M}(2,2)$-connected. Conversely, any $K_4^-$-reduction of $G$ that adds an edge $e \notin E$ is $\mathcal{M}(2,2)$-connected.
\end{lem}

\begin{proof}
    A $K_4^-$-extension is a special case of Lemma \ref{lem:connected_i-join} with $G_1 = G$ and $G_2 \cong B_1$.
    Suppose that $G'$ is formed from $G$ by the $K_4^-$-reduction that adds the edge $e$.
    The possible 1-separations of $G$ are either $(G',H_1)$ or $(H_2,G+e)$, 
    where $H_1 \cong B_1$ and $H_2 \cong K_4$.
    Since $K_4$ is not $\mathcal{M}(2,2)$-connected,
    we have that $G'$ is $\mathcal{M}(2,2)$-connected by Lemma \ref{lem:connected_1-sep}.
\end{proof}

\subsection{Ear decompositions}

Given a non-empty sequence of circuits $C_1,\dots,C_t$ in a matroid $M=(E,r)$, we define the sets $D_i=C_1\cup \dots \cup C_i$ and $\tilde C_i=C_i \sm D_{i-1}$ for each $1\leq i \leq m$.
The sequence $C_1,\dots,C_m$ is a \emph{partial ear decomposition} of $M$ if, for all $2\leq i \leq m$,
\begin{enumerate}
\item[(E1)] $C_i\cap D_{i-1}\neq \emptyset$,
\item[(E2)] $C_i \sm D_{i-1} \neq \emptyset$, and
\item[(E3)] no circuit $C_i'$ satisfying (E1) and (E2) has $C_i'\sm D_{i-1}\subsetneq C_i \sm D_{i-1}$. 
\end{enumerate}
A partial ear decomposition $C_1,\dots,C_t$ is an \emph{ear decomposition} of $M$ if $D_t=E$.
The following standard result \cite{C&H} shows the close relationship between matroid connectivity and ear decompositions.

\begin{lem}\label{lem:ear}
Let $M=(E,r)$ be a matroid with $|E|\geq 2$. Then:
\begin{enumerate}[(i)]
\item \label{lem:ear.1} $M$ is connected if and only if it has an ear decompostion.
\item \label{lem:ear.2} If $M$ is connected then every partial ear decomposition is extendable to an ear decomposition of $M$.
\item \label{lem:ear.3} If $C_1,C_2,\dots,C_t$ is an ear decomposition of $M$ then $r(D_i)-r(D_{i-1})=|\tilde C_i|-1$ for all $2\leq i \leq t$.
\end{enumerate}
\end{lem}

We say that a graph $G$ with no isolated vertices has a \emph{(partial) ear decomposition} if there is a (partial) ear decomposition of the matroid $\mathcal{M}(2,2)$ restricted to the edge set of $G$. The lack of explicit reference to the matroid in the terminology should not cause confusion, as the only matroid we consider is $\mathcal{M}(2,2)$.

\begin{example}
    Recall that $K_{3,6}$ is redundantly rigid in any analytic normed plane but does not contain a spanning $(2,2)$-circuit. Let the part of size 6 be denoted $\{v_1,v_2,\dots, v_6\}$ and the part of size 3 be denoted $\{u_1,u_2,u_3\}$. An ear decomposition for $K_{3,6}$ is given by $(C_1, C_2)$, where $C_1$ is the edge set of the $(2,2)$-circuit $K_{3,6}-v_1 \cong K_{3,5}$ and $C_2$ is the the edge set of $K_{3,6}-v_2$. Then $\tilde C_2=\{v_1u_1,v_1u_2,v_1u_3\}$ and $D_2=E(K_{3,6})$.
\end{example}

\section{A recursive construction of 2-connected, redundantly rigid graphs} \label{sec:comb}

The purpose of this section is to derive the following recursive construction of 2-connected graphs that are redundantly rigid in analytic normed planes. 
This, combined with the geometric results of Section \ref{sec:proof}, will be used to prove 
our characterisation of global rigidity in analytic normed planes (Theorem \ref{thm:grne}).

\begin{thm}\label{thm:1}
A graph $G$ is 2-connected and redundantly rigid in an analytic normed plane $X$ if and only if $G$ can be generated from $K_5^-$ or $B_1$ (see Figure \ref{fig:smallgraphs1}) by $K_4^-$-extensions, edge additions and generalised vertex splits such that each intermediate graph is 2-connected and redundantly rigid in $X$.
\end{thm}

We illustrate the theorem by describing a construction sequence for a specific example in Figure \ref{fig:sequence}.

  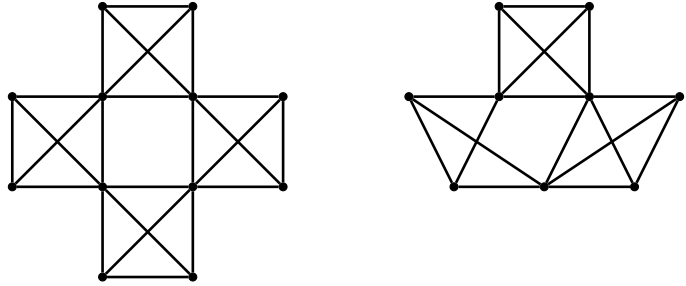
\begin{figure}[htp]
\begin{center}
    \begin{tikzpicture}[scale=0.6]
    \node[vertex] (a) at (-2,-1) {};
   \node[vertex] (b) at (0,-1) {};
   \node[vertex] (c) at (-2,1) {};
   \node[vertex] (d) at (0,1) {};
   
   \node[vertex] (e) at (2,-1) {};
   \node[vertex] (f) at (4,-1) {};
   \node[vertex] (g) at (2,1) {};
   \node[vertex] (h) at (4,1) {};
   
   \node[vertex] (i) at (0,3) {};
   \node[vertex] (j) at (2,3) {};
   
   \node[vertex] (k) at (0,-3) {};
   \node[vertex] (l) at (2,-3) {};

	\draw[edge] (a)edge(b);
\draw[edge] (a)edge(c);
\draw[edge] (a)edge(d);
\draw[edge] (b)edge(c);
\draw[edge] (b)edge(d);
\draw[edge] (c)edge(d);

\draw[edge] (e)edge(f);
\draw[edge] (e)edge(g);
\draw[edge] (e)edge(h);
\draw[edge] (f)edge(g);
\draw[edge] (f)edge(h);
\draw[edge] (g)edge(h);

\draw[edge] (d)edge(g);
\draw[edge] (d)edge(i);
\draw[edge] (d)edge(j);
\draw[edge] (g)edge(i);
\draw[edge] (g)edge(j);
\draw[edge] (i)edge(j);

\draw[edge] (b)edge(e);
\draw[edge] (b)edge(k);
\draw[edge] (b)edge(l);
\draw[edge] (e)edge(k);
\draw[edge] (e)edge(l);
\draw[edge] (k)edge(l);

\begin{scope}[xshift=250]
    \node[vertex] (a) at (-1,-1) {};
   \node[vertex] (b) at (1,-1) {};
   \node[vertex] (c) at (-2,1) {};
   \node[vertex] (d) at (0,1) {};
  
   \node[vertex] (f) at (3,-1) {};
   \node[vertex] (g) at (2,1) {};
   \node[vertex] (h) at (4,1) {};
   
   \node[vertex] (i) at (0,3) {};
   \node[vertex] (j) at (2,3) {};

	\draw[edge] (a)edge(b);
\draw[edge] (a)edge(c);
\draw[edge] (a)edge(d);
\draw[edge] (b)edge(c);
\draw[edge] (c)edge(d);

\draw[edge] (b)edge(f);
\draw[edge] (b)edge(g);
\draw[edge] (b)edge(h);
\draw[edge] (f)edge(g);
\draw[edge] (f)edge(h);
\draw[edge] (g)edge(h);

\draw[edge] (d)edge(g);
\draw[edge] (d)edge(i);
\draw[edge] (d)edge(j);
\draw[edge] (g)edge(i);
\draw[edge] (g)edge(j);
\draw[edge] (i)edge(j);

\end{scope}
\end{tikzpicture}
\end{center}
\vspace{-0.3cm}
\caption{The graph on the left is both 2-connected and redundantly rigid in any analytic normed plane. To reduce this graph, the first operation must be an edge-deletion on one of the four `inner' edges, then we may apply a $K_4^-$-reduction on the $K_4$ we deleted the edge from. This can be followed by an edge-reduction that contracts the edge added by the $K_4^-$-reduction and deletes one of the three remaining `inner' edges resulting in the graph on the right. Now a $K_4^-$-reduction gives $B_2$ and finally an edge-reduction gives $B_1$. Hence to construct this graph using Theorem \ref{thm:1}, we simply reverse this sequence of operations.}
\label{fig:sequence}
\end{figure}

It follows from Lemma \ref{lem:Mcon} that it suffices to consider $\mathcal{M}(2,2)$-connected graphs.
Given an $\mathcal{M}(2,2)$-connected graph $G$, we say that an edge-reduction of $G$ is \emph{admissible} if the resulting graph is $\mathcal{M}(2,2)$-connected;
if the edge-reduction is equivalent to a 1-reduction on a node, then we say that the node is \emph{admissible}.
Similarly, we say that an edge-deletion of $G$ is \emph{admissible} if the resulting graph is $\mathcal{M}(2,2)$-connected, and we describe the removed edge to be \emph{admissible}.
A $K_4^-$-reduction adding the edge $e$ is \emph{admissible} if the resulting graph is $\mathcal{M}(2,2)$-connected. 
Lemma \ref{l:k4-moves} implies that a $K_4^-$-reduction adding the edge $e$ is admissible if and only if $e$ is not an edge of $G$.
We will show that every $\mathcal{M}(2,2)$-connected graph has an admissible edge-reduction, an admissible edge deletion or an admissible $K_4^-$-reduction.
We require three technical lemmas.

\begin{lem}\label{lem:basic}
    Let $G = (V,E)$ be $\mathcal{M}(2,2)$-connected and let $H_i=(V_i,C_i)$, where $1 \leq i \leq t$, be the $(2,2)$-circuits in $G$ induced by an ear decomposition $C_1,C_2,\dots,C_t$ of $G$ with $t\geq 2$. Put $Y=V_t \sm \bigcup_{i=1}^{t-1}V_i$ and $X=V_t \sm Y$. The following statements hold.
    \begin{enumerate}[(i)]
        \item \label{lem:basic.1} Either $Y=\emptyset$ and $|\tilde C_t|=1$ or $Y\neq\emptyset$ and every edge $e\in \tilde C_t$ is incident to $Y$.
        \item \label{lem:basic.2} $|\tilde C_t|=2|Y|+1$.
        \item \label{lem:basic.3} If $Y\neq \emptyset$ then $X$ is critical in $H_t$.
        \item \label{lem:basic.4} If $Y\neq\emptyset$ then $G[Y]$ is connected.
        \item \label{lem:basic.5} $|X| \geq 4$.
    \end{enumerate}
\end{lem}

\begin{proof}
    Part (\ref{lem:basic.1}) is an easy consequence of (E3).
    Part (\ref{lem:basic.2}) follows from Lemma \ref{lem:ear}(\ref{lem:ear.3}) and the observation that each set $D_i$ has rank $2|\bigcup_{j=1}^i V_j|-2$ in $\mathcal{M}(2,2)$ 
    as $\bigcup_{j=1}^{i} H_j$ is
    $\mathcal{M}(2,2)$-connected. 
    For part (\ref{lem:basic.3}), 
    as $Y\neq\emptyset$ (\ref{lem:basic.1}) and (\ref{lem:basic.2}) together imply that
    \[i_{H_t}(X)=|C_t|-|\tilde C_t|=2|V_t|-1-(2|Y|+1)=2|X|-2.\]
    Hence $X$ is critical in $H_t$.
    
    For part (\ref{lem:basic.4}), as $Y \neq \emptyset$ the induced subgraph $G[Y]$ exists. 
    Let $Y_1, \dots, Y_k$ be the vertex sets of the components of $G[Y]$. 
    As $H_t$ is a $(2,2)$-circuit, $|E'| \leq 2|V'|-2$ for every proper subgraph $(V', E')$ of $H_t$. 
    So, since $X$ is critical, if $k \geq 2$ then $i_{H_t}(Y_i)+d(X,Y_i) \leq 2|Y_i|$ for all $1 \leq i \leq k$. Hence, part (\ref{lem:basic.2}) implies that
    \begin{align*}
        2|Y|+1 = |\tilde{C}_t| = \sum_{i = 1}^k (i_{H_t}(Y_i)+d(X,Y_i)) \leq \sum_{i = 1}^k 2|Y_i| = 2|Y|
    \end{align*}
    which is a contradiction. Therefore $k = 1$ and so $G[Y]$ is connected.
    
    Finally, for part (\ref{lem:basic.5}), if $Y = \emptyset$ then since $H_t$ is a $(2,2)$-circuit, $|X| = |V_t| \geq 5$. If $Y \neq \emptyset$ then part (\ref{lem:basic.3}) implies that $X$ is critical in $H_t$, so $|X| = 1$ or $|X| \geq 4$. As $C_1, \dots, C_t$ is an ear decomposition of $G$, we have $C_t \cap D_{t-1} \neq \emptyset$ which implies $i_{H_t}(X) \neq 0$ and hence $|X| \neq 1$.
\end{proof}

Before the next rather technical result, we make the following observation about 1-reductions.
Let $H = (V,E)$ be a $(2,2)$-circuit, let $v \in V$ be a node with $N(v) = \{x, y, z\}$, and suppose that $xy \notin E$. Since $H-vz$ is $(2,2)$-tight, $H-v$ is $(2,2)$-tight and so $H-v+xy$ contains a unique $(2,2)$-circuit $J$. We have $J = H-v+xy$ if and only if the $1$-reduction of $H$ at $v$ adding $xy$ is admissible. If $J \neq H-v+xy$ then $V(J)$ is the minimal $v$-critical set in $H$ containing $\{x, y\}$.

\begin{lem}\label{lem:splitinG}
Let $G = (V,E)$ be $\mathcal{M}(2,2)$-connected and let $H_i = (V_i, C_i)$, where $1 \leq i \leq t$, be the subgraphs of $G$ induced by an ear decomposition $C_1, C_2, \dots , C_t$ of $G$ with $t\geq 2$. Let $Y = V_t \sm \bigcup_{i=1}^{t-1} V_i$ and $X = V_t \sm Y$. Let $v$ be a node in $G$ contained in $Y$, take $x,y
\in N_G(v)$, and suppose $xy \notin E$. Let $J$ be the unique $(2,2)$-circuit in $H_t-v+xy$ and $C=E(J)$. If $E(H_t-v+xy)\sm E(H_t[X])\subsetneq C$ then $v$ is admissible in $G$. 
\end{lem}

\begin{proof}
We show that any pair of edges of $G-v+xy$ is contained in a $(2,2)$-circuit. Note $E(H_t[X])\subseteq D_{t-1}$  by Lemma \ref{lem:basic}(\ref{lem:basic.1}). 
Since $E(H_t-v+xy)\sm E(H_t[X])\subsetneq C$, $E(G-v+xy) = E(H_t-v+xy) \cup D_{t-1}$ and $D_{t-1} \cup C \subseteq E(G-v+xy)$, we have $D_{t-1} \cup C=E(G-v+xy)$.
Hence $G-v+xy= (\bigcup_{i=1}^{t-1}H_i)\cup J$.
As $C \subseteq E(H_t-v+xy)$ and $E(H_t-v+xy)\sm E(H_t[X])\subsetneq C$,
we have $C\cap E(H_t[X])\neq \emptyset$.
Fix an edge $e\in C\cap E(H_t[X])$ and choose any edge $f \in E(G-v+xy)-e$. 
If $f \in D_{t-1}$ then, since Lemma \ref{lem:ear}(i) implies that $\bigcup_{i=1}^{t-1}H_i$ is $\mathcal{M}(2,2)$-connected, there exists a $(2,2)$-circuit contained in $G-v +xy$ and containing $e$ and $f$. 
If $f \in C$ then, since $J$ is clearly $\mathcal{M}(2,2)$-connected, there exists a $(2,2)$-circuit containing $e$ and $f$. So, there exists $e \in E(G-v+xy)$ such that for all $f \in E(G-v+xy)-e$ there exists a $(2,2)$-circuit containing $e$ and $f$. The result now follows by the transitivity of the equivalence relation defining $\mathcal{M}(2,2)$-connectivity.
\end{proof}

\begin{lem}\label{lem:m223notink4}
Let $G = (V,E)$ be $\mathcal{M}(2,2)$-connected and let $H_i = (V_i, C_i)$, where $1 \leq i \leq t$, be the subgraphs of $G$ induced by an ear decomposition $C_1, C_2, \dots , C_t$ of $G$ with $t\geq 2$.
Let $Y=V_t\sm \bigcup_{i=1}^{t-1}V_i$ and $X=V_t\sm Y$. Suppose $Y \neq \emptyset$ and no 3-edge-separation $(F_1, F_2)$ of $H_t$ has the property that $|V(F_1)|, |V(F_2)| \geq 2$ and $F_i\subset H_t[Y]$ for some $i\in\{1,2\}$. Choose a set $X_{0}$ that contains $X$ and is critical in $H_{t}$\footnote{As $Y \neq \emptyset$, Lemma \ref{lem:basic}(\ref{lem:basic.3}) implies $X$ is critical and hence such an $X_{0}$ exists.}. Let $X_1,\dots , X_n$ be critical sets in $H_t$ with $|X_i|\geq 2$ for $1\leq i \leq n$ and let $\mathcal{Y} = V_t \setminus \bigcup^n_{i=0} X_i$. If $|\mathcal{Y}|\geq 2$ or $\bigcup^n_{i=0} H_t[X_i]$ is disconnected, then
$\mathcal{Y}$ contains at least two nodes of $G$.
\end{lem}

\begin{proof}
    Let $Z_1,\dots, Z_m$ be the vertex sets of the connected components of $\bigcup_{i=0}^n H_t[X_i]$.
    By Lemma \ref{lem:connectivity} each $X_i$ is contained in exactly one set $Z_j$.
    By reordering if necessary we may assume that $X_0 \subseteq Z_1$,
    which implies $Z_j \subseteq Y$ for each $2 \leq j \leq m$.
    Lemma \ref{lem:basic}(\ref{lem:basic.5}) implies that $|X_0| \geq |X| \geq 4$. So, as $X_i$ is critical in $H_t$ for each $1 \leq i \leq n$, we have $|X_i| \geq 4$ for all $0 \leq i \leq n$.
    As $|\mathcal{Y}| \geq 2$ or $\bigcup_{i = 0}^{n}H_{t}[X_{i}]$ is disconnected, we have $|V_t \setminus Z_j| \geq 2$ for all $1 \leq j \leq m$.
    Since no 3-edge-separation $(F_1,F_2)$ of $H_t$ has the property that $|V(F_1)|, |V(F_2)| \geq 2$ and $F_i\subset H_t[Y]$ for some $i\in\{1,2\}$ and since $V_t \setminus Z_1 \subseteq Y$ and $Z_j \subseteq Y$ for each $2 \leq j \leq m$,
    we have $d_{H_{t}}(Z_j , V_t \sm Z_j) \geq 4$ for each $1 \leq j \leq m$.
    By Lemma \ref{lem:union}, $Z_j$ is critical in $H_t$ for all $1 \leq j \leq m$. So for all $1 \leq j \leq m$,
    \[\sum_{v \in Z_{j}}(4-d_{H_{t}[Z_{j}]}(v)) = 4|Z_j|-2i_{H_{t}}(Z_j) = 4.\]
    Hence
    \begin{eqnarray*}\sum_{v \in Z_{j}}(4-d_{H_{t}}(v)) &=& 4|Z_j|-(2i_{H_{t}}(Z_j)+d_{H_{t}}(Z_j, V_t \sm Z_j))\\ &=& 4-d_{H_{t}}(Z_j, V_t \sm Z_j) \leq 0.\end{eqnarray*}
    Therefore $\sum_{j = 1}^m \sum_{v \in Z_j}(4-d_{H_{t}}(v)) \leq 0$.
    
    Since $|C_t| = 2|V_t|=1$ we have $\sum_{v \in V_t}(4 - d_{H_t}(v)) = 4|V_t|-2|C_t|= 2$. Combining this with the previously inequality gives us that
    \[2 = \sum_{v \in V_t} (4-d_{H_{t}}(v)) \leq \sum_{v \in \mathcal{Y}} (4-d_{H_{t}}(v)).\]
    As $\delta(H_t)=3$ it follows
    that $\mathcal{Y}$ contains at least two nodes of $H_{t}$. 
    As $\mathcal{Y} \subseteq Y$,
    these are also nodes of $G$.
\end{proof}

We can now prove our first main combinatorial result which establishes that admissible reductions always exist under a technical connectivity hypothesis. 

\begin{thm}\label{thm:key} 
Suppose $G = (V,E)$ is an $\mathcal{M}(2,2)$-connected graph 
with an ear decomposition $C_1,C_2,\dots,C_t$ such that $t \geq 2$, and for all $1 \leq i \leq t$ let $H_i=(V_i,C_i)$ be the subgraph of $G$ induced by $C_i$.
Let $Y=V_t\sm \bigcup_{i=1}^{t-1}V_i$ and let $X=V_t\sm Y$.
Suppose 
no 3-edge-separation $(F_1, F_2)$ of $H_t$ has the property that $|V(F_1)|, |V(F_2)| \geq 2$ and $F_i\subset H_t[Y]$ for some $i\in\{1,2\}$.
Then there is an admissible edge-reduction, edge-deletion or $K_4^-$-reduction of $G$.
\end{thm}

\begin{proof}
We proceed by contradiction. Suppose that $G$ has no admissible edge-reductions, admissible edge-deletions or admissible $K_4^-$-reductions. 
If $Y=\emptyset$ then Lemma \ref{lem:basic}(\ref{lem:basic.1}) implies that $\tilde C_t =\{e\}$ and Lemma \ref{lem:ear}(\ref{lem:ear.1}) implies $G-e$ is $\mathcal{M}(2,2)$-connected, and so $e$ is admissible. 
Hence $Y\neq \emptyset$.
By Lemma \ref{lem:basic}(\ref{lem:basic.3}) and (\ref{lem:basic.5}), $X$ is critical in $H_{t}$ and $|X| \geq 4$. 
Lemma \ref{lem:nix1} now implies that $Y$ contains a node $v$ in $H_t$.
Note that since $C_t$ is the last $(2,2)$-circuit in the ear decomposition,
every node in $H_t$ that is contained in $Y$ is a node in $G$ also.

Label the vertices in $N(v)$ by $x,y,z$
(note that this is not ambiguous since $N_G(v) = N_{H_t}(v)$).
First suppose that $N(v)\subseteq X$.
By Lemma \ref{lem:basic}(\ref{lem:basic.4}) we have $Y=\{v\}$.
Lemma \ref{lem:basic}(\ref{lem:basic.1}) implies that $C_{1}, \dots, C_{t-1}$ is an ear decomposition of $G-v$, and so $G-v$ is $\mathcal{M}(2,2)$-connected by Lemma \ref{lem:ear}(\ref{lem:ear.1}).
If $xy \notin E$ then Lemma \ref{lem:add} implies that $G-v+xy$ is $\mathcal{M}(2,2)$-connected and hence $v$ is admissible in $G$, a contradiction.
If $xy \in E$ then $G-xy$ is a $1$-extension of $G-v$ and so Lemma \ref{lem:add} implies that $G-xy$ is $\mathcal{M}(2,2)$-connected and hence $xy$ is admissible in $G$, a contradiction.

Suppose next that $|N(v) \cap X| = 2$, say $x,y\in X$ and $z\in Y$.
If $xz, yz \in E$ then $xz, yz \in C_t$ and so, as $X$ is critical in $H_{t}$, we have $i_{H_{t}}(X \cup \{v, z\}) =
2|X \cup \{v, z\}|-1$. 
Hence $Y = \{v, z\}$ and $d_G(X,\{v, z\}) = 4$.
This in turn implies that $(G[V\sm\{v,z\}], G[\{v,x,y,z\}])$ is a 2-vertex-separation with vertex cut $\{x,y\}$.
If $xy\notin E$ then $G[V \sm \{v,z\}]+xy$ is an admissible $K_4^-$-reduction of $G$ by Lemma \ref{l:k4-moves},
a contradiction.
Hence $xy\in E$.
Note that $G[V \sm \{v,z\}] = \bigcup_{i=1}^{t-1} H_i$ by Lemma \ref{lem:basic}(\ref{lem:basic.1}),
and so $G[V \sm \{v,z\}]$ is $\mathcal{M}(2,2)$-connected by Lemma \ref{lem:ear}(\ref{lem:ear.1}).
As $G-xy$ is a $K_4^-$-extension of $G[V \sm \{v,z\}]$, Lemma \ref{l:k4-moves} implies $G-xy$ is $\mathcal{M}(2,2)$-connected which
contradicts the assumption that $G$ has no admissible edge-deletions.

Alternatively we have that $|\{xz, yz\} \cap E| \leq 1$. Then we may suppose, without loss of generality, that $xz \notin E$.
Let us denote the graph obtained by the $1$-reduction of $H_{t}$ at $v$ adding the edge $xz$ by $H'_t=(V'_t,C'_t)$.
Suppose first that this 1-reduction of $H_t$ is admissible.
Then $E(H_{t}[X]) \subsetneq C'_t$. 
As $X$ is critical in $H_{t}$ and $|X| \geq 4$, we have $E(H_{t}[X]) \neq \emptyset$. 
Therefore $E(H_{t}[X]) \cap C'_t \neq \emptyset$ and so $C'_t \sm E(H_{t}[X]) \subsetneq C'_t$. 
So $v$ is admissible in $G$ by Lemma \ref{lem:splitinG}, a contradiction.

Hence the $1$-reduction of $H_{t}$ at $v$ adding the edge $xz$ is not admissible.
By switching $x$ and $y$, we also have that if $yz \notin E$ then the $1$-reduction of $H_{t}$ at $v$ adding the edge $yz$ is not admissible. 
Since $xz \notin E$,
Lemma \ref{lem:nix3} implies that there exists a minimal $v$-critical set $X_{1}$ of $H_t$ containing $\{x, z\}$ (but not containing $v$ or $y$). 
Lemma \ref{lem:union} then implies that $X \cup X_{1}$ and $X \cap X_{1}$ are critical in $H_{t}$, and $d_{H_t}(X,X_{1}) = 0$. 
As $d_G(\{v\}, X \cup X_{1}) = d_{H_t}(\{v\}, X \cup X_{1}) = 3$,
it follows that $X \cup X_{1} = V_{t}-v$ and so $Y-v \subseteq X_{1}\setminus X$.
Let $J$ be the unique $(2,2)$-circuit in $H_t'$ and let $C=E(J)$. 
Note that the minimality of $X_1$ implies that $V(J)=X_1$. Hence $C_t' \sm E(H_t[X])\subseteq C$. 
Since $X\cap X_1$ is critical, the graph $H_t[X\cap X_1]$ is connected by Lemma \ref{lem:connectivity}. 
If $|X\cap X_1|=1$ then $X\cap X_1=\{x\}$, $yz \notin E$ and $\{v,x\}$ is a 2-vertex-separation of $H_t$. 
Consider the 1-reduction of $H_t$ that deletes $v$ and adds $yz$. 
As such a 1-reduction must be non-admissible, there exists a critical set $X_2$ containing $y$ and $z$ but not $x$ or $v$. 
However $\{x\}$ is 1-vertex-separation of $H_t[X\cup X_1]$,
contradicting Lemma \ref{lem:connectivity}.
Hence $|X\cap X_1|\geq 2$ and thus $E(H_t[X])\cap C \neq \emptyset$.
It now follows that
\begin{align*}
    C_t'\sm E(H_t[X]) = C \sm E(H_t[X]) \subsetneq C.
\end{align*}
Therefore $v$ is admissible in $G$ by Lemma \ref{lem:splitinG}, a contradiction.
Hence $|N(v) \cap X| \leq 1$,
i.e., $v$ has at most one neighbour in $X$.
Since $v$ was chosen arbitrarily,
this property holds for every node in $Y$.

\begin{claim}\label{claim:6.6}
There exists a node in $Y$ that is not contained in a subgraph of $G$ isomorphic to $K_4$.
\end{claim}

\begin{proof}[Proof of claim]
Let $X_{1}, \dots, X_{n}$ be all the subsets of $V_{t}$ that induce a $K_{4}$ subgraph in $H_{t}$ and let $X = X_{0}$. Let $\mathcal{Y} = V_t \setminus \bigcup^n_{i=0} X_i$. 
If $|\mathcal{Y}|\geq 2$ or $\bigcup^n_{i=0} H_t[X_i]$ is disconnected then, as no 3-edge-separation $(F_1, F_2)$ of $H_t$ has the property that $|V(F_1)|, |V(F_2)| \geq 2$ and $F_i\subset H_t[Y]$ for some $i\in\{1,2\}$, we can apply Lemma \ref{lem:m223notink4} to deduce the claim. So suppose otherwise.
As $\bigcup_{i = 0}^{n} H_{t}[X_{i}]$ is connected, by relabelling the subscripts from $1$ to $n$ (if necessary) we may assume that $X_{0}, \dots, X_{n}$ are ordered such that $X_{j} \cap (\bigcup_{i = 0} ^{j-1} X_{i}) \neq \emptyset$ for all $1 \leq j \leq n$. Fix $s = \min \{j: \bigcup_{i = 0}^{j}X_{i} = \bigcup_{i = 0}^{n} X_{i}\}$. 

If $s = 0$ then, trivally, the claim follows,
so we may suppose that $s \geq 1$. 
Then $\bigcup_{i = 0}^{s-1} X_{i}$ is a proper subset of $V_{t}$. 
Lemma \ref{lem:union} implies that, for each $1 \leq j \leq s-1$, the set $\bigcup_{i = 0}^{j}X_i$ is critical in $H_{t}$ and
\begin{equation}\label{eqclaim:6.6}
    d_{H_{t}}\left(\bigcup_{i = 0}^{j-1}X_{i} \sm X_j, ~ X_{j} \sm \bigcup_{i = 0}^{j-1}X_{i} \right) = 0.
\end{equation}
If $|(\bigcup_{i = 0}^{s-1} X_{i}) \cap X_{s}| \geq 2$ then, as $H_t[X_s] \cong K_4$, the induced subgraph of $H_t$ on the vertex set $(\bigcup_{i = 0}^{s-1} X_{i}) \cap X_{s}$ has at least one edge.
It follows from (\ref{eqclaim:6.6}) that any edge in the aforementioned induced subgraph must be induced by one of the sets $X_0,\dots,X_{s-1}$,
hence there exists $0 \leq j \leq s-1$ such that $|X_{j} \cap X_{s}| \geq 2$.
However, as $|X_{j} \cap X_{s}| < |X_s| = 4$, we note that $X_{j} \cap X_{s}$ is not critical in $H_{t}$ and hence Lemma \ref{lem:union} implies that $X_{j} \cup X_{s} = V_{t}$. 
So, $X_{j} = X$ and $Y = X_{s} \setminus X$ contains either 1 or two nodes. 
This contradicts our assumption that every node in $Y$ has at most one neighbour in $X$.
Hence $|(\bigcup_{i = 0}^{s-1} X_{i}) \cap X_{s}| = 1$. 
Fix $Z = \bigcup_{i = 0}^{s-1} X_{i}$, so $Z$ is critical in $H_t$.
Also fix $w$ to be the unique vertex in the set $Z \cap X_{s}$. 
As $|\mathcal{Y}| \leq 1$ we now have that $H_t$ is one of the two graphs shown in Figure \ref{fig:K_4ChainReduction}.

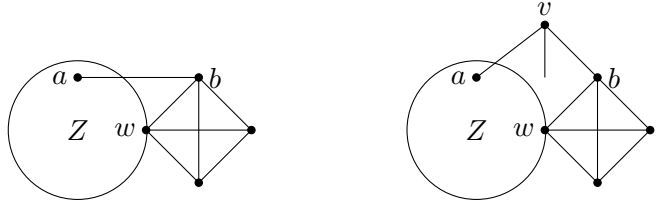
\begin{figure}[ht]
\begin{tikzpicture}[scale=0.7]
\begin{scope}[xshift=250]
    \draw (-5,10) circle (37.5pt);
    \node [rectangle, draw=white, fill=white] (b) at (-5,10) {$Z$};
    
    \filldraw (-3.7,10) circle (2pt) node[anchor=east]{$w$};
    \filldraw (-2.7,9) circle (2pt);
    \filldraw (-2.7,11) circle (2pt) node[anchor=west]{$b$};
    \filldraw (-1.7,10) circle (2pt);
    \filldraw (-5, 11) circle (2pt) node[anchor=east]{$a$};
    \filldraw (-3.7,12) circle (2pt) node[anchor=south]{$v$};
    
    \draw[black] (-3.7,10) -- (-2.7,9) -- (-1.7,10) -- (-2.7, 11) -- (-3.7, 10) -- (-1.7, 10);
    \draw[black] (-2.7, 9) -- (-2.7,11);
    \draw[black] (-5, 11) -- (-3.7, 12) -- (-2.7, 11);
    \draw[black] (-3.7, 12) -- (-3.7, 11);
\end{scope}
\begin{scope}[xshift=-250]
    \draw (5,10) circle (37.5pt);
    \node [rectangle, draw=white, fill=white] (b) at (5,10) {$Z$};
    
    \filldraw (6.3,10) circle (2pt) node[anchor=east]{$w$};
    \filldraw (7.3,9) circle (2pt);
    \filldraw (7.3,11) circle (2pt) node[anchor=west]{$b$};
    \filldraw (8.3,10) circle (2pt);
    \filldraw (5, 11) circle (2pt) node[anchor=east]{$a$};
    
    \draw[black] (6.3,10) -- (7.3,9) -- (8.3,10) -- (7.3, 11) -- (6.3, 10) -- (8.3, 10);
    \draw[black] (7.3,9) -- (7.3,11);
    \draw[black] (7.3, 11) -- (5, 11);
\end{scope}
\end{tikzpicture}
\caption{The two possibilities for the structure of $H_t$ in Claim \ref{claim:6.6} when $Z \cap X_{s} = \{w\}$; $\mathcal{Y} = \emptyset$ (left) and $\mathcal{Y} = \{v\}$ (right).} \label{fig:K_4ChainReduction}
\end{figure}

Suppose that $\mathcal{Y} = \emptyset$ (i.e., $H_t$ is the graph on the left in Figure \ref{fig:K_4ChainReduction}). Then $Z \cup X_s = \bigcup_{i = 0}^s X_{i} = V_{t}$.
As $H_t$ is a $(2,2)$-circuit, there exists $a \in Z$ and $b \in X_s \sm Z$ such that $ab \in C_t$.
Then $(H_{t}[Z], H_{t}[X_{s}+a])$ is a non-trivial 2-vertex-separation of $H_{t}$. Moreover, as $X_{s} -w \subseteq Y$, $(G[V \setminus (X_{s} - w)], G[X_{s} +a])$ is a non-trivial 2-vertex-separation of $G$.

Suppose that $aw \in E$. 
Let $(G_{1}, G_{2})$ be the 2-separation of $G$ with respect to the 2-vertex-separation $(G[V \setminus (X_{s} - w)], G[X_{s}+a])$. 
By Lemma \ref{lem:connected_2/3-sep}, $G_{1}$ and $G_{2}$ are $\mathcal{M}(2,2)$-connected. The edge-reduction of $G$ that deletes $bw$ and contracts $ab$ gives the graph $G_{1}$, so there is an admissible edge-reduction of $G$.
This is a contradiction, so
we must have $aw \notin E$. 
Let $(G_{1}, G_{2})$ be the 1-separation of $G$ with respect to the 2-vertex-separation $(G[V \setminus (X_{s} -w)], G[X_s +a])$ and let $( G'_{2},G'_{1})$ be the 1-separation of $G$ with respect to the 2-vertex-separation $(G[X_s +a], G[V \setminus (X_{s} -w)])$. 
As $\delta(G'_2) \leq 2$, Lemma \ref{lem:connected_1-sep} implies $G_{1}$ and $G_{2}$ are $\mathcal{M}(2,2)$-connected. 
The edge-reduction of $G$ that deletes $wb$ and contracts $ab$ gives a graph isomorphic to a $K_4^{-}$-extension of $G_1$. 
Lemma \ref{l:k4-moves} implies that this edge-reduction gives an $\mathcal{M}(2,2)$-connected graph.
However this also implies that there is an admissible edge-reduction of $G$,
another contradiction.

Hence $\mathcal{Y}=\{v\}$ (i.e., $H_t$ is the graph on the right in Figure \ref{fig:K_4ChainReduction}). Then $Z \cup X_s = \bigcup_{i=0}^s X_{i}$ is critical in $H_t$.
As $H_{t}$ is a $(2,2)$-circuit, it follows from Lemma \ref{lem:connectivity} that $d_{H_t}(v) =3$, and $v$ has neighbours $a,b$ in $H_t$ where $a \in Z$ and $b \notin Z$.
Since $a \in Z$, $b\in X_s \sm Z$ and $v \notin Z \cup X_s$, Lemma \ref{lem:union} implies $ab \notin C_t$.
As $b \in Y$, we have $ab \notin E$ also.
Hence $v$ is not contained in a subgraph of $G$ isomorphic to $K_4$.
\end{proof}

Recall that $V_3$ is the set of nodes in $G$\footnote{This is technically an abuse of notation, since the $(2,2)$-circuit $H_3$ has vertex set $V_3$. There is no ambiguity, however, since we only ever refer to arbitrary $(2,2)$-circuits (i.e., $H_i=(V_i,C_i)$) or $H_t$.} and $V_3^*$ is the set of nodes in $G$ that are not contained in a subgraph of $G$ isomorphic to $K_4$.
By Claim \ref{claim:6.6},
we have $Y \cap V_3^* \neq \emptyset$.
For an arbitrary node $v \in Y \cap V_3^*$, given that $N(v) = \{x,y,z\}$ and $|N(v) \cap X| \leq 1$,
we will assume without loss of generality that $y,z \in Y$.
If the $1$-reduction of $H_{t}$ at $v$ adding the edge $e \in \{xy, xz, yz\}$ is admissible then, by a similar argument as in the case that $|N(v) \cap X| = 2$, it follows that $v$ is admissible in $G$, a contradiction.
So $v$ is non-admissible in $H_{t}$. Note also that if there is an admissible $K_4^-$-reduction of $H_t$, say deleting $v,y$ and adding the edge $xz$, then the corresponding $K_4^-$-reduction of $G$ is also admissible (since $|N(v) \cap X| \leq 1$).
Hence, by Lemma \ref{lem:nix2},
$xy,xz,yz \notin C_t$,
and so $xy,xz,yz \notin E$ also.

\begin{claim}\label{claim:6.7}
There exists a node $v \in Y \cap V_3^*$ and a $v$-critical set $X^*$ in $H_t$ such that $X^*$ is node-critical in $H_t$ and $X\subseteq X^*$.
\end{claim}

\begin{proof}[Proof of claim]
Since $Y\cap V_3^*$ is non-empty,  Lemma \ref{lem:forest} implies that $H_t[Y\cap V_3^*]$ is a forest. 
Hence we may fix $v \in Y\cap V_3^*$ to be a leaf.
As $xy,xz,yz \notin C_t$ and $v$ is not admissible in $H_t$,
it follows from Lemma \ref{lem:nix3} that there exist minimal $v$-critical sets $X_1$, $X_2$ and $X_3$ in $H_t$ containing $\{x,z\}$, $\{y, z\}$ and $\{x,y\}$ respectively. 
If $x \in X$ then Lemma \ref{lem:union} implies that both $X\cup X_1$ and $X \cup X_3$ are $v$-critical sets containing $X$ in $H_t$.
As $v$ is a leaf of the forest $H_t[Y \cap V_3^*]$ and $y,z \in Y$,
we have $d_{H_t}(y) \geq 4$ or $d_{H_t}(z) \geq 4$.
Hence one of $X\cup X_1$ and $X \cup X_3$ is node-critical as required.
Suppose instead that $x \notin X$, i.e., $x,y,z\in Y$.
As $v$ is a leaf in $H_t[Y \cap V_3^*]$,
we may assume, without loss of generality, that $d_{H_t} (x)\geq 4$ and $d_{H_t} (y)\geq 4$. 
By Lemma \ref{lem:union},
$X_1 \cup X_2$ is a critical set in $H_t$.
The subgraph of $H_t$ induced by $X_1 \cup X_2 +v$ is a $(2,2)$-circuit,
and so $X_1\cup X_2 =V_t-v$.
Hence $X\cap X_i$ is non-empty for some $i \in \{1,2\}$ and therefore Lemma \ref{lem:union} implies that $X\cup X_i$ is the required $v$-critical, node-critical set in $H_{t}$ containing $X$.
\end{proof}

We may chose a node $v \in Y \cap V_3^*$ and a $v$-critical and node-critical set $X^*\subsetneq V_t$ 
such that $X \subseteq X^*$ and $|X^*|$ is as large as possible over all such choices of $v$ and $X^*$.
Fix $z$ to be the unique neighbour of $v$ that is not contained in $X^*$.
Since $X^*$ is node-critical, $d_{H_t}(z) \geq 4$,
and so $d_G(z) \geq 4$.
By applying Lemma \ref{lem:nix1} to the set $X^* +v$, which is critical in $H_t$, we deduce that the set $Z := V_t\sm (X^*+ v)$ contains a node.

\begin{claim}\label{claim:6.8}
    Every node in $Z$ is contained in a subgraph of $H_t$ isomorphic to $K_4$.
\end{claim}

\begin{proof}[Proof of claim]
    By Lemma \ref{lem:forest}, the graph $H_t[Z \cap  V_3]$ is a forest.
    Choose a leaf $w$ of $H_t[Z \cap V_3]$.
    The node $w$ has at most one neighbour in $X^*$;
    if $w$ had three neighbours in $X^*$ then the induced subgraph of $H_t$ on $X^* + w \subsetneq V_t$ would be a $(2,2)$-circuit, and if $w$ had two neighbours in $X^*$ then $X^*+w$ would be a larger $v$-critical and node-critical set than $X^*$.
    Hence either: (i) $w$ is a leaf node in $H_t$,
    (ii) $w$ is a series node in $H_t$ with exactly one neighbour in $X^*$,
    and said neighbour is also a node,
    or (iii) $w$ is contained in a subgraph of $H_t$ (and hence also $G$) isomorphic to $K_4$.
    If (i) or (ii) hold (i.e., $w \in Y \cap V_3^*$) then (as noted earlier) there are no edges of $H_t$ between the neighbours of $w$.
    By Lemma \ref{lem:nix4},
    there exists a node critical set in $H_t$ strictly containing $X^*$,
    contradicting the maximality of $X^*$.
    Hence every leaf of $H_t[Z \cap V_3]$ is contained in a subgraph of $H_t$ isomorphic to $K_4$.
    
    Suppose that $Z$ contains a node that is not a leaf of $H_t[Z \cap V_3]$.
    Then there exists a node $w \in Z$ that is not a leaf of $H_t[Z \cap V_3]$ with neighbours $a, b \in N(w) \cap Z \cap V_3$ where $a$ is a leaf of $H_t[Z \cap V_3]$. 
    Then $a$ is contained in a subgraph of $H_t$ isomorphic to $K_4$. Since $a$ and $w$ are nodes, $ab \in C_t$. 
    This contradicts the fact that $H_t[Z \cap V_3]$ is a forest. Therefore every node in $Z$ is a leaf of $H_t[Z \cap V_3]$ and hence is contained in a subgraph isomorphic to $K_4$.
\end{proof}

Let $X_0 = X^*$ and let $X_1, \dots, X_k$ be the critical sets in $H_t$ such that for $1 \leq i \leq k$, $|X_i| = 4$ and $X_i \nsubseteq X^*$. As $Z$ contains a node, we can fix a node $w \in Z$. Claim \ref{claim:6.8} then implies that $k \geq 1$.
Suppose that the induced subgraph $\bigcup_{i=0}^k H_t[X_i]$ is disconnected.
By Lemma \ref{lem:m223notink4}, the set $V_t \setminus (\bigcup_{i=0}^k X_i) \subseteq Y$ contains at least two nodes of $H_t$. 
So $V_t \setminus (\bigcup_{i=0}^k X_i + v) \subseteq Z$ contains a node, say $w'$, of $H_t$.
As $w' \notin \bigcup_{i = 0}^k X_i$, $w'$ is not contained in a subgraph of $H_t$ isomorphic to $K_4$, however this contradicts Claim \ref{claim:6.8}. 
Therefore $\bigcup_{i=0}^k H_t[X_i]$ is connected.

By reordering $X_1,\dots, X_k$, we may assume that the induced subgraph $\bigcup_{i=0}^s H_t[X_i]$ is connected for each $1 \leq s \leq k$.
By the maximality of $X^*$, the fact that $k \geq 1$, and Lemma \ref{lem:union},
it follows that $X^* \cup X_1$ must contain $z$.
As $H_t$ is a $(2,2)$-circuit and all three neighbours of $v$ are in $X^* \cup X_1$,
it follows that $X^* \cup X_1 = V_t -v$.
Hence $X_1=\{a,b,w,z\}$ and $N(w) = \{a,b,z\}$, where $a \in X^*$ and $b \notin X^*$.
Further,
Lemma \ref{lem:union} implies $N(z) = \{v,a,b,w\}$ and $d_{H_t}(b) = 3$ as depicted in Figure \ref{fig:wnode}.

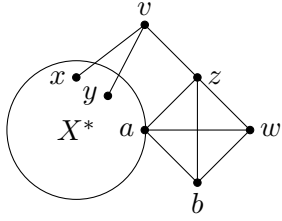
\begin{figure}[ht]
\begin{tikzpicture}[scale=0.7]

\draw (-5,10) circle (37.5pt);
\node [rectangle, draw=white, fill=white] (b) at (-5,10) {$X^*$};

\filldraw (-3.7,10) circle (2pt) node[anchor=east]{$a$};
\filldraw (-2.7,9) circle (2pt) node[anchor=north]{$b$};
\filldraw (-2.7,11) circle (2pt) node[anchor=west]{$z$};
\filldraw (-1.7,10) circle (2pt) node[anchor=west]{$w$};
\filldraw (-5, 11) circle (2pt) node[anchor=east]{$x$};
\filldraw (-4.4, 10.65) circle (2pt) node[anchor=east]{$y$};
\filldraw (-3.7,12) circle (2pt) node[anchor=south]{$v$};

\draw[black] (-3.7,10) -- (-2.7,9) -- (-1.7,10) -- (-2.7, 11) -- (-3.7, 10) -- (-1.7, 10);
\draw[black] (-2.7, 9) -- (-2.7,11);
\draw[black] (-5, 11) -- (-3.7, 12) -- (-2.7, 11);
\draw[black] (-3.7, 12) -- (-4.4, 10.65);

\end{tikzpicture}
\caption{The only possible configuration of $H_t$.} \label{fig:wnode}
\end{figure}

First note that, given $J_1 = G[V\setminus\{b,w,z\}]$ and $J_2 = G[X_1 +v]$, the pair $(J_1,J_2)$ is a 2-vertex separation of $G$.
Let $(G_{1}, G_{2})$ be the 1-separation of $G$ with respect to $(J_1,J_2)$,
and let $(G_{2}', G_{1}')$ be the 1-separation of $G$ with respect to $(J_2,J_1)$.
As $\delta(G'_2) \leq 2$, Lemma \ref{lem:connected_1-sep} implies that $G_1$ is $\mathcal{M}(2,2)$-connected. 
The edge-reduction of $G$ that deletes $za$ and contracts $vz$ gives a graph isomorphic to a $K_4^-$-extension of $G_1$.
Lemma \ref{l:k4-moves} implies that this edge-reduction gives an $\mathcal{M}(2,2)$-connected graph, 
so there is an admissible edge-reduction of $G$.
\end{proof}

An {\em atom} of $G$ is a subgraph $F$ such that $F$ is an element of a non-trivial 2-vertex-separation or non-trivial 3-edge-separation of $G$ and no proper subgraph of $F$ has this property.
Note that if $(F_1,F_2)$ is a non-trivial 2-vertex-separation or non-trivial 3-edge-separation of $G$,
then both $F_1$ and $F_2$ must contain atoms.

\begin{thm}\label{thm:recurse2}
If $G$ is an $\mathcal{M}(2,2)$-connected graph, distinct from $K_5^-$ and $B_{1}$, then there is an admissible edge-reduction, $K_{4}^{-}$-reduction, or edge-deletion of $G$.
\end{thm}

\begin{proof}
The case where $G$ is a $(2,2)$-circuit is Theorem \ref{thm:construction}. 
Suppose $G$ is not a $(2,2)$-circuit. 
Furthermore, suppose that every $\mathcal{M}(2,2)$-connected graph $H$ for which $|E(H)| +|V(H)|<|E|+|V|$ is either $K_5^-$ or $B_1$, or contains an admissible edge-reduction, $K_{4}^{-}$-reduction, or edge-deletion (we will not require this assumption until later in Case \ref{case3b}).
By Theorem \ref{thm:key} and in the notation set up there, it remains to consider the case where there exists a 3-edge-separation $(F_1, F_2)$ of $H_t$ such that $|V(F_1)|, |V(F_2)| \geq 2$ and $F_1\subset H_t[Y]$. 
As $G$ is $\mathcal{M}(2,2)$-connected, $G$ is 2-connected by Lemma \ref{lem:Mcon} and so $G$ has the structure of one of the four cases shown in Figure \ref{fig:3-edge-seps}.

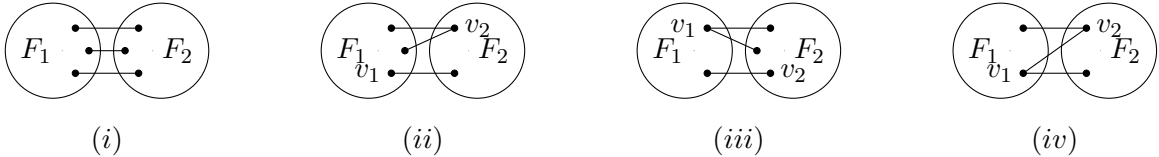
\begin{figure}[htp]
\begin{tikzpicture}[scale=0.6]

\draw (1.2,-3) circle (30pt); 
\draw (-1.2,-3) circle (30pt);

\filldraw (0.7,-2.5) circle (2pt);
\filldraw (.7,-3.5) circle (2pt);
\filldraw (.4,-3) circle (2pt); 
\filldraw (-.7,-2.5) circle (2pt);
\filldraw (-.7,-3.5) circle (2pt);
\filldraw (-.4,-3) circle (2pt);

\draw[black] (-.4,-3) -- (.4,-3);

\draw[black] (-.7,-2.5) -- (.7,-2.5);

\draw[black] (-.7,-3.5) -- (.7,-3.5);

\filldraw (-1,-3) circle (0pt)node[anchor=east]{$F_1$};
\filldraw (1,-3) circle (0pt)node[anchor=west]{$F_2$};

\node [rectangle, draw=white, fill=white] (b) at (0,-5) {$(i)$};

\draw (8.2,-3) circle (30pt); 
\draw (5.8,-3) circle (30pt);

\filldraw (7.7,-2.5) circle (2pt) node[anchor=west]{$v_2$}; 
\filldraw (7.7,-3.5) circle (2pt);
\filldraw (6.3,-2.5) circle (2pt);
\filldraw (6.3,-3.5) circle (2pt) node[anchor=east]{$v_1$}; 
\filldraw (6.6,-3) circle (2pt);

\draw[black] (6.6,-3) -- (7.7,-2.5);

\draw[black] (6.3,-2.5) -- (7.7,-2.5);

\draw[black] (6.3,-3.5) -- (7.7,-3.5);

\filldraw (6,-3) circle (0pt)node[anchor=east]{$F_1$};
\filldraw (8,-3) circle (0pt)node[anchor=west]{$F_2$};

\node [rectangle, draw=white, fill=white] (b) at (7,-5) {$(ii)$};

\draw (15.2,-3) circle (30pt); 
\draw (12.8,-3) circle (30pt);

\filldraw (14.7,-2.5) circle (2pt);
\filldraw (14.7,-3.5) circle (2pt) node[anchor=west]{$v_2$}; 
\filldraw (14.4,-3) circle (2pt);
\filldraw (13.3,-2.5) circle (2pt) node[anchor=east]{$v_1$}; 
\filldraw (13.3,-3.5) circle (2pt);
\draw[black] (13.3,-2.5) -- (14.4,-3);

\draw[black] (13.3,-2.5) -- (14.7,-2.5);

\draw[black] (13.3,-3.5) -- (14.7,-3.5);

\filldraw (13,-3) circle (0pt)node[anchor=east]{$F_1$};
\filldraw (15,-3) circle (0pt)node[anchor=west]{$F_2$};

\node [rectangle, draw=white, fill=white] (b) at (14,-5) {$(iii)$};

\draw (22.2,-3) circle (30pt); 
\draw (19.8,-3) circle (30pt);

\filldraw (21.7,-2.5) circle (2pt) node[anchor=west]{$v_2$}; 
\filldraw (21.7,-3.5) circle (2pt);
\filldraw (20.3,-2.5) circle (2pt);
\filldraw (20.3,-3.5) circle (2pt) node[anchor=east]{$v_1$}; 

\draw[black] (20.3,-3.5) -- (21.7,-2.5);

\draw[black] (20.3,-2.5) -- (21.7,-2.5);

\draw[black] (20.3,-3.5) -- (21.7,-3.5);

\filldraw (20,-3) circle (0pt)node[anchor=east]{$F_1$};
\filldraw (22,-3) circle (0pt)node[anchor=west]{$F_2$};

\node [rectangle, draw=white, fill=white] (b) at (21,-5) {$(iv)$};
\end{tikzpicture}
\caption{The four types of 3-edge-separation.} \label{fig:3-edge-seps}
\end{figure}

Firstly, we show there exists a non-trivial 2-vertex-separation or 3-edge-separation, $(F'_1, F'_2)$, of $H_t$ such that $|V(F'_1) \cap X| \leq 1$ and $F'_1$ is an atom of $H_t$ contained in $F_1$. 
In case (i), $F_1$ contains an atom of $H_t$ as $(F_1,F_2)$ is non-trivial.
For case (ii)--(iv),
label the following vertices:
for case (ii), let $v_2$ be the vertex of $F_2$ with two neighbours in $F_1$ and let $v_1$ be the vertex of $F_1$ not adjacent to $v_2$; 
for case (iii), let $v_1$ be the vertex of $F_1$ with two neighbours in $F_2$ and let $v_2$ be the vertex of $F_2$ not adjacent to $v_1$;
for case (iv), let $v_2$ be the vertex of $F_2$ with two neighbours in $F_1$ and let $v_1$ be the vertex of $F_1$ with two neighbours in $F_2$. 
Then, in cases (ii)--(iv), we note that $(H_{t}[V(F_1) + v_2], H_t[V(F_2) + v_1])$ is a non-trivial 2-vertex-separation of $H_t$. 
Hence $H_{t}[V(F_1) + v_2]$ contains an atom of $H_t$. 
In each of cases (i)--(iv), let $F'_1$ be the atom of $H_t$ contained in $F_1$ or $H_t[V(F_1) + v_2]$ respectively, and let $(F'_1, F'_2)$ be the corresponding 2-vertex-separation or 3-edge-separation of $H_t$.

Note that, in any of the cases (ii)--(iv), if $(F'_1, F'_2)$ is a non-trivial 3-edge-separation, then $v_2 \notin V(F'_1)$.
So, if $(F'_1, F'_2)$ is a non-trivial 3-edge-separation then $V(F'_1) \subseteq Y$.
If $(F'_1, F'_2)$ is a non-trivial 2-vertex-separation of $H_{t}$,
then $|V(F'_1) \setminus V(F_1)| \leq 1$, so $|V(F'_1) \cap X| \leq 1$.
Hence, whether $(F_1',F_2')$ is a non-trivial 2-vertex-separation or a non-trivial 3-edge-separation of $H_{t}$,
the subgraph $F_1'$ is also an atom of $G$.

We now split the proof into three cases.

\begin{case}{1}
$(F'_1, F'_2)$ is a non-trivial 3-edge-separation of $H_{t}$.
\end{case}

As $V(F'_1) \subseteq Y$, $(F'_1, G[V \setminus V(F'_1)])$ is a non-trivial 3-edge-separation of $G$. Let $F'_3 = G[V \setminus V(F'_1)]$ and $S = E \setminus (E(F'_1) \cup E(F'_3)) = \{xa, yb, zc\}$, where $x, y, z \in V(F'_1)$. Let $(G_1, G_2)$ be the 3-separation of $H_t$ with respect to $(F'_1, F'_2)$ and let $(G_1, G_3)$ be the 3-separation of $G$ with respect to $(F'_1, F'_3)$. 
Lemma \ref{lem:circuit_dec2} implies $G_1$ is a $(2,2)$-circuit and Lemma \ref{lem:connected_2/3-sep} implies $G_3$ is $\mathcal{M}(2,2)$-connected. As $F'_1$ is an atom of $H_t$, there are no non-trivial 2-vertex-separations or 3-edge-separations of $G_1$ and $G_{1} \ncong B_{2}$. Hence Theorem \ref{thm:circnodes} implies that $G_1 \cong K_5^-$, $G_1 \cong B_1$, or $G_1$ contains two admissible nodes. 

If $G_1 \cong K_5^-$ then $F'_1 \cong K_4$. Let $G'$ be the graph given by the edge-reduction of $G$ that contracts $xa$ and deletes $xy$. Then $G'$ is isomorphic to a graph given by applying two 1-extensions to $G_3$, so $G'$ is $\mathcal{M}(2,2)$-connected by Lemma \ref{lem:add}. Hence $G$ has an admissible edge-reduction.
If $G_1 \cong B_1$ then one of the vertices $x,y,z$ has degree 3 in $G_1$ and the other two have degree 5 in $G_1$.
Without loss of generality we may suppose that $d_{G_1}(x) = 3$, and so $d_{F'_{1}}(x) = 2$. 
Let $G'$ be the graph given by the edge-reduction of $G$ that contracts $xa$ and deletes $xy$. 
Then $G'$ is isomorphic to a 1-join of $(G_3, B_2)$, so $G'$ is $\mathcal{M}(2,2)$-connected by Lemma \ref{lem:connected_i-join}. 
Hence $G$ has an admissible edge-reduction.

Suppose that $G_1$ contains two admissible nodes.
Set $v$ to be the vertex added to $F_1'$ to form $G_1$ (i.e., $V(G_1) = V(F'_1) + v$ and $N_{G_1}(v) = \{x,y,z\}$).
Since $G_1$ contains two admissible nodes,
one of the admissible nodes is contained in $V(F'_1)$. 
Let $u$ be this admissible node, and set $N_{G_{1}}(u) = \{r, s, t\}$. 
Suppose the admissible 1-reduction of $G_1$ at $u$ adds the edge $rs$. 
The resulting graph is isomorphic to the one given by the edge-reduction of $G_1$ that contracts $ur$ and deletes $ut$, which we call $G'_1$. 
Since this 1-reduction of the $(2,2)$-circuit $G_1$ is admissible (and so cannot create a degree 2 vertex),
it follows that $v \neq t$.
Without loss of generality we shall suppose that $s \neq v$ also.
Let $N_{G}(u) = \{r', s, t\}$, where $r' = r$ if $r \neq v$, and $r' \in \{a,b,c\}$ if $r=v$. 
Let $G'$ be the graph given by the 1-reduction of $G$ that contracts $ur'$ and deletes $ut$. 
Then $G'$ is a 3-join of $(G'_1, G_3)$ and so $G'$ is $\mathcal{M}(2,2)$-connected by Lemma \ref{lem:connected_i-join}. 
Hence $G$ has an admissible edge-reduction.

\begin{case}{2}
$(F'_1, F'_2)$ is a non-trivial 2-vertex-separation of $H_{t}$ and $E(F'_1) \cap E(F'_2) \neq \emptyset$.
\end{case}

Let $V(F'_1) \cap V(F'_2) = \{x, y\}$ and $F'_3 = G[(V \setminus V(F'_1)) \cup \{x, y\}])$.
As $|V(F'_1) \cap X| \leq 1$, 
$(F_1', F_3')$ is a non-trivial 2-vertex-separation of $G$ and $E(F'_1) \cap E(F'_3) =\{xy\}$.
Let $(G_1, G_2)$ be the 2-separation of $H_t$ with respect to $(F'_1, F'_2)$ and let $(G_1, G_3)$ be the 2-separation of $G$ with respect to $(F'_1, F'_3)$. 
Lemma \ref{lem:circuit_dec2} implies $G_1$ is a $(2,2)$-circuit and Lemma \ref{lem:connected_2/3-sep} implies $G_3$ is $\mathcal{M}(2,2)$-connected. 

As $F'_1$ is an atom of $G$, there is no non-trivial 3-edge-separation of $G_1$.
If there exists a non-trivial 2-vertex-separation of $G_1$,
then $\{x,y\}$ must be the corresponding vertex cut.
As $F'_1$ is an atom of $G$, one of the components of the 2-vertex-separation must be a $K_4$ subgraph of $F'_1$ containing $xy$.
Hence $B_1$ is a subgraph of $G_1$ where $\{x,y\}$ is a vertex cut of $B_1$.
As $G_1$ is a $(2,2)$-circuit, this implies $G_1 \cong B_1$. 
However, then $F'_1 \cong K_4$, which contradicts the fact that $F'_1$ is an atom of $H_t$. Hence there exist no non-trivial 2-vertex-separations or 3-edge-separations of $G_1$ and moreover $G_1 \ncong B_1$. 
Also, $G_1 \ncong K_5^-$, since $G_1$ is not 3-connected. 
So, Theorem \ref{thm:circnodes} implies that $G_1 \cong B_2$ or $G_1$ contains two admissible nodes. 

If $G_1 \cong B_2$ then we may suppose without loss of generality that $d_{F'_{1}}(x) = 2$ and set $N_{F'_{1}}(x) = \{y, a\}$. Let $G'$ be the graph given by the edge-reduction of $G$ that contracts $xa$ and deletes $xy$. Then $G' \cong G_{3}$, so $G'$ is $\mathcal{M}(2,2)$-connected. Hence $G$ has an admissible edge-reduction.

If $G_1$ contains two admissible nodes then let $\{w, z\} = V(G_1) \setminus V(F'_1)$ and let $v$ be an admissible node in $G_1$.
Let $G'_1$ be the graph given by the admissible 1-reduction of $G_1$ that deletes $v$ and adds the edge $e$.
As $v \notin \{w, x, y, z\}$ and $v \in Y$,
$v$ is a node in $G$ and $e \notin E$. 
Now let $G'$ be the graph given by the 1-reduction of $G$ that deletes $v$ and adds the edge $e$. 
Then $G'$ is isomorphic to a 2-join of $(G'_1, G_3)$ and so $G'$ is $\mathcal{M}(2,2)$-connected by Lemma \ref{lem:connected_i-join}. 
Therefore $G$ has either an admissible 1-reduction or an admissible edge-reduction.

\begin{case}{3}
$(F'_1, F'_2)$ is a non-trivial 2-vertex-separation of $H_{t}$ and $E(F'_1) \cap E(F'_2) = \emptyset$.
\end{case}

Let $V(F'_1) \cap V(F'_2) = \{x, y\}$. Let $F'_3 = G[(V \setminus V(F'_1)) \cup \{x, y\}])$.
As $|V(F'_1) \cap X| \leq 1$, $(F'_1, F_3')$ is a non-trivial 2-vertex-separation of $G$ and $E(F'_1) \cap E(F'_3) = \emptyset$. 
Let $(G_1, G_2)$ be the 1-separation of $H_t$ with respect to $(F'_1, F'_2)$ and let $(\tilde{G}_{2}, \tilde{G}_{1})$ be the 1-separation of $H_t$ with respect to $(F'_2, F'_1)$. Lemma \ref{lem:circuit_dec3} implies either $G_1$ and $G_2$ are $(2,2)$-circuits
and $\tilde{G}_{1}$ and $\tilde{G}_{2}$ are not,
or $\tilde{G}_{1}$ and $\tilde{G}_{2}$ are $(2,2)$-circuits
and $G_{1}$ and $G_{2}$ are not.
We consider these possibilities as separate subcases.

\begin{case}{3a}
$\tilde{G}_{1}$ and $\tilde{G}_{2}$ are $(2,2)$-circuits.
\end{case}

Let $(\tilde{G}_{3}, \tilde{G}_{1})$ be the 1-separation of $G$ with respect to $(F'_{3}, F'_{1})$. 
By our assumption, $G_{1}$ is not a $(2,2)$-circuit. 
Since $\tilde{G}_1$ is a $(2,2)$-circuit that contains $G_1$,
it follows that $G_1$ is $(2,2)$-sparse and so not $\mathcal{M}(2,2)$-connected. 
Hence, Lemma \ref{lem:connected_1-sep} implies that $\tilde{G}_{3}$ is $\mathcal{M}(2,2)$-connected. 

As $F'_1$ is an atom of $G$, there is no non-trivial 3-edge-separation of $G_1$.
If there exists a non-trivial 2-vertex-separation of $G_1$,
then $\{x,y\}$ must be the corresponding vertex cut.
As $F'_1$ is an atom of $G$, one of the components of the 2-vertex-separation must be a $K_4$ subgraph of $F'_1$ containing $xy$,
however this contradicts that $xy \notin E$.
Hence there are no non-trivial 2-vertex-separations or 3-edge-separations of $\tilde{G}_{1}$. 
Also, as $\tilde{G}_{1}$ is not 3-connected, $\tilde{G}_{1} \ncong K_5^-$. 
Hence Theorem \ref{thm:circnodes} implies that $\tilde{G}_{1} \cong B_{1}$ or $B_{2}$, or $\tilde{G}_{1}$ contains two admissible nodes. 

If $\tilde{G}_{1} \cong B_{1}$ then $F'_{1} \cong K_{4}^{-}$ with $F_1'=\{x,y,r,s\}$ for some vertices $r,s$. 
Let $G'$ be the graph given by the $K_{4}^{-}$-reduction of $G$ that deletes $r$ and $s$ and adds the edge $xy$. Then $G' = \tilde{G}_{3}$, so $G'$ is $\mathcal{M}(2,2)$-connected and the $K_{4}^{-}$-reduction is admissible.
If $\tilde{G}_{1} \cong B_{2}$ then $F'_{1}$ is isomorphic to $K_4$ with an added vertex connected to exactly one vertex in the $K_4$.
Furthermore, the vertex with degree 1 in $F_1'$ must be either $x$ or $y$.
We may suppose without loss of generality that $d_{F'_{1}}(x) = 1$ and set $a$ to be the single neighbour of $x$ in $F_1'$. 
Let $G'$ be the graph given by the edge-reduction of $G$ that contracts $xa$ and deletes $ya$. 
Then $G'$ is isomorphic to a $K_4^-$-extension of $\tilde{G}_{3}$, 
so $G'$ is $\mathcal{M}(2,2)$-connected by Lemma \ref{l:k4-moves} and this edge-reduction is admissible.

If $\tilde{G}_1$ contains two admissible nodes then let $\{w, z\} = V(G_1) \setminus V(F'_1)$ and let $v$ be an admissible node in $\tilde{G}_1$.
Let $\tilde{G}'_1$ be the graph given by the admissible 1-reduction of $\tilde{G}_1$ that deletes $v$ and adds the edge $e$.
As $v \notin \{w, x, y, z\}$ and $v \in Y$,
$v$ is a node in $G$ and $e \notin E$. 
Now let $G'$ be the graph given by the 1-reduction of $G$ that deletes $v$ and adds the edge $e$. 
Then $G'$ is isomorphic to a 1-join of $(\tilde{G}_{3}, G'_{1})$ and so $G'$ is $\mathcal{M}(2,2)$-connected by Lemma \ref{lem:connected_i-join},
completing this subcase.

\begin{case}{3b}\label{case3b}
$G_{1}$ and $G_2$ are $(2,2)$-circuits.
\end{case}

Let $(G_{1}, G_{3})$ be a 1-separation of $G$ with respect to the non-trivial 2-vertex-separation $(F'_{1}, F'_{3})$.
By Lemma \ref{lem:connected_1-sep},
$G_{3}$ is $\mathcal{M}(2,2)$-connected.
As $F'_{1}$ is an atom of $H_{t}$, there are no non-trivial 2-vertex-separations or 3-edge-separations of $G_{1}$. 
Therefore Theorem \ref{thm:circnodes} implies that $G_{1} \cong K_5^-, B_{1}$, or $B_{2}$, or $G_{1}$ contains two admissible nodes.

If $G_{1} \cong K_5^-$ then $F'_{1}$ is isomorphic to one of two graphs (see Figure \ref{fig:K_5^-}). Let $V(F'_{1}) = \{x, y, a, b, c\}$. We may suppose without loss of generality that $d_{F'_{1}}(c) = 4$. If $ab \in E(F'_{1})$ then we may suppose without loss of generality that $ay \notin E(F'_{1})$. In either case, let $G'$ be the graph given by the edge-reduction of $G$ that contracts $by$ and deletes $bc$. Then $G' \cong G_{3}$, and hence this is an admissible edge-reduction.

\begin{figure}[htp]
\begin{tikzpicture}[scale=0.5]

\draw[draw=gray!30!white,fill=gray!30!white] plot[smooth, tension=1] coordinates{(0,0) (2.5,1.5) (0,3)} -- plot[smooth, tension=1] coordinates{(0,0) (2.5,1.5) (0,3)};
\draw[draw=white,fill=white] plot[smooth, tension=1] coordinates{(0,0) (1,1.5) (0,3)};
\draw[white,thick] (0,0) -- (0,3);

\filldraw (0,3) circle (3pt) node[anchor=south]{$x$};
\filldraw (0,0) circle (3pt) node[anchor=north]{$y$};
\filldraw (-2,3) circle (3pt) node[anchor=south]{$a$};
\filldraw (-4,1.5) circle (3pt) node[anchor=east]{$c$};
\filldraw (-2,0) circle (3pt) node[anchor=north]{$b$};

\draw[black,thick] (0,0) -- (-2,0) -- (-4,1.5) -- (-2,3) -- (0,3) -- (-4,1.5) -- (0,0);
\draw[black,thick] (-2,3) -- (0,0);
\draw[black,thick] (-2,0) -- (0,3);

\draw[thick] plot[smooth, tension=1] coordinates{(0,0) (1,1.5) (0,3)}; 
\draw[thick] plot[smooth, tension=1] coordinates{(0,0) (2.5,1.5) (0,3)};
\node [rectangle, draw=white, fill=white] (b) at (-1,-2) {$(i)$};

\filldraw (1.7,1) circle (0pt)node[anchor=south]{$F'_3$};

\draw[draw=gray!30!white,fill=gray!30!white] plot[smooth, tension=1] coordinates{(12,0) (14.5,1.5) (12,3)} -- plot[smooth, tension=1] coordinates{(12,0) (14.5,1.5) (12,3)};
\draw[draw=white,fill=white] plot[smooth, tension=1] coordinates{(12,0) (13,1.5) (12,3)};
\draw[white,thick] (12,0) -- (12,3);

\filldraw (12,3) circle (3pt) node[anchor=south]{$x$};
\filldraw (12,0) circle (3pt) node[anchor=north]{$y$};
\filldraw (10,3) circle (3pt) node[anchor=south]{$a$};
\filldraw (8,1.5) circle (3pt) node[anchor=east]{$c$};
\filldraw (10,0) circle (3pt) node[anchor=north]{$b$};

\draw[black,thick] (12,0) -- (10,0) -- (8,1.5) -- (10,3) -- (12,3) -- (8,1.5) -- (12,0);
\draw[black,thick] (10,3) -- (10,0) -- (12,3);
\draw[thick] plot[smooth, tension=1] coordinates{(12,0) (13,1.5) (12,3)}; \draw[thick] plot[smooth, tension=1] coordinates{(12,0) (14.5,1.5) (12,3)};
\node [rectangle, draw=white, fill=white] (b) at (11,-2) {$(ii)$};

\filldraw (13.7,1) circle (0pt)node[anchor=south]{$F'_3$};
\end{tikzpicture}
\caption{The graph $G$ in the subcase $G_1 \cong K_5^-$ of Case \ref{case3b}.} \label{fig:K_5^-}
\end{figure}
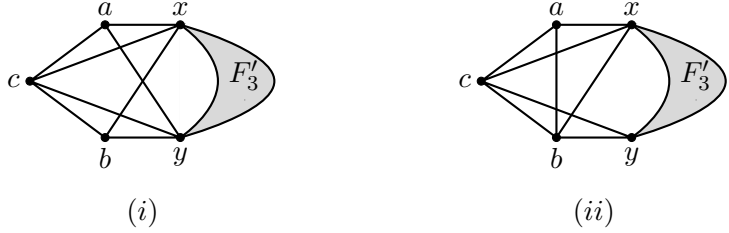

If $G_{1} \cong B_{1}$ then, as $F'_{1}$ is an atom of $H_{t}$, $F'_{1}$ is isomorphic to one of two graphs (see Figure \ref{fig:B_1new}). We may suppose without loss of generality that $d_{F'_{1}}(x) = 2$. Let $N_{F'_{1}}(x) = \{a, b\}$ such that $d_{F'_{1}}(a) \leq d_{F'_{1}}(b)$. In either case, let $G'$ be the graph given by the edge-reduction of $G$ that contracts $xa$ and deletes $xb$. Then $G'$ is isomorphic to a 2-join of $(B_{2}, G_{3})$ and hence $G'$ is $\mathcal{M}(2,2)$-connected by Lemma \ref{lem:connected_i-join}. So $G$ has an admissible edge-reduction.

\begin{figure}[htp]
\begin{tikzpicture}[scale=0.5]

\draw[draw=gray!30!white,fill=gray!30!white] plot[smooth, tension=1] coordinates{(0,0) (2.5,1.5) (0,3)} -- plot[smooth, tension=1] coordinates{(0,0) (2.5,1.5) (0,3)};
\draw[draw=white,fill=white] plot[smooth, tension=1] coordinates{(0,0) (1,1.5) (0,3)};
\draw[white,thick] (0,0) -- (0,3);

\filldraw (0,3) circle (3pt) node[anchor=west]{$y$};
\filldraw (0,0) circle (3pt) node[anchor=north]{$x$};
\filldraw (-2,3) circle (3pt) node[anchor=east]{$b$};
\filldraw (-2,0) circle (3pt) node[anchor=north]{$a$};
\filldraw (-2,5) circle (3pt);
\filldraw (0,5) circle (3pt);

\draw[black,thick] (0,0) -- (-2,0) -- (-2,3) -- (0,3);
\draw[black,thick] (-2,3) -- (0,0);
\draw[black,thick] (-2,0) -- (0,3) -- (0,5) -- (-2,5) -- (0,3);
\draw[black,thick] (0,5) -- (-2,3) -- (-2,5);
\draw[thick] plot[smooth, tension=1] coordinates{(0,0) (1,1.5) (0,3)}; 
\draw[thick] plot[smooth, tension=1] coordinates{(0,0) (2.5,1.5) (0,3)};
\node [rectangle, draw=white, fill=white] (b) at (-1,-2) {$(i)$};

\filldraw (1.7,1) circle (0pt)node[anchor=south]{$F'_3$};

\draw[draw=gray!30!white,fill=gray!30!white] plot[smooth, tension=1] coordinates{(12,0) (14.5,1.5) (12,3)} -- plot[smooth, tension=1] coordinates{(12,0) (14.5,1.5) (12,3)};
\draw[draw=white,fill=white] plot[smooth, tension=1] coordinates{(12,0) (13,1.5) (12,3)};
\draw[white,thick] (12,0) -- (12,3);

\filldraw (12,3) circle (3pt)node[anchor=south]{$y$};
\filldraw (12,0) circle (3pt)node[anchor=north]{$x$};
\filldraw (10,3) circle (3pt)node[anchor=south]{$a$};
\filldraw (10,0) circle (3pt)node[anchor=north]{$b$};
\filldraw (8,0) circle (3pt);
\filldraw (8,3) circle (3pt);

\draw[black,thick] (8,3) -- (10,0) -- (8,0) -- (8,3) -- (10,3) -- (12,0) -- (10,0) -- (10,3) -- (12,3);
\draw[black,thick] (8,0) -- (10,3);
\draw[black,thick] (10,0) -- (12,3);
\draw[thick] plot[smooth, tension=1] coordinates{(12,0) (13,1.5) (12,3)}; \draw[thick] plot[smooth, tension=1] coordinates{(12,0) (14.5,1.5) (12,3)};
\node [rectangle, draw=white, fill=white] (b) at (11,-2) {$(ii)$};

\filldraw (13.7,1) circle (0pt)node[anchor=south]{$F'_3$};
\end{tikzpicture}
\caption{The graph $G$ in the subcase $G_1 \cong B_1$ of Case \ref{case3b}.} \label{fig:B_1new}
\end{figure}
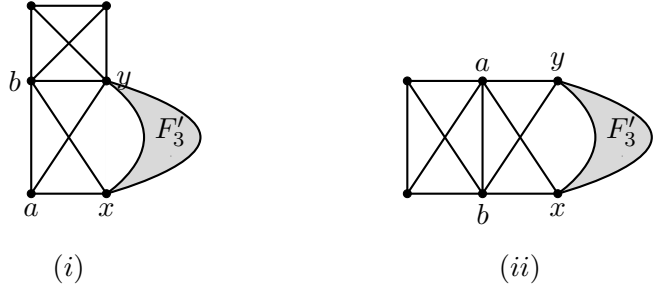

If $G_{1}$ contains two admissible nodes,
then either there exists $v \in V(F'_{1}) \setminus \{x, y\}$ such that $v$ is an admissible node in $G_{1}$, or $x$ and $y$ are the only admissible nodes of $G_{1}$. In the first case, let $G'_{1}$ be given by an admissible 1-reduction of $G_{1}$ at $v$ that adds the edge $e$. 
Importantly, $e \neq xy$, since $xy$ is an edge of $G_1$. 
As $v \notin \{x, y\}$, $v$ is a node in $G$ also. 
Let $G'$ be given by the 1-reduction of $G$ at $v$ that adds the edge $e$.
Then $G'$ is isomorphic to a 1-join of $(G'_{1}, G_{3})$, so $G'$ is $\mathcal{M}(2,2)$-connected by Lemma \ref{lem:connected_i-join}. 
Hence $G$ has an admissible edge-reduction. 
Alternatively, suppose that $x$ and $y$ are the only admissible nodes of $G_{1}$. Let $N_{G_{1}}(x) = \{a, b, y\}$ and let $G'_{1}$ be given by the admissible 1-reduction of $G_{1}$ that removes $x$ and adds the edge $e$ (see Figure \ref{fig:final}). 
As this 1-reduction is admissible, $d_{G'_{1}}(y) \geq 3$, so $e \in \{ya, yb\}$. 
Without loss of generality we may suppose that $e = ya$, so $ya \notin E$. 
Let $G'$ be the edge-reduction of $G$ that contracts $xa$ and deletes $xb$. Then, as $xy, ya \notin E$, $G'$ is isomorphic to a 1-join of $(G'_{1}, G_{3})$ and hence is $\mathcal{M}(2,2)$-connected by Lemma \ref{lem:connected_i-join}. Therefore $G$ has an admissible edge-reduction.

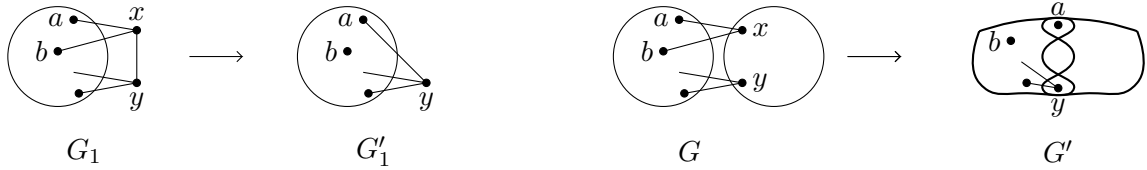
\begin{figure}[htp]
\begin{tikzpicture}[scale=0.7]

\draw (-4,10) circle (27pt);
\draw (1.5,10) circle (27pt);
\draw (7.5,10) circle (27pt);
\draw (9.6,10) circle (27pt);

\filldraw (-2.5,10.5) circle (2pt) node[anchor=south]{$x$};
\filldraw (-2.5,9.5) circle (2pt) node[anchor=north]{$y$};
\filldraw (-3.6,9.3) circle (2pt);
\filldraw (-4,10.1) circle (2pt) node[anchor=east]{$b$};
\filldraw (-3.7,10.7) circle (2pt) node[anchor=east]{$a$};

\draw[black] (-3.7,10.7) -- (-2.5,10.5) -- (-2.5,9.5);
\draw[black] (-4,10.1) -- (-2.5,10.5);
\draw[black] (-3.7,9.7) -- (-2.5,9.5) -- (-3.6,9.3);
\draw[black] (-1.5,10) -- (-.5,10) -- (-.6,9.9);
\draw[black] (-.5,10) -- (-.6,10.1);
\node [rectangle, draw=white, fill=white] (b) at (-3.5,8.2) {$G_1$};

\filldraw (3,9.5) circle (2pt) node[anchor=north]{$y$};
\filldraw (1.9,9.3) circle (2pt); 
\filldraw (1.5,10.1) circle (2pt) node[anchor=east]{$b$};
\filldraw (1.8,10.7) circle (2pt) node[anchor=east]{$a$};

\draw[black] (3,9.5) -- (1.8,10.7);
\draw[black] (1.8,9.7) -- (3,9.5) -- (1.9,9.3);
\node [rectangle, draw=white, fill=white] (b) at (2,8.2) {$G_1'$};

\filldraw (9,10.5) circle (2pt) node[anchor=west]{$x$};
\filldraw (9,9.5) circle (2pt) node[anchor=west]{$y$};
\filldraw (7.9,9.3) circle (2pt);
\filldraw (7.5,10.1) circle (2pt) node[anchor=east]{$b$};
\filldraw (7.8,10.7) circle (2pt) node[anchor=east]{$a$};

\draw[black] (7.5,10.1) -- (9,10.5) -- (7.8,10.7);
\draw[black] (7.9,9.3) -- (9,9.5) -- (7.8, 9.7);
\draw[black] (11,10) -- (12,10) -- (11.9,9.9);
\draw[black] (12,10) -- (11.9,10.1);
\node [rectangle, draw=white, fill=white] (b) at (8,8.2) {$G$};

\draw[thick] plot[smooth, tension=1] coordinates{(13.5,10.48) (13.5,9.5) (14.5,9.3) (15.3,9.4) (14.7,10) (15.3,10.6)(14.5,10.7) (13.5,10.48)};

\draw[thick] plot[smooth, tension=1] coordinates{(16.5,10.48) (16.5,9.5) (15.5,9.3) (14.7,9.4) (15.3,10) (14.7,10.6)(15.5,10.7) (16.5,10.48)};

\filldraw (15,9.4) circle (2pt) node[anchor=north]{$y$};
\filldraw (14.4,9.5) circle (2pt);
\filldraw (14.1,10.3) circle (2pt) node[anchor=east]{$b$};
\filldraw (15,10.6) circle (2pt) node[anchor=south]{$a$};

\draw[black] (14.4,9.5) -- (15,9.4) -- (14.3,9.9);
\draw[black] (11,10) -- (12,10) -- (11.9,9.9);
\draw[black] (12,10) -- (11.9,10.1);
\node [rectangle, draw=white, fill=white] (b) at (15,8.2) {$G'$};

\end{tikzpicture}
\caption{The subcase of Case \ref{case3b} when $G_1$ is not isomorphic to $K_5^-$, $B_1$ of $B_2$. Here $G'$ is the 1-join of $G_1'$ and $G_3$.} \label{fig:final}
\end{figure}

If $G_{1} \cong B_{2}$ then, as $F'_{1}$ is an atom of $H_{t}$, $G$ is isomorphic to the graph shown in Figure \ref{fig:B_2new}. 
Let $\{u\} = N_{F'_{1}}(x) \cap N_{F'_{1}}(y)$ and let
\begin{align*}
    G'_{1} = (V(G_{1}) \cup \{w, z\}, (E(G_{1}) \setminus \{xu\}) \cup \{xw, xz, yw, yz, wz\});
\end{align*}
see Figure \ref{fig:B_2new}.
Then $G'_{1}$ is isomorphic to a 1-join of $(B_{2}, B_{1})$ so $G'_{1}$ is $\mathcal{M}(2,2)$-connected by Lemma \ref{lem:connected_i-join}. 
We now turn our attention to $G_{3}$. 
As $|V(G_{3})| +|E(G_3)| = |V|+|E| - 9$,
it follows from our initial induction assumption (see the first paragraph of the proof) that either $G_3$ is isomorphic to $K_5^-$ or $B_1$, or there exists an admissible $K_{4}^{-}$-reduction, edge-reduction, or edge-deletion of $G_{3}$.
As $G_3$ is not 3-connected, $G_3 \not\cong K_5^-$.
If $G_3 \cong B_1$ then $G$ is isomorphic to the graph on the right of Figure \ref{fig:sequence}, and hence has an admissible $K_4^-$-reduction.
Suppose that there exists an admissible $K_{4}^{-}$-reduction, edge-reduction, or edge-deletion of $G_{3}$.
Label the two vertices in $V(G_{3}) = V(F'_{3})$ as $w',z'$.

\begin{figure}[htp]
\begin{tikzpicture}[scale=0.5]
\draw[draw=gray!30!white,fill=gray!30!white] plot[smooth, tension=1] coordinates{(0,0) (2.5,1.5) (0,3)} -- plot[smooth, tension=1] coordinates{(0,0) (2.5,1.5) (0,3)};
\draw[draw=white,fill=white] plot[smooth, tension=1] coordinates{(0,0) (1,1.5) (0,3)};

\draw[white,thick] (0,0) -- (0,3);

\filldraw (0,3) circle (3pt)node[anchor=south]{$y$};
\filldraw (0,0) circle (3pt)node[anchor=north]{$x$};
\filldraw (-1,4) circle (3pt);
\filldraw (-1,1.5) circle (3pt)node[anchor=east]{$u$};
\filldraw (-1,-1) circle (3pt);
\filldraw (-2,3) circle (3pt);
\filldraw (-2,0) circle (3pt);

\draw[black,thick] (-2,0) -- (-1,1.5) -- (0,3) -- (-1,4) -- (-2,3) -- (0,0) -- (-1,-1) -- (-2,0) -- (0,0);
\draw[black,thick] (0,3) -- (-2,3);
\draw[black,thick] (-1,4) -- (-1,1.5) -- (-1,-1);

\draw[thick] plot[smooth, tension=1] coordinates{(0,0) (1,1.5) (0,3)}; \draw[thick] plot[smooth, tension=1] coordinates{(0,0) (2.5,1.5) (0,3)};
\node [rectangle, draw=white, fill=white] (b) at (0,-2) {$G$};

\filldraw (1.7,1) circle (0pt)node[anchor=south]{$F'_3$};

\filldraw (10,3) circle (3pt)node[anchor=south]{$y$};
\filldraw (10,0) circle (3pt)node[anchor=north]{$x$};
\filldraw (9,4) circle (3pt);
\filldraw (9,1.5) circle (3pt)node[anchor=east]{$u$};
\filldraw (9,-1) circle (3pt);
\filldraw (8,3) circle (3pt);
\filldraw (8,0) circle (3pt);
\filldraw (12,0) circle (3pt)node[anchor=north]{$w$};
\filldraw (12,3) circle (3pt)node[anchor=south]{$z$};

\draw[black,thick] (8,0) -- (9,1.5) -- (10,3) -- (9,4) -- (8,3) -- (9,1.5);
\draw[black,thick] (10,0) -- (9,-1) -- (8,0) -- (10,0);
\draw[black,thick] (10,0) -- (12,3) -- (12,0) -- (10,0) -- (10,3) -- (8,3);
\draw[black,thick] (9,4) -- (9,1.5) -- (9,-1);
\draw[black,thick] (12,0) -- (10,3) -- (12,3);
\node [rectangle, draw=white, fill=white] (b) at (10,-2) {$G'_1$};

\filldraw (1.7,1) circle (0pt)node[anchor=south]{$F'_3$};
\end{tikzpicture}
\caption{The graphs $G$ and $G'_1$ in the subcase $G_1 \cong B_2$ of Case \ref{case3b}.} \label{fig:B_2new}
\end{figure}
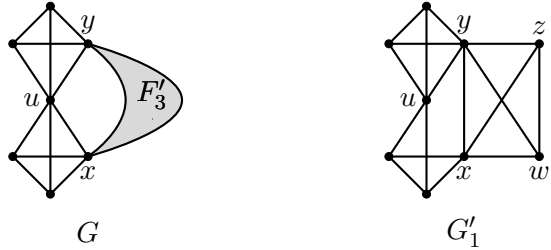

If there exists an admissible $K_{4}^{-}$-reduction of $G_{3}$ then the same $K_{4}^{-}$ is present in $G$ as an induced subgraph. Let $G'_{3}$ be the graph given by the $K_{4}^{-}$-reduction of $G_{3}$ and let $G'$ be the graph given by the $K_{4}^{-}$-reduction of $G$. Then $G'$ is isomorphic to a 1-join of $(G_{1}, G'_{3})$ and so $G'$ is $\mathcal{M}(2,2)$-connected by Lemma \ref{lem:connected_i-join}. Hence $G$ has an admissible $K_{4}^{-}$-reduction.

If there exists an admissible edge-reduction of $G_{3}$, let this contract $e$ and delete $f$ and call the resulting graph $G'_{3}$. As this edge-reduction is admissible, $e \notin \{xy, xw', xz', yw', yz', w'z'\}$ and $f \notin \{xw', xz, yw', yz', w'z'\}$.
Hence $e \in E$, and $f \in E$ if and only if $f \neq xy$.
Suppose that $f \neq xy$ and let $G'$ be given by the edge-reduction of $G$ that contracts $e$ and deletes $f$.
Then $G'$ is isomorphic to a 1-join of $(G_{1}, G'_{3})$ and so $G'$ is $\mathcal{M}(2,2)$-connected by Lemma \ref{lem:connected_i-join}. Hence $G$ has an admissible edge-reduction. 
Now suppose $f = xy$.
Without loss of generality we may suppose that $e = xa$ for some $a \in V(F'_{3})-y$ and denote the vertex resulting from the contraction of $e$ by $x'$.

Let $G'$ be given by the edge-reduction of $G$ that contracts $e$ and deletes $xu$.
If $ya \in E(G_3)$ then $G'$ is isomorphic to a 2-join of $(G'_1,G'_3)$ with $x'$ relabelled as $x$.
Since both $G_1'$ and $G_3'$ are $\mathcal{M}(2,2)$-connected,
$G'$ is $\mathcal{M}(2,2)$-connected also by Lemma \ref{lem:connected_i-join}.
If $ya \notin E(G_3)$ then $x'y \notin E(G'_3)$.
Let $(G_{4}, G_{5})$ be a 1-separation of $G'_{3}$ on $(G'_{3}[V(G'_{3}) \setminus \{w', z'\}], G'_{3}[\{x', y, w', z',\}])$ and let $(\tilde{G}_{5}, \tilde{G}_{4})$ be a 1-separation of $G'_{3}$ on $(G'_{3}[\{x', y, w', z'\}], G'_{3}[V(G'_{3}) \setminus \{w', z'\}])$. 
Then $G'$ is isomorphic to a 1-join of $(G_{4}, G'_{1})$. 
As $\tilde{G}_{5} \cong K_{4}$, Lemma \ref{lem:connected_1-sep} implies that $G_{4}$ is $\mathcal{M}(2,2)$-connected.
So Lemma \ref{lem:connected_i-join} gives that $G'$ is $\mathcal{M}(2,2)$-connected. Therefore $G$ has an admissible edge-reduction. 

Lastly, if there exists an admissible edge-deletion of $G_{3}$ that deletes the edge $e$,
then $e \notin \{xw', xz', yw', yz', w'z'\}$. 
If $e \neq xy$ then $G-e$ is isomorphic to a $1$-join of $(G_{1}, G_{3}-e)$ and so $G-e$ is $\mathcal{M}(2,2)$-connected by Lemma \ref{lem:connected_i-join}. 
Suppose $e = xy$.
Given $G_{3,1} := G_{3}[V(G'_{3}) \setminus \{w', z'\}]-e$ and $G_{3,2} := G_{3}[\{x, y, w', z'\}]-e$, let $(G_{4}, G_{5})$ be a 1-separation of $G_{3}-e$ on $(G_{3,1}, G_{3,2})$ and let $(\tilde{G}_{5}, \tilde{G}_{4})$ be a 1-separation of $G_{3}-e$ on $(G_{3,2}, G_{3,1})$.
As $\tilde{G}_{5} \cong K_{4}$, 
Lemma \ref{lem:connected_1-sep} implies that $G_{4}$ is $\mathcal{M}(2,2)$-connected.
Since $G-xu$ is isomorphic to a 1-join of $(G_{4}, G'_{1})$,
Lemma \ref{lem:connected_i-join} implies that $G-xu$ is $\mathcal{M}(2,2)$-connected. Therefore $G$ has an admissible edge-deletion. 
This concludes the proof.
\end{proof}

\begin{cor}\label{cor:combmain}
A graph $G$ is $\mathcal{M}(2,2)$-connected if and only if $G$ can be generated from $K_5^-$ or $B_1$ by generalised vertex splits, $K_4^-$-extensions and edge additions such that each generalised vertex split preserves $\mathcal{M}(2,2)$-connectivity.
\end{cor}

\begin{proof}
The case when $G$ is a $(2,2)$-circuit is Theorem \ref{thm:construction}. The case when $G$ is not a $(2,2)$-circuit, in one direction, follows from Lemmas \ref{lem:add} and \ref{lem:connected_i-join}. The converse follows from Theorem \ref{thm:recurse2}.
\end{proof}

Finally we can now easily obtain the main result of this section.

\begin{proof}[Proof of Theorem \ref{thm:1}]
    By Proposition \ref{prop:regcreg} and Lemma \ref{lem:Mcon},
    a graph is $M(2,2)$-connected if and only if it is redundantly rigid in $X$ and 2-connected.
    The result now follows from Corollary \ref{cor:combmain}.
\end{proof}

\section{Characterising globally rigid graphs in analytic normed planes}
\label{sec:proof}

The main result of the paper (Theorem \ref{thm:grne}) is obtained in this section. We begin by stating two results from \cite{DN}. The first concerns the base graphs from Theorem \ref{thm:1} (see Figure \ref{fig:smallgraphs1})\footnote{It is worth pointing out that for a given norm, global rigidity for such small graphs can be checked computationally. However, giving a rigorous proof for all analytic norms simultaneously is non-trivial.} and the second is an analogue of the ``averaging theorem'' of Connelly and Whiteley \cite{CW}.

\begin{lem}[{\cite[Theorems 5.3 \& 5.4]{DN}}]\label{lem:basegraphs}
    The graphs $K_5^-$ and $B_1$ are globally rigid in any analytic normed plane.
\end{lem}

\begin{thm}[{\cite[Theorem 3.10]{DN}}]\label{thm:ave}
	Let $(G,p)$ be a globally rigid and infinitesimally rigid framework in a smooth normed space $X$ with finitely many linear isometries.
	Then there exists an open neighbourhood $U$ of $p$ where $(G,q)$ is globally rigid and infinitesimally rigid in $X$ for all $q \in U$.
\end{thm}

\subsection{\texorpdfstring{$K_4^-$}{K4-}-extensions}

In this subsection we prove that $K_4^-$-extensions preserve global rigidity in analytic normed planes. We first require the following lemma.

\begin{lem}\label{l:limitedisom}
	Let $X$ be a non-Euclidean normed plane.
	Then for all but a finite number of points $x \in X$ with $\| x\|=1$,
	if $g$ is an isometry and $g(x) = x$ then $g$ is the identity map.
\end{lem}

\begin{proof}
	Let $A$ be the set of points $x \in X$ where $\|x \| =1$ and there exists a non-trivial linear isometry $g_x$ with $g_x(x) = x$.
	We note that if $x,y \in A$ are linearly independent then $g_x(y) \neq y$ (as otherwise $g_x$ must be the identity map),
	thus the cardinality of $A$ is less than twice the cardinality of the group of linear isometries of $X$.
	The result now follows as all non-Euclidean normed planes have a finite amount of linear isometries (see \cite[Proposition 2.5(ii)]{dew2}).
\end{proof}

Let $X$ be an analytic normed plane, let $z \in X$ be any non-zero point,
and define $\H$ and $\H'$ to be the two open half-planes formed by the line through $0$ and $z$.
It was shown in \cite[Theorem 4.3]{DN} that there exists a unique map $R_z :X \rightarrow X$ (called the \emph{$z$-reflection}) defined by the following properties: (i) if $x \in X$ does not lie on the line through $0$ and $z$, then $R_z(x)$ is the unique point distinct from $x$ where $\|R_z(x)\| = \|x\|$ and $\|R_z(x) - z\| = \|x - z\|$, and (ii) if $x \in X$ lies on the line through $0$ and $z$, then $R_z(x) = x$.
The map $R_z$ has additional useful properties that we will require.  Firstly, it is a homeomorphism and an involution, and it is an analytic diffeomorphism when its domain and codomain are restricted to the set of points that do not lie on the line through $0$ and $z$. Secondly, any point in $\H$ will be mapped to a point in $\H'$ and vice versa.

\begin{lem}\label{l:zreflection}
	Let $X$ be an analytic normed plane,
	$z \in X$ be any non-zero point and $\H$ be an open half plane formed by the line through $0$ and $z$.
	Suppose that $R_z$ is not an isometry.
	Then for any two points $x,y$ that lie in the closure of $\mathcal{H}$ and for any $\varepsilon >0$,
	there exists $x',y' \in \H$ such that $\|x-x'\|<\varepsilon$, $\|y-y'\|<\varepsilon$ and
	\begin{align*}
	    \|R_z(x')- R_z(y')\| \neq \|x'-y'\|, \qquad \|x'- R_z(y')\| \neq \|x'-y'\|, \qquad \|R_z(x')- y'\| \neq \|x'-y'\|.
	\end{align*}
\end{lem}

\begin{proof}
    Define $U \subset \H \times \H$ to be the connected subset of all points $(x,y)$ where $x \neq y$.
    It will suffice for us to show that the sets
    \begin{align*}
        D := \{(x,y) \in U : \|R_z(x)- R_z(y)\| \neq \|x-y\| \}, \qquad D' := \{(x,y) \in U : \|x- R_z(y)\| \neq \|x-y\| \}
    \end{align*}
    are both open conull subsets of $U$.
    
    Define $f : U \rightarrow \mathbb{R}$ to be the analytic function that maps $(x,y) \in U$ to $\|R_z(x)- R_z(y)\| - \|x-z\|$,
    and define $g : U \rightarrow \mathbb{R}$ to be the analytic function that maps $(x,y) \in U$ to $\|x- R_z(y)\| - \|x-z\|$.
    It is immediate that the zero sets of $f$ and $g$ are the sets $U \setminus D$ and $U \setminus D'$ respectively.
    As both $f$ and $g$ are analytic and the set $U$ is connected,
    their zero sets are either closed null sets or the entirety of $U$ (see \cite{mityagin});
    i.e., either $D$ (respectively $D'$) is empty or $D$ (respectively $D'$) is an open conull set.
    If $D = \emptyset$ then $R_z$ is an isometry (as it is an involution and invariant on the line through $0$ and $z$),
    hence $U \setminus D$ is a closed conull set.
    To see that $D'$ cannot be empty,
    note that $\|x - R_z(x)\| - \|x - x\| \neq 0$ and the map $R_z$ is continuous,
    hence we must have $g(x,y) \neq 0$ for sufficiently close $x$ and $y$.
\end{proof}

By using $z$-reflections, we can now prove that $K_4^-$-extensions preserve globally rigidity in analytic normed planes. We refer the reader to Section \ref{sec:redundant} for the definition of the compact differentiable manifold $\mathcal{C}(G,p)$.

\begin{thm}\label{thm:k4-}
	Let $G=(V,E)$ be a globally rigid graph in an analytic normed plane $X$ and let $G'=(V',E')$ be formed from $G$ by a $K_4^-$-extension.
	Then $G'$ is globally rigid in $X$ also.
\end{thm}

\begin{proof}
	Suppose $G'$ is formed from $G$ by the $K_4^-$-extension applied to the edge $v_1v_2$ that adds vertices $u_1,u_2$.
	Since $G$ is globally rigid,
	we can choose a completely strongly regular framework $(G,p)$ that is globally rigid with $p_{v_1}=0$ and $p_{v_2} = z$ for some $\|z\|=1$;
	further,
	by Lemma \ref{l:limitedisom},
	we may assume we chose $z$ such that,
	given any non-trivial linear isometry $g$,
	we have $g(z) \neq z$.
	By Theorem \ref{thm:red},
	$(G,p)$ is redundantly rigid.
	Since $(G,p)$ is completely strongly regular, the set $\mathcal{C}(G - v_1 v_2,p)$ is finite.
	Furthermore, for any placement $q \in X^V$ where $f_{G-v_1v_2}(q) = f_{G-v_1v_2}(p)$, we have	$\|q_{v_1} - q_{v_2}\| = \|p_{v_1} - p_{v_2}\|$ if and only if $q \sim p$.
	For any equivalence class $\tilde{q} \in \mathcal{C}(G,p)$,
	we note that $\|q_{v_1} - q_{v_2}\| = \|q'_{v_1} - q'_{v_2}\|$ for all $q,q' \in \tilde{q}$.
	Since $\mathcal{C}(G - v_1 v_2,p)$ is a finite set,
	it follows that
	\begin{align*}
		t := \min \left\{ |\|q_{v_1} - q_{v_2}\| - \|p_{v_1} - p_{v_2}\| |: f_{G-v_1v_2}(q) = f_{G-v_1v_2}(p), ~ q \not\sim p \right\}.
	\end{align*}
	is well-defined and non-zero.
	
	Choose $\varepsilon < t/4$.
	Let $\H, \H'$ be the open hyperplanes defined by the line through $0,z$.
	By our choice of $z$,
	the $z$-reflection $R_z$ is not an isometry.
	By Lemma \ref{l:zreflection},
	there exists $x_1,x_2 \in \H$ such that $\|x_1\| < \varepsilon$, $\|x_2 - z\| < \varepsilon$, $\|R_z(x_1) - R_z(x_2)\| \neq \|x_1-x_2\|$, $\|x_1 - R_z(x_2)\| \neq \|x_1-x_2\|$ and $\|R_z(x_1) - x_2\| \neq \|x_1-x_2\|$.
	We will now choose $p'$ to be a placement of $G'$ where $p'_v=p_v$ for all $v \in V$ and $p'_{u_i} = x_i$ for each $i \in \{1,2\}$.
	As $\|p'_{v_i} - p'_{u_i} \| < \varepsilon$ for $i \in \{1,2\}$,
	we have for any $q \in f_{G'}^{-1}(f_{G'}(p'))$ that
	\begin{align*}
		\Big|\left\|q_{v_1} - q_{v_2} \| - \|q_{u_1} - q_{u_2} \right\|\Big| \leq \|(q_{v_1} - q_{u_1}) + (q_{u_2} - q_{v_2})\| \leq \|q_{v_1} - q_{u_1}\| + \|q_{v_2} - q_{u_2}\| < 2 \varepsilon
	\end{align*}
	Hence,
	for any $q \in f_{G'}^{-1}(f_{G'}(p'))$ we have
	\begin{eqnarray*}
		\Big|\|q_{v_1} - q_{v_2} \| - \|p'_{v_1} - p'_{v_2} \| \Big| &=&
		\Big|\|q_{v_1} - q_{v_2} \| - \|q_{u_1} - q_{u_2} \| + \|p'_{u_1} - p'_{u_2} \| - \|p'_{v_1} - p'_{v_2} \| \Big|\\
		&\leq& \Big|\|q_{v_1} - q_{v_2} \| - \|q_{u_1} - q_{u_2} \| \Big| + \Big| \|p'_{u_1} - p'_{u_2} \| - \|p'_{v_1} - p'_{v_2} \| \Big|\\
		&<& 4 \varepsilon < t.
	\end{eqnarray*}
	By our choice of $t$ it follows that for all $q \in f_{G'}^{-1}(f_{G'}(p'))$ we have $\|q_{v_1}-q_{v_2}\| = \|p'_{v_1}-p'_{v_2}\|$,
	and thus $(G',p')$ is globally rigid if and only if $(G'+v_1 v_2,p')$ is globally rigid.
	
	Fix a placement $q \in X^V$ where $(G'+v_1v_2,q) \sim (G'+v_1v_2,p')$.
	As $(G,p)$ is a globally rigid subframework of $(G'+v_1v_2,p')$,
	we may suppose we chose $q$ such that $q_v=p'_v$ for all $v \in V$.
	By the uniqueness of the $z$-reflection $R_z$,
	we have $q_{u_i} \in \{ p'_{u_i}, R_z(p'_{u_i}) \}$ for each $i \in \{1,2\}$.
	However, by our choice of $x_1,x_2$, we must have $q_{u_i} = p'_{u_i}$ for each $i \in \{1,2\}$.
	It follows that $(G'+ v_1 v_2,p')$ is globally rigid,
	and thus $(G',p')$ is also globally rigid.
	
	We conclude by proving $(G',p')$ is infinitesimally rigid;
	it will then follow from Theorem \ref{thm:ave} that $G'$ is globally rigid.
	Choose an infinitesimal flex $a \in X^{V'}$ of $(G',p')$.
	Since $(G,p)$ is redundantly rigid,
	$a$ restricted to the vertices of $G$ must be a trivial infinitesimal flex.
	As every non-Euclidean normed plane has finitely many linear isometries (see \cite[Proposition 2.5(ii)]{dew2}),
	we may suppose that $a_v = 0$ for all $v \in V$.
	By our choice of $x_1$ and $x_2$,
	the vectors $p'_{u_i}- p'_{v_1}$ and $p'_{u_i}- p'_{v_2}$ are linearly independent for each $i \in \{1,2 \}$.
	As $X$ is smooth,
	the vectors $p'_{u_i}- p'_{v_1}$ and $p'_{u_i}- p'_{v_2}$ have unique support functionals,
	and as $X$ is strictly convex,
	the support functionals $\varphi_{p'_{u_i}- p'_{v_1}}$ and $\varphi_{p'_{u_i}- p'_{v_2}}$ are linearly independent.
	For each $i,j \in \{1,2\}$ we have
	\begin{align*}
	    \varphi_{p'_{u_i}- p'_{v_j}}(a_{u_i} - a_{v_j}) = \varphi_{p'_{u_i}- p'_{v_j}}(a_{u_i}) = 0.
	\end{align*}
	Hence $a_{u_1} =a_{u_2} = 0$,
	i.e., $a$ is trivial and $(G',p')$ is infinitesimally rigid.
\end{proof}

\subsection{Main result}

Similarly to how $K_4^-$-extensions preserve global rigidity in analytic normed planes,
generalised vertex split also preserve global rigidity when certain criteria are met.

\begin{thm}[{\cite[Theorem 5.4]{DHN}}]\label{thm:vsplitglobal}
    Let $G$ be a globally rigid graph in an analytic normed plane $X$.
    Let $G'$ be a generalised vertex split of $G$ at the vertex $z$ with new vertices $u,v$ and suppose that $G'-uv$ is rigid in $X$.
    Then $G'$ is globally rigid in $X$.
\end{thm}

With this, we are now finally ready to prove our main result (Theorem \ref{thm:grne}). Recall that this states, for a graph $G=(V,E)$ with $|V|\geq 5$ and an analytic normed plane $X$, $G$ is globally rigid in $X$ if and only if $G$ is 2-connected and redundantly rigid in $X$.

\begin{proof}[Proof of Theorem \ref{thm:grne}]
	If $G$ is globally rigid in $X$ then it is 2-connected and redundantly rigid in $X$ by Theorem \ref{thm:hendrickson}.
	Conversely, suppose $G$ is 2-connected and redundantly rigid in $X$.
	By Theorem \ref{thm:1},
	$G$ can be generated from $K_5^-$ or $B_1$ by generalised vertex splits that preserve redundant rigidity in $X$, $K_4^-$-extensions and edge additions. 
	The graphs $K_5^-$ and $B_1$ are globally rigid in $X$ by Lemma \ref{lem:basegraphs}.
	By Theorem \ref{thm:k4-},
	$K_4^-$-extensions preserve global rigidity in $X$,
	and by Theorem \ref{thm:vsplitglobal} generalised vertex splits that preserve redundant rigidity in $X$ also preserve global rigidity in $X$.
	Since edge additions certainly preserve global rigidity, the proof is complete.
\end{proof}

We remark that the characterisation is effective since there exist efficient algorithms that can check redundant rigidity in $X$ \cite{B&J,L&S} and 2-connectivity \cite{Tar}.
There is also an immediate link between global rigidity in analytic normed planes and global rigidity in the Euclidean plane.

\begin{cor}\label{cor:euclidean}
    Let $G=(V,E)$ be a globally rigid graph in $\mathbb{E}^2$ with $|V| \geq 2$ and let $X$ be an analytic normed plane.
    Then $G$ is globally rigid in $X$ if and only if $|E| > 2|V|-2$.
\end{cor}

\begin{proof}
If $G$ is globally rigid in $X$ then $|E| > 2|V|-2$ by Theorems \ref{thm:grne} and \ref{thm:rigidne}. 
For the converse suppose $G$ is globally rigid in $\mathbb{E}^2$ with $|E| > 2|V|-2$. 
Since $G$ is simple, $|V|\geq 5$, and so $G$ is 2-connected by Theorem \ref{thm:euclidglobal}. 
Choose any edge $e$ of $G$.
By Theorems \ref{thm:euclidglobal} and \ref{thm:rigide},
there exists a 
spanning $(2,3)$-tight subgraph $H \subset G-e$.
As $|E| > 2|V|-2$,
there exists an edge $f$ contained in $G-e$ but not $H$.
The graph $H+f$ is $(2,2)$-tight, hence $G-e$ is rigid in $X$ by Theorem \ref{thm:rigidne}. Since $e$ was arbitrary, the result follows from Theorem \ref{thm:grne}. 
\end{proof}

\subsection{Sufficient conditions for global rigidity}
Motivated by analogous results in the Euclidean setting, e.g., 6-connectivity implies global rigidity \cite[Theorem 7.2]{J&J}, we give some sufficient conditions for global rigidity in analytic normed planes.
We begin by recalling three well-known results from graph theory.
We denote the vertex and edge connectivity of a graph $G$ by $\kappa(G)$ and $\lambda(G)$ respectively.

\begin{thm}[\cite{catlin}]\label{t:cls}
    A graph $G=(V,E)$ is $2k$-edge-connected if and only if for all $F \subseteq E$ with $|F|\leq k$,
    the graph $G-F$ contains $k$ edge-disjoint spanning trees.
\end{thm}

\begin{thm}[\cite{hv08}]\label{thm:hv08}
    Let $G=(V,E)$ be a graph with minimum degree $\delta$ and maximal degree $\Delta$,
    and let $G^c= (V,E^c)$ be the complement of $G$.
    \begin{enumerate}[(i)]
        \item \label{thm:hv08.1} Either $\lambda(G) = k$ or $\lambda(G^c) = |V| - \Delta - 1$.
        \item \label{thm:hv08.2} $\kappa(G) + \kappa(G^c) \geq \min \{\delta , |V| - \Delta - 1 \} + 1$.
    \end{enumerate}
\end{thm}

\begin{thm}[{\cite[Theorem 1]{chartrand}}]\label{thm:chartrand-edge-conn}
    Let $G = (V,E)$ be a graph with minimum degree $\delta$. 
    If $\delta \geq \left\lfloor \frac{|V|}{2}\right\rfloor$ then $G$ is $\delta$-edge-connected.
\end{thm}

We first obtain a sufficient connectivity condition for global rigidity.

\begin{cor}\label{c:connected}
  Any 2-connected and 4-edge-connected graph is globally rigid in any analytic normed plane.
\end{cor}

\begin{proof}
    The theorem follows from Theorems \ref{thm:grne}, \ref{thm:rigidne} and \ref{t:cls}.
\end{proof}

It is easy to see that these connectivity conditions are best possible since 2-connectivity is a necessary condition for global rigidity (Theorem \ref{thm:connectivity}) and there exist graphs that are both 2-connected and 3-edge-connected but do not contain a $(2,2)$-tight spanning subgraph (e.g., the complete bipartite graph $K_{3,3}$).
We go out on a limb and conjecture that every $2d$-edge-connected and 2-connected graph is globally rigid in any $d$-dimensional normed space with a finite number of linear isometries. (Recall that when $d>2$ and the norm is non-Euclidean, the weaker property of rigidity is only understood in certain special cases \cite{DKN}.)

We next use Corollary \ref{c:connected} to obtain sufficient minimum degree conditions for global rigidity.

\begin{cor}
    Let $G=(V,E)$ be a graph with minimum degree $\delta \geq 4$ and maximal degree $\Delta \leq |V| - 5$.
    Then either $G$ or its complement $G^c=(V,E^c)$ are globally rigid in any analytic normed plane.
\end{cor}

\begin{proof}
    By Theorem \ref{thm:hv08}(\ref{thm:hv08.1}),
    at least one of $G$ and $G^c$ is 4-edge-connected.
    Without loss of generality we may assume $G$ is 4-edge-connected.
    By Theorem \ref{thm:hv08}(\ref{thm:hv08.2}), $\kappa(G) + \kappa(G^c) \geq 5$.
    If $\kappa(G) \geq 2$ then $G$ is globally rigid in any analytic normed plane by Corollary \ref{c:connected}.
    If $\kappa(G) \leq 1$ then $4 \leq \kappa(G^c) \leq \lambda (G^c)$,
    and hence $G^c$ is is globally rigid in any analytic normed plane by Corollary \ref{c:connected}.
\end{proof}

\begin{cor}\label{cor:ham}
	Let $G = (V,E)$ be a graph with minimum degree $\delta$ and let $X$ be an analytic normed plane.
	If $\delta \geq \max \{4, |V|/2\}$,
	then $G$ is globally rigid in $X$.
\end{cor}

\begin{proof}
	By Dirac's theorem,
	$G$ contains a Hamilton cycle and is hence 2-connected.
	By Theorem \ref{thm:chartrand-edge-conn},
	$G$ is 4-edge-connected.
	Hence $G$ is globally rigid in $X$ by Corollary \ref{c:connected}.
\end{proof}

Corollary \ref{cor:ham} is best possible.
We can see this by constructing counterexamples with $n$ vertices and minimum degree $\delta$ for each of the three cases:
$\delta < \lfloor n/2 \rfloor$, $n$ is odd and $\delta = \lfloor n/2 \rfloor$, and $\delta \leq 3$.
For $\delta < \lfloor n/2 \rfloor$ we may take the disjoint union of any two complete graphs $K_s,K_t$ where $s+t = n$ and $s,t \geq \lfloor n/2 \rfloor$.
If $n$ is odd and $\delta = \lfloor n/2 \rfloor$,
then we can construct a connected but not 2-connected graph by gluing two copies of $K_{\lceil n/2 \rceil}$ at a single vertex.
Finally,
for $\delta \leq 3$ we note that every cycle and every cubic graph graph is $(2,2)$-sparse and so cannot be redundantly rigid in any non-Euclidean normed plane.

Next we give spectral sufficiency conditions for global rigidity in analytic normed planes.
Define the \emph{algebraic connectivity} of a graph to be the second smallest eigenvalue of the Laplacian matrix of the graph.

\begin{cor}
    Let $G$ be a 2-connected graph with minimum degree $\delta \geq 5$ and algebraic connectivity $\mu > \frac{4}{\delta +1}$.
    Then $G$ is globally rigid in any analytic normed plane.
\end{cor}

\begin{proof}
    By \cite[Theorems 4.10 \& 4.12(ii)]{cdgg22} and Theorem \ref{thm:rigidne},
    $G$ is redundantly rigid in any analytic normed plane.
    The result now follows from Theorem \ref{thm:grne}.
\end{proof}

An alternative sufficient condition for global rigidity in $d$-dimensional Euclidean space was provided by Tanigawa \cite{Tan}. Our next result, which follows from Theorem \ref{thm:grne}, shows that an analogous statement is true for analytic normed planes.

\begin{cor}\label{cor:vertexred}
	Let $G = (V,E)$ be a graph, $|V|\geq 3$ and $X$ be an analytic normed plane. 
	If $G -v$ is rigid in $X$ for all $v \in V$ then $G$ is globally rigid in $X$.
\end{cor}

\begin{proof}
Since rigid graphs are connected and $|V|\geq 3$, $G$ is $2$-connected. Moreover by considering any edge incident to a given vertex, $G$ must be redundantly rigid in $X$. Hence $G$ is globally rigid in $X$ by Theorem \ref{thm:grne}.
\end{proof}

Again one might expect the corollary to remain valid in arbitrary dimensions, under the hypothesis that the normed space has finitely many linear isometries.

Our combinatorial description of global rigidity can also easily be applied to \emph{vertex-transitive} (the automorphism group acts transitively on the vertex set) or \emph{edge-transitive} (the automorphism group acts transitively on the edge set) graphs with sufficient minimum degree.

\begin{cor}\label{cor:trans}
	Let $G = (V,E)$ be a connected graph with minimum degree $\delta \geq 4$. 
	If $G$ is vertex-transitive or edge-transitive,
	then $G$ is globally rigid in any analytic normed plane.
\end{cor}

\begin{proof}
Since $\delta \geq 4$ and $G$ is connected, if $G$ is vertex-transitive it is 4-connected by \cite[Theorem 3.4.2]{GodsilRoyle},
    and if $G$ is edge-transitive it is 4-connected by \cite[Corollary 1A]{Watkins}.
    Hence $G$ is globally rigid in any analytic normed plane by Corollary \ref{c:connected}.
\end{proof}

We conclude the paper by applying Corollary \ref{c:connected} to two natural random graph models.
An \emph{Erd\"{o}s-R\'{e}nyi random graph} $G_{n,p}$ is a graph with $n$ vertices where each edge exists with probability $p$.

\begin{cor}
    Let $G_{n,p}$ be an Erd\"{o}s-R\'{e}nyi random graph with
    \begin{align*}
        p = \frac{\log (n) + k\log \log (n) + c_n}{n}
    \end{align*}
    for some positive integer $k$ and some real sequence $(c_n)_{n \in \mathbb{N}}$.
    Then the following holds for any analytic normed plane $X$.
    \begin{enumerate}[(i)]
        \item\label{c:errandom1} If $k = 1$ then:
        \begin{enumerate}[(a)]
            \item $G_{n,p}$ is rigid in $X$ a.a.s. if $c_n \rightarrow \infty$.
            \item $G_{n,p}$ is flexible in $X$ a.a.s. if $c_n \rightarrow -\infty$.
        \end{enumerate}
        \item\label{c:errandom2} If $k = 2$ then:
        \begin{enumerate}[(a)]
            \item $G_{n,p}$ is globally rigid in $X$ a.a.s. if $c_n \rightarrow \infty$.
            \item $G_{n,p}$ is not globally rigid in $X$ a.a.s. if $c_n \rightarrow -\infty$. 
        \end{enumerate}
    \end{enumerate}
\end{cor}

\begin{proof}
    We first consider the case where $c_n \rightarrow \infty$.
    By \cite[Corollary 1.2]{LNPR},
    $G_{n,p}$ is rigid in $\mathbb{E}^{k+1}$ a.a.s.
    First suppose $k=1$.
    As $G_{n,p}$ is rigid in $\mathbb{E}^{2}$ a.a.s.,
    Theorem \ref{thm:rigide} implies $G_{n,p}$ contains a spanning $(2,3)$-tight subgraph $H$ a.a.s.
    Since the expected number of edges of $G_{n,p}$ is $\frac{1}{2}(n-1)(\log (n) + \log \log (n) + c_n)$ which is strictly larger than $2n-3$ for sufficiently large $n$,
    there exists an edge $e \in E(G_{n,p}) \sm E(H)$ a.a.s.
    As $H+e$ (when both $H$ and $e$ exist) is a $(2,2)$-tight graph,
    $G_{n,p}$ is rigid in $X$ a.a.s. by Theorem \ref{thm:rigidne}.
    Now suppose $k=2$.
    By \cite[Theorem 1.18]{GHT},
    $G_{n,p}$ is globally rigid in $\mathbb{E}^2$ a.a.s.
    Since the expected number of edges of $G_{n,p}$ is $\frac{1}{2}(n-1)(\log (n) + \log \log (n) + c_n)$ which is strictly larger than $2n-2$ for sufficiently large $n$,
    $G_{n,p}$ is globally rigid in $X$ a.a.s. by Corollary \ref{cor:euclidean}.
    
    We now consider the case where $c_n \rightarrow -\infty$.
    Here we have that $G_{n,p}$ has minimum degree at most $k$ a.a.s.;
    see \cite[Section 7]{boll} for more details.
    If $k=1$ then $G_{n,p}$ is flexible in $X$ a.a.s.
    If $k=2$ then $G_{n,p}$ is not redundantly rigid in $X$ a.a.s., and hence is not globally rigid in $X$ a.a.s., by Theorem \ref{thm:grne}.
\end{proof}

Let $\mathcal{G}_{n,k}$ be the set of all $k$-regular graphs with $n$ vertices.
By equipping the set $\mathcal{G}_{n,k}$ with the uniform probability measure,
we define a graph chosen at random from $\mathcal{G}_{n,k}$ to be a \emph{random $k$-regular graph}.

\begin{cor}
    Let $G$ be a random $k$-regular graph and let $X$ be an analytic normed plane.
    If $k \geq 4$ then $G$ is globally rigid in $X$ a.a.s.
\end{cor}

\begin{proof}
    A random $k$-regular graph is $k$-connected a.a.s. \cite{worm},
    hence $G$ is globally rigid in $X$ a.a.s. by Corollary \ref{c:connected}.
\end{proof}

\bibliographystyle{abbrv}
\def\lfhook#1{\setbox0=\hbox{#1}{\ooalign{\hidewidth
  \lower1.5ex\hbox{'}\hidewidth\crcr\unhbox0}}}

\end{document}